\documentclass[11pt,a4paper,reqno]{amsart}
\usepackage[T1]{fontenc}
\usepackage{amsmath}
\usepackage{amsthm}
\usepackage{amsfonts}
\usepackage{amssymb}
\usepackage{graphicx}
\usepackage{amsbsy}
\usepackage{mathrsfs}
\usepackage{bbm}
\newtheorem{theorem}{Theorem}[section]
\newtheorem{lemma}[theorem]{Lemma}

\newtheorem{proposition}[theorem]{Proposition}

\newtheorem{definition}[theorem]{Definition}
\theoremstyle{remark}
\newtheorem{remark}[theorem]{\it \bf{Remark}\/}

\numberwithin{equation}{section}
\catcode`@=11
\def\section{\@startsection{section}{1}%
  \z@{1.5\linespacing\@plus\linespacing}{.5\linespacing}%
  {\normalfont\bfseries\large\centering}}
\catcode`@=12
\newcommand{\be}{\begin{equation}}
\newcommand{\ee}{\end{equation}}
\newcommand{\bea}{\begin{eqnarray}}
\newcommand{\eea}{\end{eqnarray}}
\newcommand{\bee}{\begin{eqnarray*}}
\newcommand{\eee}{\end{eqnarray*}}

\def\pa{\partial}

\def\RR{\mathbb{R}}

\def\fref#1{{\rm (\ref{#1})}}

\def\tn{\tilde{n}}

\def\tm{\tilde{m}}

\catcode`@=11
\def\supess{\mathop{\operator@font Sup\,ess}}
\catcode`@=12

\def\bt{\tilde{b}}

\def\RR{\mathbb{R}}

\def\Mod{\rm Mod}

\def\e{\varepsilon}

\def\fref#1{{\rm (\ref{#1})}}

\def\R2+{\RR ^2_+}

\def\hb{\hat{b}}

\def\bh{\hat{b}}

\def\lsl{\frac{\lambda_s}{\lambda}}
\def\tt{\tilde{T}}

\def\pa{\partial}

\def\lim{\mathop{\rm lim}}

\def\R{\Bbb R}
\def\e{\varepsilon}
\def\l{\lambda}

\def\log{{\rm log}}

\def\et{\tilde{\e}}

\def\lsl{\frac{\lambda_s}{\lambda}}

\def\Psit{\tilde{\Psi}}

\def\ut{\tilde{u}}

\def\qbt{\tilde{Q}_{b}}

\def\tPsi{\tilde{\Psi}}
\def\Phit{\tilde{\Phi}}
\def\tPhi{\tilde{\Phi}}

\def\pa{\partial}

\def\wh{\hat{w}}

\def\zb{\zeta_{\rm big}}
\def\zs{\z_{\rm sm}}

\def\et{\tilde{\e}}

\def\eh{\hat{\e}}
\def\matchal{\mathcal}

\def\pa{\partial}
\def\psit{\tilde{\psi}}
\def\z{\zeta}
\def\hw{\hat{w}}

\def\hF{\hat{F}}
\title[]{On the stability of critical chemotactic aggregation}
\author[P. Rapha\"el]{Pierre Rapha\"el}
\address{Laboratoire J.A Dieudonn\'e, Universit\'e de Nice-Sophia Antipolis, and Institut Universitaire de France, }
\email{praphael@math.unice.fr}
\author[R. Schweyer]{R\'emi Schweyer}
\address{Institut de Math\'ematiques de Toulouse, Universit\'e Toulouse III, France}
\email{remi.schweyer@math.univ-toulouse.fr}

\begin{document}
\maketitle

\begin{abstract}

We consider the two dimensional parabolic-elliptic Patlak-Keller-Segel model of chemotactic aggregation for radially symmetric initial data. We show the existence of a stable mechanism of singularity formation and obtain a complete description of the associated aggregation process.

\end{abstract}


\section{Introduction}



\subsection{Setting of the problem}

 
We consider the two dimensional Patlak-Keller-Segel problem:
\be
\label{kps}
(PKS)\ \ \left\{\begin{array}{lll}
\pa_tu=\nabla\cdot(\nabla u+u\nabla \phi_u),\\
\phi_u=\frac{1}{2\pi}\log |x|\star u\\
u_{|t=0}=u_0> 0
\end{array} \right .\ \ (t,x)\in \Bbb R\times \Bbb R^2,\
\ee
This degenerate nonlocal diffusion equation is a canonical limit of kinetic models of particles evolving through a nonlocal attractive force, and is one of the model which arises in the description of colonies of bacteria, \cite{KS}, \cite{Patlak}. This system has attracted a considerable attention for the past twenty years in the mathematical community, in particular in connection with the local regularity of weak solutions, the qualitative behavior of solutions and the possibility of finite time blow up for large enough data corresponding to the aggregation of bacteria. We refer to \cite{BCC} for an extensive introduction to the literature on these subjects. Our point of view in this paper is to replace the problem within the sets of {\it critical blow up problems} which for both parabolic and dispersive problems have also attracted a considerable attention for the past ten years.\\ 
The existence of unique local smooth solutions in some suitable Sobolev sense can be obtained from standard fixed point arguments as in \cite{Nagaired}. The corresponding non negative strong solution satisfies the conservation of mass 
\be
\label{consmasse}
\int_{\Bbb R^2}u(t,x)dx=\int_{\Bbb R^2}u(0,x)dx
\ee
and the flow dissipates the logarithmically degenerate entropy: 
\be
\label{deacyetnropy}
\matchal F(u)=\int u\log u+\frac12\int u\phi_u\leq \matchal F(u_0).
\ee
The scaling symmetry $$u_\l(t,x)=\l^2u(\l^2 t,\l x)$$ leaves the $L^1$ norm unchanged $$\int_{\Bbb R^2}u_\l(t,x)dx=\int_{\Bbb R^2}u(\l^2 t,x)dx$$ and hence the problem is $L^1$ critical. Note from \fref{deacyetnropy} that the problem is also {\it almost energy critical} and therefore particularly degenerate.


\subsection{Previous results}


A fundamental role is played for the analysis by the explicit {\it ground state} stationary solution 
 \be
 \label{defQ}
 Q(x)=\frac{8}{(1+|x|^2)^2}.
 \ee
 From \cite{Beckner}, \cite{CL},  $Q$ is up to symmetry the unique minimizer of the logarithmic Hardy-Littlewood-Sobolev inequality:
 $\forall u\geq 0$ with $\int u=M$, 
\be
\label{logsobolev}
\int u\log u+\frac{4\pi}{M}\int \phi_u u\geq M\left[1+\log\pi-\log M\right].
\ee

 For smooth well localized data with small mass $\int u_0<\int Q$, the flow is global and zero is the universal local attractor, see \cite{DNR}, \cite{BDP}, \cite{PV}, and in this sense $Q$ is the first nonlinear object. For data above the ground state $\int u_0>\int Q$, the virial type identity $$\frac{d}{dt}\int |x|^2u(t,x)dx=4\left(1-\frac{\int u}{\int Q}\right)\int u$$ implies that all smooth well localized data {\it blow up} in finite time. Note however that as usual, this argument is very unstable by perturbation of the equation, and provides  almost no insight into the structure of the singularity formation.\\
 Substantial progress have been made in the critical case $\int u=\int Q$. In \cite{BCM}, finite variance initial data $\int |x|^2u_0<+\infty$ with minimal mass are shown to grow up in {\it infinite time}. The argument is again by contradiction and does not give the associated blow up rate. More dynamical results for domains are obtained in \cite{KavSou}. In \cite{BCC}, it is shown that infinite variance initial data $\int |x|^2u_0=+\infty$ exhibit a completely different behavior and generate a global flow asymptotically attracted by the soliton \fref{defQ}, see also \cite{CF} for quantitative convergence rates. The proof involves the use of an additional Lyapounov functional at the minimal mass level, and the importation of tools from optimal transportation thanks to the gradient flow structure of the problem.\\
All the above subcritical mass results heavily rely on the fact that for $\int u_0\leq \int Q$, the dissipated entropy \fref{deacyetnropy} coupled with the variational characterization of $Q$ through the logarithmic HLS \fref{logsobolev} imply a priori uniform bounds on the solution. This structure is completely lost for $\int u_0>\int Q$, and the description of the flow for mass super critical data is very poorly understood. In the pioneering work \cite{HV}, Herrero and Velasquez use formal matching asymptotics, the radial reduction of the problem and ODE techniques to produce the first example of blow up solution.


\subsection{Connection with critical problems}


The $L^1$ critical structure of the (PKS) system is canonical from the point of view of critical problems, and a number of examples extracted both from the parabolic and dispersive worlds have recently attracted a considerable attention and led to a new approach for the construction of blow up solutions. Examples of such critical flows are the mass critical Non Linear Schr\"odinger equations of nonlinear optics 
\be
\label{nls}
(NLS) \ \ \left\{\begin{array}{ll}i\pa_tu+\Delta u+u|u|^2=0,\\u_{|t=0}=u_0,\end{array}\right., \ \ (t,x)\in \Bbb R\times \Bbb R^2,\ \ u\in \Bbb C,
\ee
 the geometric energy critical parabolic heat flow to the 2-sphere of crystal physics $$(HF) \ \ \left\{\begin{array}{ll}\pa_tu=\Delta u+|\nabla u|^2u,\\u_{|t=0}=u_0,\end{array}\right., \ \ (t,x)\in \Bbb R\times \Bbb R^2, \ \ u\in \Bbb S^2$$ and its dispersive Schr\"odinger map analogue in ferromagnetism: $$(SM) \ \ \left\{\begin{array}{ll}u\wedge \pa_tu=\Delta u+|\nabla u|^2u,\\u_{|t=0}=u_0,\end{array}\right., \ \ (t,x)\in \Bbb R\times \Bbb R^2, \ \ u\in \Bbb S^2.$$ For all these problems, a robust approach for the construction of blow up solutions and the study of their possible stability has been developed in the past ten years, see \cite{MR1}, \cite{MR2}, \cite{MR3}, \cite{MR4}, \cite{MR5} for the mass critical NLS, \cite{RaphRod}, \cite{MRR} for the wave and Schr\"odinger maps, and \cite{RS} for the harmonic heat flow. The strategy proceeds in two steps: the construction of suitable approximate solutions through the derivation of the leading order ODE's driving the trajectory of the solution on the modulated manifold of ground states; the control of the exact flow near these approximate profiles through the derivation of suitable Lyapounov functionals and a robust energy method. The first step avoids the common use of matching asymptotics which is most of the time delicate both to implement and to make rigorous, see for example \cite{LPSS} for (NLS), \cite{BHG} for the heat flow, \cite{LS} for (PKS). The second step is a pure energy method which therefore applies both to dispersive problems and parabolic systems, and makes in this last case no use of the maximum principle commonly used for scalar parabolic problems. The sharp knowledge of the spectral structure of the linearized operator, which is typically a delicate problem for (PKS), see \cite{LS}, is replaced by canonical spectral gap estimates as initiated in \cite{RodSter}, \cite{RaphRod}.


\subsection{Statement of the result}


Our main claim in this paper is that despite the nonlocal structure of the problem and in particular of the linearized operator close to $Q$, and the almost energy critical degeneracy of the problem\footnote{which is reflected by the weakness of the a priori information \fref{deacyetnropy} for mass super critical initial data.},
the above route map can be implemented for (PKS). We address the radial case only for the sake of simplicity, but the full non radial problem can be analyzed in principle along similar lines.\\ 
In order to make our statement precise, let us consider the following function spaces. Let the weighted $L^2$ space 
\be
\label{l2rnoms}
\|\e\|_{L^2_Q}=\left(\int\frac{\e^2}{Q}\right)^{\frac 12}
\ee and the weighted $H^2$ space:
\be
\label{defhtwow}
\|\e\|_{H^2_Q}=\|\Delta \e\|_{L^2_Q}+\left\|\frac{\nabla \e}{1+|x|}\right\|_{L^2_Q}+\|\e\|_{L^2}.
\ee
We introduce the energy norm
\be
\label{energyspcaee}
\|\e\|_{\matchal E}=\|\e\|_{H^2_Q}+\|\e\|_{L^1}.
\ee
The main result of this paper is the complete description of a {\it stable} chemotactic blow up with radial data arbitrarily close to the ground state in the $L^1$ critical topology.

\begin{theorem}[Stable chemotactic blow up]
\label{thmmain}
There exists a set of initial data of the form $$u_0=Q+\e_0\in\mathcal E, \ \ u_0> 0, \ \ \|\e_0\|_{\mathcal E}\ll 1$$ such that the corresponding solution $u\in \mathcal C ([0,T),\mathcal E)$ to \fref{kps} satisfies the following:\\
{\em (i) Small super critical mass}: $$ 8\pi<\int u_0<8\pi+ \alpha^*$$ for some $0<\alpha^*\ll1$ which can be chosen arbitrarily small;\\
{\em (ii) Blow up} : the solution blows up in finite time $0<T<+\infty$;\\
{\em (iii) Universality of the blow up bubble}: the solution admits for all times $t\in [0,T)$ a decomposition $$u(t,x)=\frac{1}{\l^2(t)}(Q+\e)\left(t,\frac{x}{\l(t)}\right)$$ with 
\be
\label{boievbebeo}
\|\e(t)\|_{H^2_Q}\to 0\ \ \mbox{as}\ \ t\to T
\ee and the universal blow up speed: 
\be
\label{blkpeoghenepojg}
\lambda(t)=\sqrt{T-t}e^{-\sqrt{\frac{|\log (T-t)|}{2}}+O(1)}\ \ \mbox{as} \ \ t\to T.
\ee
{\em (iv) Stability}: the above blow up dynamics is stable by small perturbation of the data in $\mathcal E$: $$v_0> 0, \ \ \|v_0-u_0\|_{\mathcal E}<\epsilon(u_0).$$
\end{theorem}

{\it Comments on the result}\\

{\it 1. On the stability statement}: Formal arguments to prove the stability of chemotactic blow up are presented in \cite{Vlinear}, and in fact Theorem \ref{thmmain} answers one of the problems mentioned in \cite{VICM}. It is likely that there are many other blow regimes, possibly stable or unstable depending on the tail of the initial data, as is expected for the heat flow \cite{BHG}, see also \cite{NT2}, or known for the super critical heat equation \cite{Mizo}, see also \cite{BT} for a further illustration of the  role of the topology for the long time dynamics. Note that a complete description of the flow near the ground state is obtained for the first time for a critical problem in \cite{MMR1}, \cite{MMR2}, \cite{MMR3} for the critical (KdV) problem, and this includes pathological blow up rates depending on the structure of the data. In this sense,  Theorem \ref{thmmain} is a first step toward the description of the flow near $Q$ which even in the radial setting is still a challenging problem.\\

{\it 2. On the non locality}: It is known that in the radial setting, the nonlocal (PKS) can be turned to a nonlinear {\it local} problem for the partial mass, see for example \cite{HV}. We shall not use this structure at all, and therefore our approach can in principle be extended to the non radial setting. The non locality of the chemotactic problem induces a nonlinear linearized operator $\matchal L$ close to $Q$, which due to the slow decay of the soliton makes the use of spectral techniques delicate, see \cite{LS} for an attempt in this direction. Moreover, and this is for example a {\it major} difference with the mass critical (NLS) \fref{nls}, the entropy control \fref{deacyetnropy} is very weak and yields almost no information on the flow near $Q$. However, following the approach in \cite{RaphRod}, simple spectral gap estimates can be obtained on $\matchal L$ {\it and its iterates} as a consequence of the variational characterization of $Q$ which lead to the control of derivatives of $u$ in suitable weighted $L^2$ norms. In other words, for the linearized flow close to $Q$, we control the solution through the control of the natural linearized entropy involving the derivatives of $u$, and the associated dissipated quantities. The control of derivatives yields through Hardy inequalities the local control of the solution on the soliton core which is the heart of the control of the speed of concentration. We also obtain sharp convergence rates of the solution towards the solitary wave after renormalization.\\

Our proof provides a robust approach to the construction of blow up regimes in the presence of a nonlocal nonlinearity, and is a first step towards the understanding of the flow near the soliton. It also provides a new insight to attack more complicated chemotactic coupling, in particular the parabolic-parabolic Patlak-Keller-Segel problem for which blow up in the critical setting is open.\\

{\bf Aknowledgements.} The authors would like to thank A. Blanchet, J. Dolbeaut and P. Laurencot for their interest and support during the preparation of this work. Part of this work was done while P.R was visiting the MIT Mathematics department which he would like to thank for its kind hospitality. This work is supported by the ERC/ANR grant SWAP and the advanced ERC grant BLOWDISOL\\

{\bf Notations}. For a given function $u$, we note its Poisson field $$\phi_u=\frac{1}{2\pi}\log |x|\star u.$$ We note the real $L^2$ scalar product $$(f,g)=\int_{\Bbb R^2}f(x)g(x)dx.$$ Given a vector $y\in \Bbb R^2$, we will systematically note $r=|y|=\sqrt{y_1^2+y_2^2}$. We will generically note $$\delta(z^*)=o(1)\ \ \mbox{as}\ \ z^*\to 0.$$ We let $\chi\in \mathcal C^{\infty}_c(\R^2)$ be a radially symmetric cut off function with $$\chi(x)=\left\{\begin{array}{ll} 1\ \ \mbox{for}\ \ |x|\leq 1,\\ 0\ \ \mbox{for}\ \ |x|\geq 2\end{array}\right., \ \ \chi(x)\geq 0.$$ We introduce the differential operator $$\Lambda f=2f+y\cdot\nabla f=\nabla\cdot(yf).$$ Given $f\in L^2$, we introduce its decomposition into radial and non radial part:
\be
\label{radialnonradial}
f^{(0}(r)=\frac{1}{2\pi}\int_0^{2\pi}f(r,\theta)d\theta, \ \ f^{(1)}(x)=f(x)-f^{(0)}(r).
\ee
Observe that $$(\Lambda f)^{(0)}=(r\pa_rf)^{(0)}=r\pa_rf^{(0)}=\Lambda f^{(0)}.$$ Given $b>0$, we let 
\be
\label{defb}
B_0=\frac{1}{\sqrt{b}}, \ \ B_1=\frac{|\log b|}{\sqrt{b}}.
\ee
Give two linear operators $(A,B)$, we note their commutator: $$[A,B]=AB-BA.$$


\subsection{Strategy of the proof}


We briefly explain the main steps of the proof of Theorem \ref{thmmain} which follows the approach developed in \cite{RaphRod}, \cite{MRR}, \cite{RS}.\\

{\it Step 1: Construction of approximate blow up profiles}. Let us renormalize the (PKS) flow by considering $$u(t,x)=\frac1{\l^2(t)}v(s,y), \ \ \frac{ds}{dt}=\frac{1}{\l^2}, \ \ y=\frac{x}{\l}$$ which leads to the renormalized flow $$\pa_s v+b\Lambda v=\nabla\cdot(\nabla v+v\nabla \phi_v), \  \ b=-\frac{\l_s}{\l}.$$ We look for a slowly modulated approximate solution of the form $$v(s,y)=Q_{b(s)}(y), \ \ Q_b(y)=Q(y)+bT_1(y)+b^2T_2(y)$$ and where the law $b_s=F(b)=O(b^2)$ is an unknown of the problem. The construction of $T_1$, $T_2$ amounts solving stationary in $s$ elliptic problems of the form 
\be
\label{cneioeihivoe}
\mathcal LT_j=F_j(T_{j-1})
\ee
 where we introduced the linearized operator close the soliton:
 \be
\label{formuleL}
\mathcal L\e=\nabla \cdot(Q\nabla \mathcal M\e)=\Delta \e+2 Q\e+\nabla \phi_Q\cdot\nabla \e+\nabla Q\cdot\nabla \phi_\e.
\ee 
This operator displays explicit resonances\footnote{i.e. slowly decaying zeroes.} induced by the symmetry group of (PKS). In the radial setting, the solution to \fref{cneioeihivoe} is almost explicit, and the slow decay of $Q$ at $+\infty$ induces slowly decaying function $T_j$ with explicit tails at $+\infty$. In particular, $$T_1(r)\sim \frac{c}{r^2}\ \  \mbox{as}\ \ r\to+\infty$$ and to leading order in terms of tails at $+\infty$, the $T_2$ equation looks like $$\mathcal LT_2=-b_sT_1-b^2\Lambda T_1+\mbox{lot}.$$ We now observe the cancellation $$\Lambda T_1\sim \frac{1}{r^4}\ \ \mbox{as}\ \ r\to +\infty$$ and hence the choice $b_s=0$ yields the most decreasing right hand side. In fact, a more careful analysis leads to the logarithmically degenerate law: $$b_s=-\frac{2b^2}{|\log b|},$$ see the proof of Proposition \ref{propblowup}. We refer to \cite{RaphRod}, \cite{MRR} for a further discussion on the derivation of these logarithmic gains. The outcome is the derivation of the leading order dynamical system 
\be
\label{odeintro}
\left\{\begin{array}{ll} \frac{ds}{dt}=\frac{1}{\l^2}\\b=-\lsl, \ \ b_s=-\frac{2b^2}{|\log b|}\end{array}\right.
\ee
which after reintegration in time implies that $\l$ touches zero in finite time $T<+\infty$, this is blow up, with the asymptotics \fref{blkpeoghenepojg}.\\

{\it Step 2: Energy bounds}. We know consider an initial data of the form $$u_0=Q_{b_0}+\e_0\ \ \mbox{with}\ \ \|\e_0\|\lesssim b_0^{10}$$ in some suitable topology. We claim that we can bootstrap the existence of a decomposition $$u(t,x)=\frac{1}{\l^2(t)}(Q_{b(t)}+\e)(t,\frac{x}{\l(t)})\ \ \ \ \mbox{with}\ \ \|\e(t)\|\lesssim (b(t))^{10}$$ with $(\l(t),b(t))$ satisfying to leading order the ODE's \fref{odeintro}. Roughly speaking, the equation driving $\e$ is of the form $$\pa_s\e=\mathcal L\e-b\Lambda \e+\Psi_b+\Mod +N(\e)$$ where $\Psi_b=O(b^3)$ is the error in the construction of $Q_b$, $\Mod$ is the forcing term induced by the modulation parameters $$\Mod=(\lsl+b)\Lambda Q_b-b_s\frac{\partial Q_b}{\partial b},$$ and $N(\e)$ denotes the lower order nonlinear term. In order to treat the $b\Lambda $ correction to the linearized term {\it which cannot be treated perturbatively in blow up regimes}, we reformulate the problem in terms of the original variables $$\tilde{u}(t.x)=\frac{1}{\l^2}\e(t,\frac x{\lambda})$$ which satisfies 
\be
\label{nkenenenoe}
\pa_t \tilde{u}=\mathcal L_\l\tilde{u}+\frac{1}{\l^4}\left[\Psi_b+\Mod +N(\e)\right](t,\frac{x}{\l}).
\ee
We now control $\tilde{u}$ using an energy method on its derivatives as in \cite{RaphRod}. Let $\mathcal M$ be the linearized entropy close to $Q$, see \fref{defm}. One can show using the explicit knowledge of the resonances of $\matchal L$, which itself essentially follows from the variational characterization of $Q$, that modulo suitable orthogonality conditions which remove the zero modes and correspond to the dynamical adjustment of $(b(t),\l(t))$, the generalized linearized HLS energy is coercive on suitable derivatives: $$(\mathcal M \mathcal L\e,\mathcal L\e)\gtrsim \|\mathcal L\e\|_{L^2_Q}^2\gtrsim \int \frac{|\Delta \e|^2}{Q}+\int |\e|^2\sim \|\e\|_{H^2_Q}^2.$$ Applying this to $\tilde{u}$, we are led to consider the natural higher order entropy for \fref{nkenenenoe}: $$\frac{d}{dt}\left(\mathcal M_{\l}\mathcal L_{\l}\tilde{u},\mathcal L_{\l}\tilde{u}\right)$$ where $\mathcal M_{\l},\mathcal L_{\l}$ correspond to the linearization close to the singular bubble $\frac{1}{\l^2}Q(\frac x\l)$. Two difficulties occur. First one needs to treat the local terms on the soliton core induced by the time dependance $\l(t)$ of the renormalized operator $\mathcal M_\l,\mathcal L_\l$, and here the dissipated terms together with the sharp control of tails at infinity play a crucial role. An additional problem occurs 
due to the pointwise control of $\Mod$ terms which is not good enough in our regime\footnote{This is a technical unpleasant problem which directly relates to the slow decay of $Q$ at infinity. Similar issues already occurred in related settings, see for example \cite{MMfintietime}.}, and requires some further integration by parts in time, see Lemma \ref{lemmasharpmod}.\\
The energy method leads roughly to a pointwise bound $$\frac{d}{dt}\left(\mathcal M_{\l}\mathcal L_{\l}\tilde{u},\mathcal L_{\l}\tilde{u}\right)\lesssim \frac{1}{\l^6}\frac{b^4}{|\log b|^2}$$ which after rescaling leads {\it in the regime} \fref{odeintro} to the pointwise bound: $$\|\e\|_{H^2_Q}^2\lesssim \frac{b^3}{|\log b|^2}.$$ It turns out that this estimate is sharply good enough to close the ODE for b and show that the $\e$ radiation does not perturb the leading order system \fref{odeintro}, and this concludes the proof of Theorem \ref{thmmain}.\\

This work is organized as follows. In section \ref{spectral}, we prove spectral gap estimates for the nonlocal linearized operator close to $Q$ in the continuation of \cite{RaphRod}. In section \ref{construction}, we build the family of approximate self similar profiles which contains the main qualitative informations on the singularity formation, and adapts in the non local setting the strategy developed in \cite{RaphRod}, \cite{MRR}, \cite{RS}. In section \ref{sectionboot}, we start the analysis of the full solution and explicitly display the set of initial data leading to the conclusions of Theorem \ref{thmmain} and the bootstrap argument to control the radiation. In section \ref{htwowbound}, we exhibit the $H^2_Q$ monotonicity formula which is  the heart of our analysis. In section \ref{proofthmmai}, we collect the estimates of sections \ref{sectionboot}, \ref{htwowbound} to close the bootstrap and complete the description of the singularity formation.


\section{Spectral gap estimates}
\label{spectral}


This section is devoted to the derivation of spectral gap estimates for the linearized operator close to $Q$ given by \fref{formuleL}. The radial assumption is useless here and we address directly the general case. These energy bounds are at the heart of the control of the radiation correction for flows evolving close to the approximate $\qbt$ blow up profiles. We derive two type of estimates: energy estimates which are mostly a consequence of the linearization close to $Q$ of the sharp log Sobolev estimate\footnote{following the celebrated proof by Weinstein \cite{W2} for nonlinear Schr\"odinger equations.} and higher order iterated estimates which rely on the explicit knowledge of the kernel of $\mathcal L$ and suitable repulsivity properties at infinity.


\subsection{Energy bounds}


We start with the study of the linearized log Sobolev energy on the weighted $L^2$ space \fref{l2rnoms}.

\begin{lemma}[Structure of the linearized energy]
\label{lemmakernel}
The operator 
\be
\label{defm}
\mathcal Mu=\frac{u}{Q}+\phi_u
\ee is a continuous self adjoint operator $\mathcal M:L^2_Q\to (L^2_Q)^*$ ie:
\be
\label{mcontiniuos}
\int Q|\mathcal Mu|^2\lesssim \|u\|_{L^2_Q}^2,
\ee
\be
\label{efadjointness}
\forall (u,v)\in L^2_Q\times L^2_Q, \ \ (\mathcal M u,v)=(u,\mathcal Mv).
\ee
Moreover, there holds:\\
{\it (i) Algebraic identities}:
\be
\label{relationsm}
\mathcal M(\Lambda Q)=-2, \ \ \mathcal M(\pa_iQ)=0, \ i=1,2.
\ee
{\it (ii) Generalized kernel}: Let $(u,\phi)\in \mathcal C^{\infty}(\Bbb R^2)$ satisfy 
\be
\label{euiheheoifho}
\left\{\begin{array}{ll}\nabla\left(\frac{u}{Q}+\phi\right)=0\\
\Delta \phi=u
\end{array}\right.,\ \ \ \int\frac{\phi^2}{(1+r^4)(1+|\log r|^2)}<+\infty,
\ee then 
\be
\label{identificationkernel}
u\in \mbox{Span}(\Lambda Q,\pa_1 Q,\pa_2Q), \ \ \phi=\phi_u.
\ee
\end{lemma}

\begin{remark} This structure is reminiscent from kinetic transport models like the gravitational Vlasov Poisson equation, see \cite{LMRlinearized}.
\end{remark}

\begin{proof}[Proof of Lemma \ref{lemmakernel}]
{\bf step 1} Continuity and self adjointness. The continuity of $\mathcal M$ as an operator from $L^2_Q\to (L^2_Q)^*$ follows from \fref{linftybound}:
$$\int Q|\mathcal Mu|^2\lesssim \int\frac{u^2}{Q}+\left\|\frac{\phi_u}{1+|\log r|}\right\|^2_{L^{\infty}}\|(1+|\log r|)Q\|_{L^1}\lesssim \|u\|_{L^2_Q}^2.$$ The self adjointness \fref{efadjointness} is equivalent to: 
\be
\label{cneneoneo}
\forall (u,v)\in L^2_Q\times L^2_Q, \  \ (\phi_u,v)=(u,\phi_v).
\ee
Indeed, the integrals are absolutely convergent arguing like for \fref{linftybound}:$$\int|\log |x-y|||u(y)||v(x)|dxdy\lesssim \|u\|_{L^2_Q}\|v\|_{L^2_Q}$$ and the claim follows from Fubbini.\\

{\bf  step 2}. Proof of (i).  By definition of $Q$:
$$\phi_Q'(r)=\frac{1}{r}\int_0^r\frac{8\tau}{(1+\tau^2)^2}d\tau=\frac{4r}{1+r^2}.$$ Moreover, using the convolution representation $$\phi_Q(0)=\int_0^{+\infty}\frac{8r\log r}{(1+r^2)^2}dr=4\lim_{(\e,R)\to (0,+\infty)}\left[\frac{r^2\log r}{1+r^2}-\frac12\log (1+r^2)\right]_\e^R=0,$$ and thus: $$\phi_Q(r)=\int_0^r\frac{4\tau}{1+\tau^2}d\tau=2\log (1+r^2).$$ We therefore obtain the ground state equation: 
\be
\label{eqqsoliton}
\log Q+\phi_Q-\log 8=0.
\ee
We translate with $x_0\in \R^2$: $$ \log Q(x-x_0)+\phi_Q(x-x_0)-\log 8=0$$ and differentiate at $x_0\in \R^2$ to get: 
\be
\label{cnecneoneonenoe}
\frac{\nabla Q}{Q}+\nabla \phi_Q=0,
\ee and hence \be
\label{pofeueoowj}
\mathcal M(\nabla Q)=\frac{\nabla Q}{Q}+\phi_{\nabla Q}=\frac{\nabla Q}{Q}+\nabla \phi_Q=0.
\ee Similarily, given $\l>0$, let $Q_\l(y)=\l^2Q(\l y)$. We compute using $\int Q=8\pi$:
$$\phi_{Q_\l}(y)=\frac{1}{2\pi}\int \log |x-y| Q_{\l}(y)dy=\frac{1}{2\pi}\int \log\left(\frac{1}{\l}|\lambda x -z|\right) Q(z)dz=-4\log \lambda +\phi_Q(\lambda y)$$
and hence from \fref{eqqsoliton}: 
$$-\log 8+\log Q_\l+\phi_{Q_\l}+2\log \l=0.$$ We differentiate this expression at $\l=1$ and obtain: 
\be
\label{philambdaq}
\frac{\Lambda Q}{Q}+\phi_{\Lambda Q}+2=0,
\ee this is \fref{relationsm}.\\

{\bf step 3} Generalized kernel. Let now $(u,\phi)$  smooth solve \fref{euiheheoifho}, then:
$$\Delta \phi+Q\phi=cQ$$ for some $c\in \Bbb R$. $Q$ being radially symmetric, we develop $\phi$ in real spherical harmonics:
$$\phi(x)=\Sigma_{k\geq 0}\left[\phi_{1,k}(r)\cos(k\theta)+\phi_{2,k}(r)\sin(k\theta)\right]$$ and obtain $$\forall k\in \Bbb N, \ \ H_k\phi_{i,k}=  -\phi_{i,k}''-\frac{\phi_{i,k}'}{r}+\frac{k^2}{r^2}\phi_{i,k}-Q\phi_{i,k}=-cQ\delta_{k=0}.$$ We omit in the sequel the $i$-dependance to simplify notations. For $k\geq 2$, we can from standard argument using the decay $Q=O(\frac1{r^4})$ at infinity construct a basis of solutions to $H_k\psi_{i,k}=0$ with at infinity $$\psi_{1,k}\sim \frac{1}{r^k}, \ \ \psi_{2,k}\sim r^k.$$ The a priori bound \fref{euiheheoifho} implies $\phi_k\in \mbox{Span}\{\psi_{1,k}\}$ and then the regularity of $\phi$ and $H_k\phi_k=0$ easily yield 
\be
\label{cnoneoneo}
\int |\pa_r\phi_k|^2+\int\frac{(\phi_k)^2}{r^2}<+\infty,
\ee and hence $\phi_k$ belongs to the natural energy space associated to $H_k$. Now $$\forall k\geq 2, \ \ H_k\geq H_1+\frac1{r^2}$$ in the sense of quadratic forms, while from \fref{pofeueoowj}, $$H_1(\phi_Q')=0, \ \phi_Q'(r)>0\ \ \mbox{for}\ \ r>0$$ implies from standard Sturm Liouville argument that $\phi_Q'$ is the bound state of $H_1$. Hence $H_k>0$ for $k\geq 2$ and then from \fref{cnoneoneo}: 
\be
\label{bvoebobeo}
\phi_k=0\ \ \mbox{for}\ \ k\ge2.
\ee For $k=1$, the zeroes of $H_1$ are explicit: $$\psi_{1,1}(r)=\phi'_Q(r)=\frac{4r}{1+r^2}, \ \ \psi_{2,1}(r)=-\psi_{1,1}(r)\int_1^r\frac{d\tau}{\tau\psi_{1,1}^2(\tau)}.$$ The second zero is singular at the origin $\psi_{2,1}(r)\sim \frac{c}{r}$ and thus the assumed regularity of $\phi$ implies:
\be
\label{cnbocnoneo}
\phi_1\in \mbox{Span}(\phi_Q').
\ee
For $k=0$, we first observe from \fref{relationsm} that $$\tilde{u}_0=u_0+\frac c2\Lambda Q, \ \ \tilde{\phi}_0=\phi_0+\frac c2\phi_{\Lambda Q}$$ satisfies 
\be
\label{cnbjibeibiberi}
\left\{\begin{array}{ll}\frac{\tilde{u}_0}{Q}+\tilde{\phi}_0=0\\
\Delta \tilde{\phi}_0=\tilde{u}_0
\end{array}\right.\ \ \mbox{and thus}\ \ H_0\tilde{\phi}_0=0
\ee
 Let the partial mass be $$m_0(r)=\int_0^r\tau \tilde{u}_0(\tau)d\tau\ \ \mbox{so that}\ \ m'_0=r \ut_0, \ \ \tilde{\phi}'_0=\frac{m_0}{ r},$$
then:
$$
0=rQ\left(\frac{\ut_0}{Q}+\tilde{\phi}_0\right)'=r\left[\ut_0'-\frac{Q'}{Q}\ut_0+Q\tilde{\phi}'_0\right]=r\left[\left(\frac{m'_0}{r}\right)'-\frac{Q'}{yQ}m_0'+\frac{Q}{r}m_0\right].$$
Equivalently, $$L_0m_0=0$$  where 
\bea
\label{defmo}
L_0 m_0 & = &  -m_0''+\left(\frac{1}{r}+\frac{Q'}{Q}\right)m_0'-Qm_0\\
\nonumber& = & -m_0''-\frac{3r^2-1}{r(1+r^2)}m_0'-\frac{8}{(1+r^2)^2}.
\eea
The basis of solutions to this homogeneous equation is explicit\footnote{this is due to the fact that the non linear equation satisfied by the partial mass of $Q_b$ is invariant by $m(r)\mapsto m(\lambda r)$, see \fref{defphib}, and indeed $\psi_0\in \mbox{Span} \{rm'_Q\}$, $m_Q(r)=\int_0^rQ(\tau)\tau d\tau$.} and given by:
 \be
 \label{basislzero}
 \psi_0(r)=\frac{r^2}{(1+r^2)^2}, \ \ \psi_1(r)=\frac{1}{(1+r^2)^2}\left[r^4+4r^2\log r-1\right].
 \ee
The regularity of $\tilde{u}_0$ implies $m_0(r)\to 0$ as $r\to 0$ and thus $$m_0\in \mbox{Span}(\psi_0), \ \ \tilde{u}_0=\frac{m'_0}{r}\in \mbox{Span}(\frac{\psi'_0}{y})=\mbox{Span}(\Lambda Q)$$ from direct check. Hence $\tilde{u_0}=c\Lambda Q$ and the regularity of $\tilde{\phi}_0$ at the origin\footnote{recall that the only radially symmetric harmonic function in $R^2$ is $\log r$ which is singular at the origin.} implies $\tilde{\phi}_0(r)=c\phi_{\Lambda Q}$,  but then from \fref{relationsm}, \fref{cnbjibeibiberi}: $$0=\frac{\tilde{u}_0}{Q}+\tilde{\phi}_0=c\mathcal M\Lambda Q=-2c\ \ \mbox{and thus}\ \ \tilde{u}_0=\tilde{\phi}_0=0.$$
 Together with \fref{bvoebobeo}, \fref{cnbocnoneo}, this implies: $$\phi\in \mbox{Span}(\phi_{\Lambda Q}, \pa_1 \phi_Q,\pa_2\phi_Q), \ \ u=\Delta \phi\in \mbox{Span}(\Lambda Q, \pa_1Q,\pa_2Q),$$ and \fref{identificationkernel} follows. This concludes the proof of Lemma \ref{lemmakernel}.
\end{proof}
 
 The explicit knowledge of the kernel of $\matchal M$ given by Lemma \ref{lemmakernel} and the variational characterization of $Q$ imply from standard argument the modulated coercivity of the linearized energy. 

\begin{proposition}[Coercivity of the linearized energy]
\label{propenergy}
There exists a universal constant $\delta_0>0$ such that for all $u\in L^2_Q$,
\be
\label{corec}
(\mathcal Mu,u)\geq \delta_0\int\frac{u^2}{Q}-\frac{1}{\delta_0}\left[(u,\Lambda Q)^2+(u,1)^2+\Sigma_{i=1}^2(u,\pa_i Q)^2\right].
\ee
\end{proposition}

{\it Proof of Proposition \ref{propenergy}} A slightly weaker version is proved in \cite{Vlinear} using explicit and somewhat miraculous computations. We give a simpler and stronger argument by adapting the robust proof by \cite{W2}, and which here relies on the classical sharp Hardy Littlewood Sobolev, see also \cite{LMRlinearized} for a very similar argument.\\

{\bf step 1} Positivity. We claim:
\be
\label{positivut}
\forall u\in L^2_Q\ \ \mbox{with} \ \ \int u=0, \ \ (\mathcal Mu,u)\geq 0.
\ee
From standard density argument, it suffices to prove \fref{positivut} for $u\in \mathcal C^\infty_0(\Bbb R^2)$ with $\int u=0$. We apply the sharp logarithmic HLS inequality \fref{logsobolev} at minimal mass $M=8\pi$ to $v=Q+\lambda u>0$ for $\lambda\in \Bbb R$ small enough: $\int v=\int Q=8\pi$ and thus $$F(\l)=\int v\log v+\frac{1}{2}\int v\phi_v\geq F(0).$$ We compute 
\bee
F'(\l) & = & \int\left[(\log(Q+\l u)+1)u+\frac{1}{2}(u\phi_{Q+\l u}+(Q+\l u)\phi_u)\right]\\
& = &  \int\left[(\log(Q+\l u)-\log 8+\phi_Q)u+\l u\phi_u)\right]\\
\eee where we used \fref{cneneoneo} and the vanishing $\int u=0$, and then $Q$ is a critical point from \fref{eqqsoliton}: $$F'(0)=0,$$ and the Hessian is positive: $$F''(\l)_{|\l=0}=\int\left[\frac{u}{Q}+\phi_{u}\right]u=(\mathcal M u,u)\geq 0.$$

{\bf step 2} Coercivity. We now claim the spectral gap: 
\be
\label{cneonoen}
I=\inf\left\{(\mathcal Mu,u), \ \ \|u\|_{L^2_Q}=1, \ \ (u,1)=(u,\Lambda Q)=(u,\pa_i Q)=0\right\}>0.
\ee
We argue by contradiction and consider from \fref{positivut} a sequence $u_n\in L^2_Q$ with $$0\leq (\mathcal Mu_n,u_n)\leq \frac 1n, \ \ \|u_n\|_{L^2_Q}=1, \ \ (u_n,1)=(u_n,\Lambda Q)=(u_n, \pa_i Q)=0.$$ Hence up to a subsequence, 
\be
\label{weakconv}
u_n\rightharpoonup u\ \ \mbox{in}\ \ L^2_Q
\ee which implies using the decay $|Q(y)|\lesssim \frac{1}{1+r^4}$:
\be
\label{cenokneoen}
\int \frac{u^2}{Q}\leq \liminf_{n\to +\infty}\int\frac{u_n^2}{Q}=1, \ \ (u,1)=(u,\Lambda Q)=(u,\pa_i Q )=0.
\ee 
Observe now from \fref{improvedlinfty} that $$\int u_n\phi_{u_n}=-\int |\nabla \phi_{u_n}|^2$$ and we claim from standard argument:
\be
\label{strongfilsltwoloc}
\nabla \phi_{u_n}\to \nabla \phi_{u} \ \mbox{in}\ \ L^2.
\ee
Assume \fref{strongfilsltwoloc}, then from \fref{cenokneoen}: $$(\mathcal Mu,u)\leq 0, \ \ \int\frac{u^2}{Q}\leq 1$$ and thus from \fref{positivut}, $$(\mathcal Mu,u)=0, \ \ \int\frac{u^2}{Q}=\int |\nabla \phi_u|^2\leq 1.$$ Moreover, from the normalization of the sequence: $$(\mathcal Mu_n,u_n)=\int \frac{u_n^2}{Q}-\int |\nabla \phi_{u_n}|^2=1-\int|\nabla \phi_{u_n}|^2\leq \frac 1n$$ and thus using \fref{strongfilsltwoloc}: $$\int|\nabla \phi_u|^2=\lim_{n\to +\infty}\int |\nabla \phi_{u_n}|^2\geq 1,$$ from which 
\be
\label{cneoneoen}
(\mathcal Mu,u)= 0, \ \ \int\frac{u^2}{Q}=1, \ \ (u,1)=(u,\Lambda Q)=(u,\pa_iQ)=0.
\ee Hence $u$ is non trivial and attains the infimum. We conclude from standard Lagrange multiplier argument using the selfadjointness of $\mathcal M$ that: $$\mathcal Mu=a+b\Lambda Q+c\cdot\nabla Q+d\frac{u}{Q}, \ \ u\in L^2_Q.$$ Taking the scalar product with $u$ yields $d=0$ and thus $\mathcal M u\in L^2_Q$. We then take the scalar product with $1,\Lambda Q,\pa_iQ$ and obtain: 
\be
\label{eqcminimia}
\mathcal Mu=0.
\ee
Moreover $(u,\phi_u)\in \mathcal C^{\infty}(\Bbb R^2)$ from standard bootstrapped regularity argument and $\phi_u\in L^{\infty}$ from \fref{improvedlinftybis}. Hence $(u,\phi_u)$ satisfy the generalized kernel equation \fref{euiheheoifho} and $u\in \mbox{Span}(\Lambda Q, \pa_1Q,\pa_2Q)$ from Lemma \ref{lemmakernel}. The orthogonality conditions on $u$ now imply $u\equiv 0$ which contradicts \fref{cneoneoen} and concludes the proof of \fref{cneonoen}.\\
{\em Proof of \fref{strongfilsltwoloc}}: From \fref{improvedlinfty}:
$$\int |\nabla \phi_u-\nabla \phi_{u_n}|^2=-\int (u_n-u)(\phi_{u_n}-\phi_u).$$ From \fref{improvedlinftybis}, \fref{improvedlinfty}, $\phi_{u_n}$ is bounded in $H^1(|x|\leq R)$ for any $R>0$, and thus from the local compactness of Sobolev embeddings and a standard diagonal extraction argument, we may find $\psi \in H^1_{\rm loc}(\R^2)$ such that up to a subsequence\footnote{we need to prove that $\phi=\phi_u$}:
\be
\label{strongconvergcne}
\phi_{u_n}\to \phi\ \ \mbox{in}\ \ L^2_{\rm loc}(\R^2).
\ee
Moreover 
\be
\label{cbeibeibe}
\phi\in L^{\infty}(\R^ 2)
\ee from \fref{improvedlinftybis}. We  then split the integral in two. For the outer part, we use Cauchy Schwarz and \fref{improvedlinftybis} to estimate:
\bee
 -\int_{|x|\geq R} (u_n-u)(\phi_{u_n}-\phi_u)\lesssim  \frac{1}{R}(\|\phi_{u}\|_{L^{\infty}}+\|\phi_{u_n}\|_{L^{\infty}})(\|u\|_{L^2_Q}+\|u_n\|_{L^2_Q})\lesssim \frac{1}{R}.
\eee
For the inner part, we use the strong local convergence \fref{strongconvergcne}, the a priori bound \fref{cbeibeibe} and the weak convergence \fref{weakconv} to conclude:
\bee
& &-\int_{|x|\leq R} (u_n-u)(\phi_{u_n}-\phi_u)=  -\int_{|x|\leq R} (u_n-u)(\phi_{u_n}-\phi)-\int_{|x|\leq R} (u_n-u)(\phi-\phi_u)\\
& \lesssim & (\|u_n\|_{L^2_Q}+\|u\|_{L^2_Q})\|\phi_{u_n}-\phi\|_{L^2(|x|\leq R)}-\int_{|x|\leq R} (u_n-u)(\phi-\phi_u)\\
& \to & 0 \ \ \mbox{as} \ \ n\to +\infty,
\eee
and \fref{strongfilsltwoloc} is proved.\\

{\bf step 3} Conclusion. Let now $u\in L^2_Q$ and $$v=u-aQ-b\Lambda Q-c_1\pa_1Q-c_2\pa_2Q$$ with $$a=\frac{(u,1)}{\|Q\|_{L^1}}, \ \ b=\frac{1}{\|\Lambda Q\|_{L^2}^2}\left[(u,\Lambda Q)+a(Q,\Lambda Q)\right], \ \ c_i=\frac{(u,\pa_iQ)}{\|\pa_iQ\|_{L^2}^2},$$ then $$(v,1)=(v,\Lambda Q)=(v,\pa_iQ)=0$$ and thus from \fref{cneonoen}: $$(\mathcal Mv,v)\geq \delta_0\int\frac{v^2}{Q}\geq \delta_0\int\frac{u^2}{Q}-\frac{1}{\delta _0}(a^2+b^2+c_1^2+c_2^2)$$ and the claim follows by expanding $(\mathcal Mv,v)$. This concludes the proof of Proposition \ref{propenergy}.


\subsection{Structure of $\mathcal L$}

 
 We now study the linearized operator close to $Q$ of the (PKS) flow given by \fref{formuleL} for perturbations in $\matchal E$ given by \fref{energyspcaee}. We define its formal adjoint for the $L^2$ scalar product by 
 \be
\label{foraldjioint}
\mathcal L^*\e=\mathcal M\nabla\cdot(Q\nabla \e).
\ee
  
\begin{lemma}[Continuity of $\mathcal L$ on $\mathcal E$]
\label{lemmacontiniuty}
{\it (i)} Continuity: 
\be
\label{upperbound}
\|\mathcal L\e\|_{L^2_Q}\lesssim \|\e\|_{\mathcal E}.
\ee
{\it (ii)} Adjunction: $\forall (\e,\et)\in \mathcal E^2$,
\be
\label{biebibvei}
 (\mathcal L\e,\et)=(\e,\mathcal L^*\et).
\ee
 {\it (iii) Algebraic identities}: 
  \be
 \label{cneneovneo}
 \mathcal L(\Lambda Q)=\mathcal L(\pa_1Q)=\mathcal L(\pa_2Q)=0,
 \ee
 \be
 \label{nveovneneon}
 \mathcal L^*(y_1)=\mathcal L^*(y_2)=\mathcal L^*(1)=0, \ \ \mathcal L^*(|y|^2)=-4.
 \ee
 {\em (iv) Vanishing average}: $\forall \e\in\mathcal E$,
 \be
 \label{estvooee}
 (\mathcal L\e,1)=0.
 \ee
 \end{lemma}
 
{\it Proof of Lemma \ref{lemmacontiniuty}}.\\
{\em Proof of (i)}: We have from the explicit formula \fref{formuleL} and the energy bound \fref{estcahmos} on the Poisson field:
\bee
\int \frac{|\mathcal L\e|^2}{Q} & \lesssim & \int\frac{|\Delta \e|^2}{Q}+\int Q|\e|^2+\int\frac{|\nabla \phi_Q|^2}{Q}|\nabla \e|^2+\int \frac{|\nabla Q|^2}{Q}|\nabla \phi_\e|^2\\
& \lesssim & \|u\e|_{H^2_Q}^2+\int\frac{|\nabla\phi_\e|^2}{1+r^4}\lesssim \|\e\|_{\mathcal E}^2.
\eee
{\em Proof of  (ii)}: The representation formula
\be\label{reperesnetiaion}
\mathcal L \e=\nabla\cdot(Q\nabla \mathcal M\e)
\ee
ensures that the formal adjoint of $\matchal L$ for the $L^2$ scalar product is given by \fref{foraldjioint}. To justify the integration by parts \fref{biebibvei}, we first remark that both integrals are absolutely convergent. Indeed, from \fref{upperbound}: $$\int |\mathcal L\e||\et|\lesssim \|\mathcal L\e\|_{L^2_Q}\|\et\|_{L^2}\lesssim \|\e\|_{\mathcal E}\|\et\|_{\mathcal E}.$$ Moreover, $$\mathcal L^*\et=\frac{1}{Q}\nabla \cdot(Q\nabla \et)+\phi_{\nabla \cdot (Q\nabla \et)}$$ and from $\int \nabla \cdot(Q\nabla \et)=0$ and \fref{improvedlinftybis}:
$$\|\phi_{\nabla \cdot (Q\nabla \et)}\|^2_{L^{\infty}}\lesssim \|\nabla\cdot(Q\nabla \et)\|^2_{L^2_Q}\lesssim \int \frac{|Q\Delta \et|^2+|\nabla Q\cdot\nabla \et|^2}{Q}\lesssim\|\et\|^2_{H^2_Q},
$$
from which:
$$\int |\e||\mathcal L^*\et|\lesssim \|\e\|_{L^1}\|\phi_{\nabla \cdot (Q\nabla \et)}\|_{L^{\infty}}+\int |\e|\left[|\Delta \et|+\frac{|\nabla Q|}{Q}|\nabla\et|\right]\lesssim \|\e\|_{\mathcal E}\|\et\|_{\mathcal E}.$$
Hence: $$(\mathcal L\e,\et)=\lim_{R\to +\infty}\int_{|x|\leq R}\et\mathcal L\e.$$ Now 
\bee
\int_{|x|\leq R}\et\mathcal L\e& = & \int_{|x|\leq R}\nabla \cdot(Q\nabla \mathcal M \e) \et=\int_{S_R}Q\et\pa_r(\mathcal M\e)d\sigma_R-\int_{|x|\leq R}Q\nabla \mathcal M \et\cdot\nabla \et\\
& = & \int_{S_R}\left[Q\et\pa_r(\mathcal M\e)-Q\pa_r\et\mathcal M\e \right]d\sigma_R+\int_{|x|\leq R}\mathcal M\e\nabla \cdot(Q\nabla \et).
\eee
Moreover, $$\|\nabla \cdot (Q\nabla \et)\|^2_{L^2_Q}\lesssim \int \frac{|\nabla \et\cdot\nabla Q|^2+Q^2|\Delta \et|^2}{Q}\lesssim \|\et\|^2_{H^2_Q},$$ and thus from \fref{efadjointness}:
$$\lim_{R\to +\infty}\int_{|y|\leq R}\mathcal M\e\nabla \cdot(Q\nabla \et)=(\mathcal M\e,\nabla \cdot(Q\nabla \et))=(\e,\mathcal M(\nabla \cdot(Q\nabla \et)))=(\e,\mathcal L^*\et).$$ Moreover, we estimate from Cauchy Schwarz and \fref{estcahmos} :
\bee
\left|\int Q\et\pa_r(\mathcal M\e)\right|\lesssim\int Q|\et|\left[\frac{|\pa_r\e|}{Q}+\frac{|\pa_rQ||\e|}{Q^2}+|\pa_r\phi_\e|\right]<+\infty
\eee 
and from \fref{mcontiniuos}:
\bee
\left|\int Q\pa_r \et\mathcal M\e\right|\lesssim \left(\int Q|\mathcal M\e|^2\right)^{\frac 12}\left(\int Q |\nabla \et|^2\right)^{\frac12}\lesssim  \|\et\|_{H^2_Q}\|\e\|_{H^2_Q}
\eee
and thus we can find a sequence $R_n\to +\infty$ such that $$ \int_{S_{R_n}}\left[Qv\pa_r( \mathcal Mu)-Q\pa_rv\mathcal Mu \right]d\sigma_{R_n}\to 0\ \ \mbox{as}\ \ n\to +\infty,$$ and \fref{biebibvei} follows.\\
{\em  Proof of (iii)} The algebraic identities \fref{cneneovneo} follow directly from \fref{relationsm}, \fref{foraldjioint}: $$\mathcal L^*(y_i)=\mathcal M(\nabla \cdot(Q\nabla y_i))=\mathcal M(\pa_iQ)=0,$$ $$\mathcal L^*(|y|^2)=\mathcal M(\nabla \cdot(Q\nabla |y|^2)=\mathcal M(2\Lambda Q)=-4.$$
{\em  Proof of (iv)}: From \fref{upperbound}, $\matchal L\e\in L^2_Q\subset L^1$, and thus from \fref{reperesnetiaion}: $\forall r_n\to +\infty$,
$$\int\mathcal L\e=\lim_{r_n\to +\infty}\int_{|x|\leq r_n}\mathcal L\e=\lim_{r_n\to +\infty}r_n\int_{0}^{2\pi}Q\pa_r(\mathcal M\e)d\theta.$$
But $$\int|Q\nabla \mathcal M\e|^2\lesssim \int\left[ |\nabla \e|^2+\frac{|\e|^2}{1+|y|^2}+\frac{|\nabla \phi_\e|^2}{1+r^4}\right]<+\infty$$ where we used \fref{estcahmos} , and thus we can find a sequence $r_n\to +\infty$ such that $$\int_0^{2\pi}|Q\pa_r\mathcal M\e|^2d\theta=o\left(\frac{1}{r_n^2}\right).$$ Hence $$r_n\int_{0}^{2\pi}Q\pa_r(\mathcal M\e)d\theta\lesssim \left(r_n^2\int_0^{2\pi}|Q\pa_r \mathcal M \e|^2d\theta\right)^{\frac 12}\to 0\ \ \mbox{as}\  \ r_n\to +\infty,$$ and \fref{estvooee} follows. This concludes the proof of Lemma \ref{lemmacontiniuty}.\\
 
 We now introduce the directions $\Phi_M,\Phi_M^{(1,2)}$ which are localizations of the generalized kernel of $\matchal L^*$, and on which we will construct our set of orthogonality conditions. We anticipate on section \ref{secitonanan} and let $T_1$ be the explicit radially symmetric function given by \fref{deftun} and which satisfies from \fref{eqtonekvne}, \fref{esttoneprop}: 
 \be
 \label{cneobeoeojjejeo}
 \mathcal LT_1=\Lambda Q, \ \ |r^i\pa_r^iT_1|\lesssim \frac{1}{1+r^2}.
 \ee
 We claim:
 
 \begin{lemma}[Direction $\Phi_M$, $\Phi_M^{(1,2)}$]
 \label{direcitonroth}
 Given $M\geq M_0>1$ large enough, we define the directions:
 \be
 \label{defphimzero}
 \Phi_{0,M}(y)=\chi_Mr^2,
 \ee
 \be
\label{defphim}
\Phi_M(y)=\chi_M r^2+c_M\mathcal L^*(\chi_Mr^2), \ \ c_M=-\frac{(\chi_M|y|^2,T_1)}{(\chi_M|y|^2,\Lambda Q)},
\ee
\be
\label{defphitilde}
\Phi_{i,M}=\chi_M y_i, \ \  i=1,2.
\ee
Then:\\
{\it (i) Estimate on $\Phi_M$}:
\be
\label{estphim}
\Phi_M(r)=r^2\chi_M-4c_M\chi_M+\frac{M^2}{\log M}O\left(\frac{1}{1+r^2}{\bf 1}_{r\leq 2M}+{\bf 1}_{M\leq r\leq 2M}\right),
\ee
\be
\label{orthophim}
(\Phi_M,T_1)=0, \ \ (\Phi_M,\Lambda Q)=-(32\pi) \log M+O_{M\to +\infty(1)},
\ee
 \be
 \label{nezscalrar}
(\Phi_{i,M},\pa_iQ)=-\|Q\|_{L^1}+O\left(\frac1M\right).
\ee
{\it (ii) Estimate on scalar products}: $\forall \e \in L^1$, 
\be
\label{estimationorthobis}
|(\e,\Phi_M)| \lesssim  \int_{r\leq 2M}(1+r^2)|\e|+\frac{M^2}{\log M}\left[|(\e,1)|+\int_{r\geq M}|\e|\right],
\ee
\be
\label{bjebbeibei}
|(\e,\mathcal L^*\Phi_{0,M})|\lesssim \int_{r\leq 2M}|\e|,
\ee
\be
\label{newestimate}
\left|(\e,\mathcal L^*\Phi_M)\right|\lesssim  \int_{r\leq 2M}|\e|+\frac{M^2}{\log M}\int_{r\geq M} \frac{|\e|}{1+r^2},
\ee
\be
\label{estfonamentalebus}
 \left|(\e,\mathcal L^*\Phi_M)+c_M^{(1)}(\e,1)\right|+\left|(\e,\mathcal L^*\Phi_{0,M})+c_M^{(2)}(\e,1)\right|\lesssim \int_{r\geq M}|\e|
\ee
with $$c_M^{(i)}=4+O\left(\frac{1}{M^2}\right), \ \ i=1,2,$$
\be
\label{esterivd}
 |(\e,\mathcal L^*\Phi_{i,M})|\lesssim\int_{M\leq r\leq 2M}\frac{|\e|}{1+r}+\frac{1}{M^2}\int\frac{|\e|}{1+r} \ \ i=1,2.
\ee
\end{lemma}

 \begin{remark} It is important for the analysis to keep a sharp track of the $M$ dependence of constants in Lemma \ref{direcitonroth}. The estimate \fref{estfonamentalebus} will play a distinguished role in the analysis for the derivation of the blow up speed. The estimates on scalar products \fref{estimationorthobis}, \fref{newestimate}, \fref{estfonamentalebus}, \fref{esterivd} are a sharp measurement of the effect of localization on the exact formulae \fref{nveovneneon}.
 \end{remark}
 
 \begin{remark} For $\e\in L^2_Q$, \fref{estimationorthobis} implies by Cauchy Schwarz:
 \be
\label{estimationortho}
|(\e,\Phi_M)| +|(\e,\Phi_{0,M})|\lesssim  M\|\e\|_{L^2_Q}+\frac{M^2}{\log M}|(\e,1)|.
\ee
Also \fref{newestimate} implies the rough bound 
\be
\label{novneioheonoen}
|(\e,\matchal L^*\Phi_M)|\lesssim M\|\e\|_{L^2}.
\ee
\end{remark}

{\it Proof of Lemma \ref{direcitonroth}} \\
{\bf step 1} Proof of (i). We integrate by parts and use the cancellation $\mathcal L(\Lambda Q)=0$ to compute:$$(\Phi_M,\Lambda Q)=(\chi_M|y|^2,\Lambda Q)=-2\int|x|^2\chi_M Q-\int|x|^2Qx\cdot\nabla \chi_M=-(32\pi)\log M+O(1).$$ This yields using \fref{cneobeoeojjejeo} the upper bound: 
\be
\label{controlccm}
 |c_M|\lesssim \frac{1}{\log M}\int_{|y|\leq 2M}\frac{|y|^2}{1+|y|^2}\lesssim \frac{M^2}{\log M}.
 \ee
We estimate by definition:
\be
\label{cbeioveibeoiheo}
\mathcal L^*(\chi_M|y|^2)=\mathcal M r_M, \ \ r_M=\nabla \cdot\left[Q\nabla(\chi_M|x|^2)\right]
\ee and thus using \fref{gfijgbeigei}:
\bea
\label{estrm}
\frac{r_M}{Q} & = & \frac{1}{Q}\left[2\chi_M\Lambda Q+|x|^2(Q\Delta \chi_M+\nabla Q\cdot\nabla \chi_M)+4 Qx\cdot\nabla \chi_M\right]\\
\nonumber & = & -4\chi_M+O\left(\frac{1}{1+r^2}{\bf 1}_{r\leq 2M}+{\bf 1}_{M\leq r\leq2M}\right)
\eea
Moreover, solving the Poisson equation for radial fields: $$\left|\phi_{r_M}(r)\right|=\left|\int_{r}^{+\infty}Q\pa_r(\chi_Mr^2)dr\right|\lesssim\frac{1}{1+r^2} {\bf 1}_{r\leq 2M}$$ and thus 
\be
\label{eopjepjpeg}
\mathcal L^*(\chi_M|y|^2)=\mathcal Mr_m=-4\chi_M+O\left(\frac{1}{1+r^2}{\bf 1}_{r\leq 2M}+{\bf 1}_{M\leq r\leq2M}\right)
\ee
 which together with \fref{controlccm} yields \fref{estphim}.\\
We now compute the scalar product: $$(\Phi_M,T_1)=(\chi_M |y|^2+c_M\mathcal L^*(\chi_M|y|^2),T_1)=(\chi_M|y|^2,T_1)+c_M(\mathcal LT_1,\chi_M|y|^2)=0$$ where we used \fref{cneobeoeojjejeo} and the definition \fref{defphim} of $c_M$, and integration by parts which is easily justified using the compact support of $\chi_M|y|^2$ and the decay \fref{cneobeoeojjejeo}. We finally compute after an integration by parts: $$  (\pa_iQ,\chi_Mx_i)=-\|Q\|_{L^1}+O\left(\frac1M\right).$$

{\bf step 2} Proof of \fref{estimationorthobis}, \fref{bjebbeibei}, \fref{newestimate}, \fref{estfonamentalebus}. We claim the pointwise bound:
 \be
 \label{vbeojboihe}
 (\mathcal L^*)^2(\chi_Mr^2)= d_M\chi_M+\frac{{\bf 1}_{r\geq M}}{1+r^2} \ \ \mbox{with}\ \ d_M=O\left(\frac1{M^4}\right).
 \ee
Assume \fref{vbeojboihe}, then from \fref{controlccm}, \fref{eopjepjpeg}, \fref{vbeojboihe}:
$$|(\e,\mathcal L^*\Phi_M)|\lesssim \int_{r\leq 2M}|\e|+\frac{M^2}{\log M}\left[\frac{1}{M^4}\int_{r\leq M}|\e|+\int_{r\geq M} \frac{|\e|}{1+r^2}\right],$$ which implies \fref{newestimate}. The estimate \fref{bjebbeibei} follows from \fref{eopjepjpeg}. Recall now \fref{cbeioveibeoiheo}, then:
\be
\label{estrmtilde}
\tilde{r}_M=r_M-2\Lambda Q=\nabla \cdot\left[Q\nabla((\chi_M-1)|x|^2)\right]=O\left(\frac{{\bf 1}_{|y|\geq M}}{1+|y|^4}\right),
\ee 
and computing the poisson field in radial coordinates:
\bee
\phi_{\tilde{r}_M} & = & -\int_r^{+\infty}Q\pa_{\tau}\left[(1-\chi_M)\tau^2\right] d\tau\\
& = & -{\bf 1}_{r\leq M}\int_M^{+\infty}Q\pa_{\tau}\left[(1-\chi_M)\tau^2\right] d\tau- {\bf 1}_{r\ge M}\int_r^{+\infty}Q\pa_{\tau}\left[(1-\chi_M)\tau^2\right] d\tau\\
& = & e_M\chi_M+O\left(\frac{{\bf 1}_{r\ge M}}{1+r^2}\right)
\eee
with $$e_M=-\int_M^{+\infty}Q\pa_{\tau}\left[(1-\chi_M)\tau^2\right] d\tau=O\left(\frac{1}{M^2}\right).$$
We then compute from $\mathcal M\Lambda Q=-2$: 
\bea
\label{estfonamentale}
  \left|(\e,\mathcal L^*(\chi_M|y|^2))+(4-e_M)(\e,1)\right| & = &   \left|(\e,\mathcal M\tilde{r}_M-e_M)\right|\lesssim  \int_{r\geq M}|\e|.\eea
Hence:
\bee
|(\e,\Phi_M)|& \lesssim&  |(\e,|y|^2\chi_M)|+|c_M|\left|(\e,\mathcal L^*(\chi_M|y|^2))\right|\\
& \lesssim & \int_{r\leq 2M}r^2|\e|+\frac{M^2}{\log M}\left[|(\e,1)|+\int_{r\geq M}|\e|\right]
\eee
 and \fref{estimationorthobis} is proved.  Moreover, from \fref{controlccm}, \fref{vbeojboihe}, \fref{estfonamentale}:
 \bee
  \left|(\e,\mathcal L^*\Phi_M)+(4-e_M-c_Md_M)(\e,1)\right|& \lesssim & \int_{r\geq M}|\e|+  \frac{M^2}{|\log M|}\left[\int_{r\geq M}\frac{|\e|}{1+r^2}\right]\\
 & \lesssim & \int_{r\geq M}|\e|
 \eee
 and \fref{estfonamentalebus} is proved.\\
 {\it Proof of \fref{vbeojboihe}}: First observe from $(\mathcal L^*)^2(|y|^2)=\mathcal L^*(-4)=0$ that $$(\mathcal L^*)^2(\chi_M|y|^2)=\mathcal L^*f_M \ \ \mbox{with}\ \ f_M=\mathcal M\tilde{r}_M.$$ We compute explicitly using the radial representation of the Poisson field
\bee
\mathcal L^*f_M=\mathcal M\nabla \cdot (Q\nabla f_M) & = & \frac{1}{rQ}\pa_r(rQ\pa_rf_M)+\phi_{\frac{1}{r}\pa_r(rQ\pa_rf_M)}\\
& = & \pa^2_r f_M+\left(\frac{1}{r}+\frac{\pa_r Q}{Q}\right)\pa_r f_M-\int_r^{+\infty}Q\pa_{\tau}f_Md\tau.
\eee
We then estimate from \fref{estrmtilde}:
$$|\pa^i_r\tilde{r}_M|\lesssim \frac{{\bf 1}_{r\geq M}}{1+r^{4+i}}, \ \ \left|\pa^i_r\left(\frac{\tilde{r}_M}{Q}\right)\right|\lesssim \frac{{\bf 1}_{r\geq M}}{1+r^{i}}, \ \ i\geq 0.$$ To estimate the Poisson field, we use: $$\pa_r\phi_{\tilde{r}_M}(r)=\frac{1}{r}\int_0^r\tau\tilde{r}_Md\tau$$ and thus: $$|\pa_r\phi_{\tilde{r}_M}(r)|\lesssim {\bf 1}_{r\geq M}\frac{1}{r}\int_M^{+\infty}\frac{\tau d\tau}{1+\tau^4}\lesssim \frac{{\bf 1}_{r\geq M}}{M^2r},$$ and similarily: $$|\pa^i_r\phi_{\tilde{r}_M}(y)|\lesssim \frac{{\bf 1}_{r\geq M}}{M^2r^{i}}, \ \ i\geq 1.$$ This yields the bounds $$|\pa^i_rf_M|\lesssim  \frac{{\bf 1}_{r\geq M}}{1+r^{i}}, $$ 
\bee
\mathcal L^*f_M & = & -{\bf 1}_{r\leq M}\int_M^{+\infty}Q\pa_{\tau}f_Md\tau+{\bf 1}_{r\geq M}\int_r^{+\infty}Q\pa_{\tau}f_Md\tau+O\left(\frac{1}{1+r^2}\right)\\
& = & O\left(\frac{{\bf 1}_{r\geq M}}{1+r^2}\right)+d_M\chi_M
\eee
with $$d_M=-\int_M^{+\infty}Q\pa_{\tau}f_Md\tau=O\left(\frac1{M^4}\right)$$
 and \fref{vbeojboihe} is proved.\\

 {\bf step 3} Proof of \fref{esterivd}.  We estimate using \fref{nveovneneon}: $$\mathcal L^*(\chi_My_i)=\frac{1}{Q}\nabla \cdot(Q\nabla f_M^{(i)})+\phi_{\nabla \cdot(Q\nabla f_M^{(i)})}, \ \ f^{(i)}_M=(\chi_M-1)y_i.$$ The local term is estimated in brute force: $$\left|\frac{1}{Q}\nabla \cdot(Q\nabla f_M^{(i)})\right|=\left|\Delta f_M^{(i)}+\frac{\nabla Q}{Q}\cdot\nabla f_M^{(i)}\right|\lesssim \frac{{\bf 1}_{M\leq |y|\leq 2M}}{1+|y|}.$$ We compute the Poisson field by writing: $$f_M^{(i)}=\zeta_M^{(i)}\frac{y_i}{r}, \ \ \zeta_M^{(i)}(r)=(\chi_M-1)r, \ \ \phi_{\nabla \cdot(Q\nabla f_M^{(i)})}=\Phi_{i,M}(r)\frac{y_i}{r}$$ with $$\pa_r^2\Phi_{i,M}+\frac{\pa_r\Phi_{i,M}}{r}-\frac{\Phi_{i,M}}{r^2}=g_{M}^{(i)}, $$$$ g_M^{(i)}=\frac{1}{r}\pa_r(rQ\pa_r\zeta_M^{(i)})-\frac{Q}{r^2}\zeta_M^{(i)}=O\left(\frac{{\bf 1}_{r\geq M}}{1+r^5}\right).$$ Letting $v^{(i)}_M=\frac{\Phi_{i,M}}{r}$, we obtain: $$\frac{1}{r^3}\pa_r(r^3\pa_rv_M^{(i)})=\frac{g_M^{(i)}}{r}, \ \ \pa_rv_M^{(i)}=\frac{1}{r^3}\int_0^r\tau^2g_M^{(i)}d\tau=O\left(\frac{{\bf 1}_{r\geq M}}{M^2r^3}\right), $$ $$\Phi_{i,M}(r)=-r\int_r^{+\infty}\pa_\tau v_M^{(i)}d\tau=O\left(\frac{1}{M^2(1+r)}\right).$$ We therefore obtain the estimate: 
\be
\label{wbfebebei}
|\mathcal L^*(\chi_M y_i)|\lesssim \frac{1}{1+r}\left[{\bf 1}_{M\leq |y|\leq 2M}+\frac{1}{M^2}\right]
\ee 
and \fref{esterivd} is proved.\\
 This concludes the proof of Lemma \ref{direcitonroth}.


\subsection{$H^2_Q$ bounds}


We now claim the following coercivity property which is the keystone of our analysis. We recall the radial/non radial decomposition \fref{radialnonradial}. The radial symmetry of $Q$ and the structure of the Poisson field ensure: $$(\mathcal L\e)^{(0)}=\mathcal L\e^{(0)}.$$

\begin{proposition}[Coercivity of $\mathcal L $]
\label{interpolationhtwo}
There exist universal constants $\delta_0,M_0>0$ such that $\forall M\geq M_0$, there exists $\delta(M)>0$ such that the following holds. Let $\e\in \mathcal E$ with 
\be
\label{orthowappendix}
(\e,\Phi_M)=(\e,\matchal L^*\Phi_M)=(\e,\Phi_{i,M})=(\e,\mathcal L^*\Phi_{i,M})=0,\ \ i=1,2,
\ee
then there hold the bounds:\\
(i) Control of $\mathcal L\e$: 
\be
\label{coerclwotht}
(\mathcal M\mathcal L\e^{(0)},\mathcal L\e^{(0)})\geq \frac{\delta_0(\log M)^2}{M^2} \int\frac{(\mathcal L \e^{(0)})^2}{Q},
\ee
\be
\label{idemnonraidal}
(\mathcal M\mathcal L\e^{(1},\mathcal L\e^{(1)})\geq \frac{\delta_0}{\log M} \int\frac{(\mathcal L \e^{(1)})^2}{Q}
\ee
(ii) Coercivity of $\mathcal L$: 
\bea
\label{contorlcoerc}
\nonumber \frac{1}{\delta(M)}\int\frac{(\mathcal L \e)^2}{Q}& \geq & \int (1+r^4)|\Delta \e|^2+\int(1+r^2)|\nabla \e|^2+\int \e^2\\
 & + & \int\frac{|\nabla \phi_\e|^2}{r^2(1+|\log r|)^2}.
\eea
\end{proposition}

{\begin{proof}[Proof of Proposition \ref{interpolationhtwo}] 
{\bf step 1} Control of $\matchal L\e$. Let $$\e_2=\mathcal L\e\in L^2_Q.$$ From \fref{estvooee}, $$(\e_2,1)=0$$ and from the choice of orthogonality conditions \fref{orthowappendix} and \fref{biebibvei}: $$(\e_2,\Phi_M)=(\matchal L\e,\Phi_M)=(\e,\mathcal L^*\Phi_M)=0, \ \ (\e_2,\Phi_{i,M})=(\matchal L\e,\Phi_{i,M})=(\e,\mathcal L^*\Phi_{i,M})=0$$ Let:
$$\tilde{\e}_2=\e_2-a_1\Lambda Q-c_1\pa_1Q-c_2\pa_2Q$$ with $$ a_1=\frac{(\e_2,\Lambda Q)}{\|\Lambda Q\|_{L^2}^2}, \ \ c_i=\frac{(\e_2,\pa_iQ)}{\|\pa_iQ\|_{L^2}^2}.$$ Then:
$$(\tilde \e_2,\Lambda Q)=0, \ \ (\tilde{\e}_2,\pa_iQ)=0,\ \ (\tilde{\e}_2,1)=(\e_2,1)=0$$ where we used the fundamental $L^1$ critical degeneracy: $$(\Lambda Q,1)=0.$$ We split into radial and non radial parts and conclude from \fref{corec}:
\be
\label{cnbeoiefhiofhe}
(\mathcal M\tilde{\e}^{(0)}_2,\tilde{\e}^{(0)}_2)\geq \delta_0\|\tilde{\e}^{(0)}_2\|^2_{L^2_Q}, \ \ (\mathcal M\tilde{\e}^{(1)}_2,\tilde{\e}^{(1)}_2)\geq \delta_0\|\tilde{\e}^{(1)}_2\|^2_{L^2_Q}.
\ee
 We now use the orthogonality conditions on $\e_2$, the radially of $Q$ and the bounds \fref{estimationortho}, \fref{orthophim} to estimate:
\bee
|a_1|&=& \left|\frac{(\tilde{\e}_2,\Phi_M)}{(\Phi_M,\Lambda Q)}\right|=\left|\frac{(\tilde{\e}^{(0)}_2,\Phi_M)}{(\Phi_M,\Lambda Q)}\right|\lesssim \frac{1}{|\log M|}\left[M\|\tilde{\e}^{(0)}_2\|_{L^2_Q}+\frac{M^2}{\log M}|(\tilde{\e}^{(0)}_2,1)|\right] \\
& \lesssim & \frac{M}{\log M}\|\tilde{\e}^{(0)}_2\|_{L^2_Q},
\eee
and thus 
\be
\label{cenneoneoen}
\|\e^{(0)}_2\|_{L^2_Q}\lesssim \frac{M}{\log M}\|\tilde{\e}^{(0)}_2\|_{L^2_Q}
\ee
 On the other hand,  thanks to the orthogonality $(\e_2,1)=(\e_2^{(0)},1)=0$ and \fref{relationsm}:
\bee
(\mathcal M \tilde{\e}^{(0)}_2,\tilde{\e}^{(0)}_2) & = & (\mathcal M\e^{(0)}_2+2a_1, \e_2-a_1\Lambda Q)=(\mathcal M\e^{(0)}_2,\e^{(0)}_2)
\eee
and thus from \fref{cnbeoiefhiofhe}, \fref{cenneoneoen}:
$$(\mathcal M\e^{(0)}_2,\e^{(0)}_2)=(\mathcal M \tilde{\e}^{(0)}_2,\tilde{\e}^{(0)}_2)\geq \delta_0\|\tilde{\e}_2^{(0)}\|^2_{L^2_Q}\gtrsim \frac{\delta_0(\log M)^2}{M^2}\|\e_2^{(0)}\|^2_{L^2_Q},$$ this is \fref{coerclwotht}.\\
Similarly, we estimate using \fref{nezscalrar}, \fref{esterivd}:
\bee
|c_i|& = &  \left|\frac{(\tilde{\e}_2,\Phi_{i,M})}{(\Phi_{i,M},\pa_iQ)}\right|=\left|\frac{(\tilde{\e}^{(1)}_2,\Phi_{i,M})}{(\Phi_{i,M},\pa_iQ)}\right|\lesssim\int_{r\leq 2M}r|\et^{(1)}_2|\lesssim \sqrt{\log M}\|\et_2\|_{L^2_Q}
\eee
and thus $$\|\e^{(1)}_2\|_{L^2_Q}\lesssim \sqrt{\log M}\|\tilde{\e}^{(1)}_2\|_{L^2_Q}.$$ We conclude from $\mathcal M\pa_iQ=0$:
$$(\mathcal M\e^{(1)}_2,\e_2^{(1)})=(\mathcal M\et^{(1)}_2,\et_2^{(1)})\geq \delta_0\|\tilde{\e}^{(1)}_2\|^2_{L^2_Q}\geq \frac{\delta_0}{\log M}\|\e^{(1)}_2\|^2_{L^2_Q},$$ this is \fref{idemnonraidal}.\\

{\bf step 2} Subcoercivity of $\mathcal L$. The spectral gap \fref{contorlcoerc} follows as in \cite{RaphRod}, \cite{MRR}, \cite{RS} from a compactness argument and the explicit knowledge of the kernel of $\matchal L$. We first claim as a consequence of two dimensional Hardy inequalities and the explicit repulsive structure of the operator far out the subcoercivity estimate: $\forall \e\in H^2_Q$,
\bea
\label{subcerovitybis}
\nonumber  \int\frac{(\mathcal L\e)^2}{Q} & \gtrsim & \int (1+r^4)(\Delta \e)^2+\int\left[\frac{1}{r^2(1+|\log r|)^2}+r^2\right]|\nabla \e|^2+\int \e^2+\int\frac{|\nabla \phi_\e|^2}{r^2(1+|\log r|)^2}\\
&  - &   \int\frac{|\nabla \phi_\e|^2}{1+r^4}-\int\frac{\e^2}{1+r^2}-\int |\nabla\e|^2.
\eea
Indeed, from \fref{formuleL}:
\bea
\label{cnoneoneoc}
\nonumber \int \frac{(\mathcal L\e)^2}{Q}&\gtrsim& \int\frac{(\Delta \e+\nabla \phi_Q\cdot\nabla \e)^2 }{Q}-\int\frac{|\nabla Q|^2}{Q}|\nabla \phi_\e|^2-\int Q\e^2\\
& \gtrsim & \int\frac{(\Delta \e+\nabla \phi_Q\cdot\nabla \e)^2 }{Q}-\int \frac{\e^2}{1+r^4}- \int\frac{|\nabla \phi_\e|^2}{1+r^4}.
\eea
We develop:
\bee
\int\frac{(\Delta \e+\nabla \phi_Q\cdot\nabla \e)^2 }{Q}=\int\frac{(\Delta \e)^2}{Q}+\frac{(\pa_r\phi_Q\pa_r\e)^2 }{Q}+\int\frac{2}{Q}\Delta \e\nabla\phi_Q\cdot \nabla \e.
\eee
We observe from \fref{cnecneoneonenoe}: $$\frac{2\nabla\phi_Q}{Q}=-2\frac{\nabla Q}{Q^2}=2\nabla \psi, \ \ \psi=\frac{1}{Q}.$$ We now recall the classical Pohozaev integration by parts formula:
\bee
\nonumber &&2\int \Delta \e\pa_r\psi\pa_r\e=    2 \int_0^{+\infty}\int_0^{2\pi}\left[\pa_r(r\pa_r\e)+\frac{1}{r}\pa^2_\theta \e\right]\pa_r\psi\pa_r\e drd\theta\\
\nonumber & = & -\int_0^{+\infty}\int_0^{2\pi}(r\pa_r\e)^2\pa_r\left(\frac{\pa_r\psi}{r}\right)drd\theta+\int_0^{+\infty}\int_0^{2\pi}(\pa_\theta \e)^2\pa_r\left(\frac{\pa_r\psi}{r}\right)drd\theta\\
& = & -\int\left[(\pa_r\e)^2-\left(\frac1r\pa_{\theta}\e\right)^2\right] \left[\pa^2_r\psi-\frac{\pa_r\psi}{r}\right].
\eee
We compute for $r\gg 1$: $$\psi(r)=\frac{1}{Q}=\frac{r^4}{8}+O(r^2), \ \ \pa^2_r\psi-\frac{\pa_r\psi}{r}=r^2+O(1),$$ $$\phi_Q'(r)=\frac{1}{r}\int_0^rQ(\tau)\tau d\tau=\frac{4}{r}+O(\frac{1}{r^3}),\ \ \frac{(\pa_r\phi_Q)^2}{Q}=2r^2+O(1)$$  from which:
\bee
&&\int\frac{(\Delta \e+\nabla \phi_Q\cdot\nabla \e)^2 }{Q}\\
&  \gtrsim &  \int(1+r^4)(\Delta \e)^2+\int\left[(2r^2-r^2)(\pa_r\e)^2+r^2(\frac{1}{r}\pa_{\theta}\e)^2\right]-\int|\nabla\e|^2.
\eee
Injecting this into \fref{cnoneoneoc} yields the lower bound:
$$
 \int\frac{(\mathcal L\e)^2}{Q}  \gtrsim  \int(1+r^4)(\Delta \e)^2+\int r^2|\nabla \e|^2-\int \frac{\e^2}{1+r^4}-\int|\nabla\e|^2- \int\frac{|\nabla \phi_\e|^2}{1+r^4}
$$
and \fref{subcerovitybis} follows from the Hardy bounds \fref{hardyboundbis}, \fref{nenoeneo} and the relation $\e=\Delta \phi_\e$.\\

{\bf step 3} Coercivity of $\mathcal L$. We are now in position to prove the coercivity \fref{contorlcoerc}. We claim that it is enough to prove it for $\e\in \mathcal C^{\infty}_c(\R^2)$. Indeed, if so, let $\e\in \mathcal L$ and $$\e_n\to \e\ \ \mbox{in}\ \ \mathcal E, \ \ \e_n\in \mathcal C^{\infty}_c(\R^2),$$ then $$\|\mathcal L\e_n\|_{L^2_Q}\to \|\mathcal L\e\|_{L^2_Q}$$ from \fref{upperbound}. Moreover, from \fref{estcahmos} and Sobolev: $$\int\frac{|\nabla \phi_{\e_n}-\nabla \phi_\e|^2}{r^2(1+|\log r|)^2}\lesssim \|\nabla \phi_{\e_n}-\nabla\phi_\e\|_{L^{\infty}}^2\lesssim \|\e_n-\e\|_{\mathcal E}^2\to 0,$$ and we may thus pass to the limit in \fref{contorlcoerc} for $\e_n$ and conclude that \fref{contorlcoerc} holds for $\e$.\\
In order to prove \fref{contorlcoerc} for $\e\in \mathcal C^{\infty}_c(\R^2)$, assume by contradiction that there exists a sequence $\e_n\in \mathcal C^{\infty}_c(\R^2)$ such that\footnote{Working with $\e_n\in \mathcal C^{\infty}(\R^2)$ ensures that $\phi_{\e_n}$ makes sense while only $\nabla\phi_\e$ is well defined for $\e\in \mathcal E$, and we may thus recover a Hardy type control on $\phi_{\e_n}$ from \fref{nenoeneobis}.} \bea
\label{estinitiale}
&&\int (1+r^4)(\Delta \e_n)^2+\int\left[\frac{1}{r^2(1+|\log r|)^2}+r^2\right]|\nabla \e_n|^2+\int \e_n^2\\
\nonumber & + & \int\frac{|\nabla \phi_{\e_n}|^2}{r^2(1+|\log r|)^2}+\int\frac{\phi_{\e_n}^2}{(1+r^4)(1+|\log r|)^2}=1,
\eea 
$$(\e_n,\Phi_M)=(\e_n,\matchal L^*\Phi_M)=(\e_n,\pa_iQ)=0,\ \ i=1,2$$  and 
\be
\label{degenineo}
\int \frac{(\mathcal  L \e_n)^2}{Q}\leq \frac 1n.
\ee 
From \fref{estinitiale} and the local compactness of Sobolev embeddings\footnote{recall that $\e_n=\Delta \phi_{\e_n}$}, we have up to a subsequence:
\be
\label{stronenoenvo}
\phi_{\e_n}\to \phi\ \ \mbox{in}\ \ H^1_{\rm loc}
\ee
 Similarily, $\e_n$ is bounded in $H^2(\Bbb R^2)$ and we may extract up to a subsequence 
\be
\label{vnononre}
\e_n\rightharpoonup \e\ \ \mbox{in}\ \ H^2, \ \ \e_n\to \e\ \mbox{in}\ \ H^1_{\rm loc}
\ee 
and there holds from standard lower semi continuity estimates the  a priori bounds:
\bea
\label{lowerestimte}
&&\int (1+r^4)(\Delta \e)^2+\int\left[\frac{1}{r^2(1+|\log r|)^2}+r^2\right]|\nabla \e|^2+\int \e^2\\
\nonumber & + & \int\frac{|\nabla \phi|^2}{r^2(1+|\log r|)^2}+\int\frac{\phi^2}{(1+r^4)(1+|\log r|)^2}\lesssim 1.
\eea
The subcoercivity estimate \fref{subcerovitybis}, the initialization \fref{estinitiale}, the assumption \fref{degenineo} and the strong convergences \fref{stronenoenvo}, \fref{vnononre} yield the non degeneracy:
\bea
\label{cbeobeoheoheog}
\nonumber
&& \int\frac{\e^2}{1+r^2}+\int|\nabla \e|^2+\int\frac{|\nabla \phi|^2}{1+r^4}+\int\frac{\phi^2}{1+r^6}\\
\nonumber & = & \lim_{n\to +\infty}\left[\int\frac{\e_n^2}{1+r^2}+\int|\nabla \e_n|^2+\int\frac{|\nabla \phi_{\e_n}|^2}{1+r^4}+\int\frac{\phi_{\e_n}^2}{1+r^6}\right]\\
& \gtrsim &1.
\eea
Passing to the limit in the distribution sense, we also conclude $$\Delta \phi=\e\ \ \mbox{in}\ \ \mathcal D'(\R^2).$$ Note however that we do not know the relation $\phi=\phi_\e$.\\
We now pass to the limit in $$\mathcal L\e_n=\Delta \e_n+2 Q\e_n+\nabla \phi_Q\cdot\nabla \e_n+\nabla Q\cdot\nabla \phi_{\e_n}$$ and conclude from \fref{stronenoenvo}, \fref{vnononre} that $$\mathcal L\e_n\rightharpoonup \Delta \e+2Q\e+\nabla \phi_Q\cdot\nabla \e+\nabla Q\cdot\nabla \phi\ \ \mbox{in}\ \ \mathcal D'(\Bbb R^2)$$ while $\mathcal L\e_n\to 0$ in $L^2(\R^2)$ from \fref{degenineo} and thus 
\be
\label{eqaieje}
 \left\{\begin{array}{ll} \Delta \e+2Q\e+\nabla \phi_Q\cdot\nabla \e+\nabla Q\cdot\nabla \phi=0\\
   \Delta \phi=\e\end{array}\right.\ \ \mbox{in}\ \ \mathcal D'(\R^2).
 \ee
 The a priori bound \fref{lowerestimte} and standard elliptic regularity estimates ensure the bootstrapped regularity $(\e,\phi)\in C^{\infty}(\Bbb R^2)$ and thus \fref{eqaieje} holds is strong sense.\\
 Let $$f=\frac{\e}{Q}+\phi,$$ we rewrite \fref{eqaieje} as the divergence equation:
 \be 
 \label{divejbfebi}
 \nabla \cdot(Q\nabla f)=0.
 \ee
 Moreover, we have the a priori bound from \fref{lowerestimte}:
 \be
 \label{eniowenoeno}
 \int \frac{Q}{1+r^2}|\nabla f|^2\lesssim \int\frac{1}{1+r^2}\left[\frac{|\nabla \e|^2}{Q}+\frac{\e^2}{(1+y^2)Q}+Q|\nabla \phi|^2\right]<+\infty
 \ee
 and thus in particular there exists a sequence $r_n\to +\infty$ such that:
 \be
 \label{ypboundaryond}
 \lim_{r_n\to +\infty}\int_0^{2\pi}Q|\nabla f_n|^2(r_n,\theta)d\theta=0.
 \ee
 We claim from standard argument in Liouville classification of diffusion equations that this implies 
 \be
 \label{nabkabweero}
 \nabla f=0.
 \ee 
 Assume \fref{nabkabweero}, then from \fref{lowerestimte}, \fref{nabkabweero}, $(\e,\phi)$ satisfies \fref{euiheheoifho} and hence from Lemma \ref{lemmakernel}: $$\e\in \mbox{Span}(\Lambda Q,\pa_1 Q,\pa_2Q), \ \ \phi=\phi_\e.$$ The orthogonality conditions $(\e,\Phi_M)=(\e,\pa_1Q)=(\e,\pa_2Q)=0$ and \fref{orthophim} imply $\e\equiv 0$ and thus $\phi=\phi_\e\equiv 0$ which contradicts the non degeneracy \fref{cbeobeoheoheog} and concludes the proof of \fref{contorlcoerc} for $\e\in \mathcal C^{\infty}_c(\R^2)$.\\
 {\it Proof of \fref{nabkabweero}}: We integrate \fref{divejbfebi} on $B_r$ and use the regularity of $f$ at the origin and the spherical symmetry of $Q$ to compute: 
 $$ 0=r\int_0^{2\pi}Q(r)\pa_rf(r,\theta)d\theta=rQ(r)\pa_r\left(\int_0^{2\pi}f(r,\theta)d\theta\right).
 $$ and hence 
 \be
 \label{cneienoeno}
\pa_rf^{(0)}(r)=0, \  \ f^{(0)}(r)=\frac{1}{2\pi}\int_0^{2\pi}f(r,\theta)d\theta.
\ee
 We then multiply  \fref{divejbfebi} by $f$, integrate over $B_r$ and estimate using \fref{cneienoeno}, the spherical symmetry of $Q$ and the sharp Poincar\'e inequality on $[0,2\pi]$:
 \bea
 \label{cneenoenoen}
\nonumber && \int_{B_r}Q(y)|\nabla f(y)|^2dy  =  r\int_0^{2\pi}Q(r)f(r,\theta)\pa_rf(r,\theta)d\theta=rQ(r)\int_0^{2\pi}\left[f-f^{(0)}\right]\pa_rfd\theta\\
 \nonumber & \leq& rQ(r)\left(\int_0^{2\pi}(f-f^{(0)})^2d\theta\right)^{\frac12}\left(\int_0^{2\pi}(\pa_rf)^2d\theta\right)^{\frac 12}\\
\nonumber  & \leq &  rQ(r)\left(\int_0^{2\pi}(\pa_{\theta}f)^2d\theta\right)^{\frac12}\left(\int_0^{2\pi}(\pa_rf)^2d\theta\right)^{\frac 12}\\
 & \leq & \frac{r^2Q(r)}{2}\int_0^{2\pi}\left[(\pa_rf)^2+\left(\frac1r\pa_\theta f\right)^2\right]d\theta=\frac{r^2}{2}\int_0^{2\pi}Q(r)|\nabla f|^2(r,\theta)d\theta.
 \eea
 Hence $$g(r)= \int_{B_r}Q(y)|\nabla f(y)|^2dy$$ satisfies the differential inequation: $$g'(r)=r\int_0^{2\pi}Q(r)|\nabla f|^2(r,\theta)d\theta\geq \frac{2g}{r}$$ ie 
 \be
 \label{neonronro}
 \frac{d}{dr}\left[\frac{g}{r^2}\right]\geq 0.
 \ee On the other hand, the boundary condition \fref{ypboundaryond} and the control \fref{cneenoenoen} imply $$\lim_{r_n\to +\infty}\frac{g(r_n)}{r_n^2}=0$$ which together with the positivity of $g$ and the monotonicity \fref{neonronro} implies $g\equiv 0$, and \fref{nabkabweero} is proved.\\
 This concludes the proof of Proposition \ref{interpolationhtwo}.
\end{proof}


\section{Construction of approximate blow up profiles}
\label{construction}

This section is devoted to the construction of suitable approximate blow up profiles which contain the main qualitative informations on the solution. These profiles are one mass super critical continuation of the exact mass subcritical self similar solutions exhibited in \cite{BDP}.


\subsection{Approximate blow up profiles}
\label{secitonanan}

We build the family of approximate self similar solutions. The key in the construction is to track in a sharp way the size of tails at infinity.  We start with building radial profiles which contain the leading order terms.

\begin{proposition}[Radial blow up profiles]
\label{propblowup}
Let $M>0$ enough large, then there exists a small enough universal constant $b^*(M)>0$ such that the following holds. Let $b \in ]0,b^*(M)[$ and $B_0,B_1$ be given by \fref{defb}, then there exist radially symmetric profiles $T_1(r), T_2(b,r)$ such that 
$$Q_b(r) = Q(r) + bT_1(r)+ b^2T_2(b,r)$$
is an approximate self similar solution in the following sense. Let the error:
\be
\label{deferreur}
\Psi_b=\nabla\cdot(\nabla Q_b+Q_b\nabla \phi_{Q_b})-b\Lambda Q_b+c_b b^2 \chi_{\frac{B_0}{4}}T_1,
\ee
with $c_b$ given by \fref{defcb}, then there holds:\\
{\em (i) Control of the tails}: $ \forall r\geq 0$, $ \forall i\geq 0$:
\be
\label{esttoneprop}
|r^i\pa^i_rT_1|\lesssim r^2{\bf 1}_{r\le 1}+\frac{1}{r^2}{\bf 1}_{r\geq 1},
\ee
and $ \forall r\leq 2B_1$, $ \forall i\geq 0$:
\be
 \label{esttwo}
 |r^i\pa_r^iT_2|\lesssim r^4{\bf 1}_{r\le1} +\frac{1 + |\log (r\sqrt b)|}{|\log b|}{\bf 1}_{1\leq r\leq 6B_0}+ \frac{1}{b^2r^4|\log b|}{\bf 1}_{r\geq 6B_0},
 \ee
\be
\label{estttwodtdb}
|b\pa_br^i\pa_r^iT_2|\lesssim\frac{1}{|\log b|}\left[ r^4{\bf 1}_{r\le1} +\frac{ |\log r|}{|\log b|}{\bf 1}_{1\leq r\leq 6B_0}+ \frac{1}{b^2r^4}{\bf 1}_{r\geq 6B_0}\right].
\ee
{\em (ii) Control of the error in weighted norms}: for $i\geq 0$,
\be
\label{roughboundltaow}
\int_{r\leq 2B_1} |r^i\pa_r^i\Psi_b|^2+\int_{r\leq 2B_1} \frac{|\phi'_{\Psi_b}|^2}{1+\tau^2}+\int_{r\leq 2B_1}\frac{|\mathcal L\Psi_b|^2}{Q} \lesssim \frac{b^5}{|\log b|^2},
\ee
\be
\label{cneneoneonoenoe}
\int_{r\leq 2B_1} Q|\nabla \mathcal M\Psi_b|^2\lesssim \frac{b^4}{|\log b|^2}.
\ee
\end{proposition}

\begin{remark} The error $\Psi_b$ displays slowly decaying tail as $r\o +\infty$, and hence the norm in which is measured the error is essential. This explains why polynomials in $b$ in the error vary from \fref{roughboundltaow} to \fref{cneneoneonoenoe}. These errors are the leading order terms in the control of the radiation for the full solution, see step 4 of the proof of Proposition \ref{htwoqmonton}.
\end{remark}

\begin{proof}[Proof of Proposition \ref{propblowup}]
{\bf step 1} Setting the computation on the mass.\\
We look for radial profiles and it is therefore simpler -but not necessary- to work with the partial mass: 
\be
\label{partialmass}
m_b(r)=\int_0^r Q_b(\tau)\tau d\tau.
\ee
 The Poisson field for radial solutions is given by $$\phi'_{Q_b}=\frac{m_b}{r}.$$ Let using \fref{deferreur}: 
 \be
 \label{defphib}
 \Phi_b=m''_b-\frac{m_b'}{r}+\frac{m_bm_b'}{r}-brm'_b,\ \ 
 \Psi_b=\frac{1}r\Phi_b'+c_b b^2 \chi_{\frac{B_0}{4}}T_1.
 \ee
 We proceed to an expansion $$m_b=m_0+bm_1+b^2m_2$$ where $$m_0(r)=\int_0^rQ(\tau)\tau d\tau=\frac{4r^2}{1+r^2}=-\frac{rQ'}{Q} = r  \phi'_Q.$$ Correspondingly, $$Q_b=Q+bT_1+b^2T_2 =\frac{m_b'}{r}.$$ We let the linearized operator close to $m_0$ be given by \fref{defmo}:
 $$ L_0 m=-m''+\left(\frac{1}{r}+\frac{Q'}{Q}\right)m'-Qm=-m''-\frac{3r^2-1}{r(1+r^2)}m'-\frac{8}{(1+r^2)^2}m
 $$
 and obtain the expansion:
 \bea
 \label{computationphib}
 \nonumber \Phi_b&=& b\left[-L_0m_1-rm'_0\right]+b^2\left[-L_0m_2+\frac{m_1m'_1}{r}-rm'_1\right] \\
 \nonumber &+&   b^3\left[ \frac{(m_1m_2)'}{r}-rm'_2\right] +b^4\left[\frac{m_2m_2' }{r} \right].
  \eea
 
 {\bf step 2} Inversion of $L_0$. The Green's functions of $L_0$ are explicit\footnote{This structure is reminiscent from the parabolic heat flow problem and one could show that this operator can be factorized $L_0=A_0^*A_0$ where the adjoint is taken against $\frac{(1+r^2)^2}rdr$ and $A_0$ is first order, and this explains why all formulas are explicit.} 
 and the set of radial solutions to the homogeneous problem $$L_0m=0$$ is spanned according to \fref{basislzero} by:
 \be
 \label{defpsiunspdio}
 \psi_0(r)=\frac{r^2}{(1+r^2)^2}, \ \ \psi_1(r)=\frac{r^4+4r^2\log r-1}{(1+r^2)^2}
 \ee with Wronskian 
 \be
 \label{defW}
 W=\psi_1'\psi_0-\psi_1\psi_0'=\frac{rQ}{4}=\frac{2r}{(1+r^2)^2}.
 \ee
  Hence a solution to $$L_0m=-f$$ can be found by the method of variation of constants:
 \be
 \label{fromulau}
 m=A\psi_0+B\psi_1\ \ \mbox{with}\ \ \left\{\begin{array}{ll}A'\psi_0+B'\psi_1=0,\\ A'\psi'_0+B'\psi_1'=f, \end{array}\right ..
 \ee 
This leads to
 $$B'=\frac{f\psi_0}{W}=\frac{r}{2}f, \ \ A'=-\frac{f\psi_1}{W}=-\frac{r^4+4r^2\log r-1}{2r}f$$ and a solution is given by:
  \be
 \label{inversionknot}
 m(r)=-\frac12\psi_0(r)\int_0^r\frac{\tau^4+4\tau^2\log \tau-1}{\tau}f(\tau)d\tau+\frac12\psi_1(r)\int_0^r\tau f(\tau)d\tau.
 \ee
 We compute
 \be
 \label{estderivatibves}
 \frac{\psi_0'}{r}=\frac{2(1-r^2)}{(1+r^2)^3}, \ \  \frac{\psi_1'}{r}=\frac{8(1+r^2-(r^2-1)\log r)}{(1+r^2)^3}.
 \ee
and then \fref{fromulau} yields 
\bea
\label{formulamprime}
\frac{m'}{r} & = & A\frac{\psi'_0}{r}+B\frac{\psi'_1}{r}\\
\nonumber & = & -\frac{1-r^2}{(1+r^2)^3}\int_0^r\frac{\tau^4+4\tau^2\log \tau-1}{\tau}f(\tau)d\tau+\frac{4(1+r^2-(r^2-1)\log r)}{(1+r^2)^3}\int_0^r\tau f(\tau)d\tau.
\eea

{\bf step 3} Construction of $m_1,T_1$. We let $m_1$ be the solution to 
\be
\label{demlo}
L_0m_1=-rm'_0=-r^2Q=-\frac{8r^2}{(1+r^2)^2}=-8\psi_0
\ee
 given by \fref{inversionknot}, explicitly: $$m_1(r)=-4\psi_0(r)\int_0^r\frac{\tau(\tau^4+4\tau^2\log \tau-1)}{(1+\tau^2)^2}d\tau+4\psi_1(r)\int_0^r\frac{\tau^3}{(1+\tau^2)^2}d\tau.$$ Then from \fref{formulamprime}:
 \bea
 \label{deftun}
 T_1& = & \frac{m'_1}{r}\\
 \nonumber & = & -\frac{8(1-r^2)}{(1+r^2)^3}\int_0^r\frac{\tau(\tau^4+4\tau^2\log \tau-1)}{(1+\tau^2)^2}d\tau+\frac{32(1+r^2-(r^2-1)\log r)}{(1+r^2)^3}\int_0^r\frac{\tau^3}{(1+\tau^2)^2}d\tau.
 \eea
 There holds the behavior at the origin $$m_1(r)=O(r^4), \ \ T_1(r)=O(r^2).$$ For $r$ large, we use the explicit formula:
 $$\int_0^r\frac{\tau^3}{(1+\tau^2)^2}d\tau=\frac{\log (1+r^2)}{2}+\frac12\left(\frac{1}{1+r^2}-1\right)$$ to compute:
 \be
 \label{boundm1}
 m_1(r)=4\left(\log r- 1\right)+O\left(\frac{|\log r|^2}{r^2}\right),
 \ee
 \be
 \label{developpementmonprime}
 m'_1(r)=\frac{4}{r}+O\left(\frac{|\log r|^2}{r^3}\right),
 \ee
 \be
 \label{developpementT1}
 T_1(r)=\frac{4}{r^2}+O\left(\frac{|\log r|^2}{r^4}\right).
 \ee 
 This yields in particular the bound for $i\geq 0$: $$ |r^i\pa_r^{i}T_1|\lesssim r^2{\bf 1}_{r\leq 1}+\frac{1}{r^2}{\bf 1}_{r\geq 1}.$$
 
 {\bf step 4} Construction of the radiation. We now introduce the radiation term which will allow us to adjust in a sharp way the tail of $T_2$ outside the parabolic zone $r\geq B_0$.  Let 
 \be
 \label{defcb}
 c_b = \frac 1{\int _0^{+\infty}  \chi_{\frac{B_0}4} \psi_0(\tau)  \tau d\tau}=\frac{2}{|\log b|}\left[ 1 + O \left( \frac{1}{|\log b|}\right)\right]
 \ee
 \be
 \label{defdb}
 d_b=4c_b\int_0^{+\infty} \chi_{\frac{B_0}{4}}\psi_1(\tau) \tau d\tau  = O\left(\frac{1}{b|\log b|}\right)
 \ee
  We let the radiation $\Sigma_b$ be the solution to 
 \be
 \label{equationradiation}
 L_0\Sigma_b=-8c_b\chi_{\frac{B_0}{4}}\psi_0+d_bL_0\left[(1-\chi_{3B_0})\psi_0\right]
 \ee
 given from \fref{inversionknot} by 
\bea
\label{fromulezardiaition}
\nonumber \Sigma_b(r) & = &-4c_b\psi_0(r)\int_0^r\frac{\tau(\tau^4+4\tau^2\log \tau-1)}{(1+\tau^2)^2}\chi_{\frac{B_0}{4}}d\tau+4c_b\psi_1(r)\int_0^r\frac{\tau^3}{(1+\tau^2)^2}\chi_{\frac{B_0}{4}}d\tau\\
& + &d_b(1-\chi_{3B_0})\psi_0(r).
\eea
Observe that by definition:
\be
\label{propradiaiton}
\Sigma_b=\left\{\begin{array}{ll} c_bm_1\ \ \mbox{for}\ \ r\leq \frac{B_0}{4} \\4\psi_1\ \ \mbox{for}\ \ r\geq 6B_0\end{array}\right ..
\ee
and we estimate for $\frac{B_0}{4} \le r \le 6B_0$:
\bea
\nonumber \Sigma_b(r)=4+O\left(\frac{1}{|\log b|}\right).
\eea
For $r\geq 6B_0$, we observe from \fref{defpsiunspdio}, \fref{propradiaiton} the degeneracy: $$\Sigma_b=4+O\left(\frac{\log r}{r^2}\right), \ \ r\pa_r\Sigma_b=O\left(\frac{\log r}{r^2}\right).$$
This yields in particular the bounds using \fref{boundm1}:
$$r^i\pa_r^i\Sigma_b(r)=O\left(\frac{r^4}{|\log b|}\right)\ \ \mbox{for}\ \ r\leq 1,\ \ i\ge 0,$$
\be
\label{estimateradiation}
\forall 1\leq r\leq 6B_0, \ \ \Sigma_b(r)=\frac{8\log r}{|\log b|}+O\left(\frac{1}{|\log b|}\right),
\ee
\be
\label{estcninoenoloin}
\forall r\geq 6B_0, \ \ \Sigma_b(r)=4+O\left(\frac{\log r}{r^2}\right).
\ee
We now observe an improved bound for derivatives far out. Indeed, from \fref{formulamprime}: for $r\leq 2B_1$,
\bee
\left|\frac{\Sigma'_b}{r}\right|& =-& \left| -4c_b\frac{\psi'_0(r)}{r}\int_0^r\frac{\tau(\tau^4+4\tau^2\log \tau-1)}{(1+\tau^2)^2}\chi_{\frac{B_0}{4}}d\tau+4c_b\frac{\psi'_1(r)}{r}\int_0^r\frac{\tau^3}{(1+\tau^2)^2}\chi_{\frac{B_0}{4}}d\tau\right.\\
& + & \left. d_b(1-\chi_{3B_0})\frac{\psi'_0(r)}{r}-d_b\chi'_{3B_0}\frac{\psi_0}{r}\right|\\
& \lesssim & \frac{1}{|\log b|(1+r^4)}\left[r^2{\bf 1}_{r\le1 }+ r^2{\bf 1}_{r\leq 6B_0}+\frac{1}{b}{\bf 1}_{r\geq 6B_0}\right]\\
& + & \frac{1+|\log r|}{|\log b|(1+r^4)}\left[(1+|\log r|){\bf 1}_{1\leq r\leq 6B_0}+|\log b|{\bf 1}_{r\geq B_0}\right]+  \frac{1}{b|\log b|(1+r^4)}{\bf 1}_{r\geq 3B_0}\\
& \lesssim & \frac{1}{|\log b|}\left[r^2{\bf 1}_{r\leq 1}+\frac{1}{r^2}{\bf 1}_{1\leq r\leq 6B_0}+\frac{1}{br^4}{\bf 1}_{r\geq 6B_0}\right]
\eee
and thus: for $i\geq 1$, $r\leq 2B_1$,
\be
\label{estradiationimproved}
|r^i\pa^i_r\Sigma_b|\lesssim \frac{1}{|\log b|}\left[r^4{\bf 1}_{r\leq 1}+{\bf 1}_{r\leq 6B_0}+\frac{1}{br^2}{\bf 1}_{r\geq 6B_0}\right].
\ee
We now estimate the $b$ dependance of $\Sigma_b$. From \fref{defcb}, \fref{defdb}, $$\frac{\pa c_b}{\pa b}=O\left(\frac{1}{b|\log b|^2}\right), \ \ \frac{\pa d_b}{\pa b}=O\left(\frac{1}{b^2|\log b|^2}\right),$$ and from \fref{propradiaiton}, $$\frac{\pa \Sigma_b}{\pa b}=\left\{\begin{array}{ll} \frac{\pa c_b}{\partial b}m_1\ \ \mbox{for}\ \ r\leq \frac{B_0}{4}\\ 0\ \ \mbox{for} \ \ r\geq 6B_0\end{array}\right..$$ This leads to the bound:
$$r^i\pa_r^i\frac{\pa\Sigma_b}{\pa b}(r)=O\left(\frac{r^4}{b|\log b|^2}\right)\ \ \mbox{for}\ \ r\leq 1,$$ $$r^i\pa_r^i\frac{\pa\Sigma_b}{\pa b}(r)=O\left(\frac{1+|\log r| }{b|\log b|^2}\right)\ \ \mbox{for}\ \ 1\leq r\leq \frac{B_0}{4}, \ \ i\geq 0.$$
 In the transition zone $\frac{B_0}{4}\leq r\leq 6B_0$, we estimate from \fref{fromulezardiaition}:
\bee
\left|\frac{\pa\Sigma_b}{\pa b}\right|& \lesssim & \frac{1}{b|\log b|}+\frac{|c_b|}{1+r^2}\int \tau\left|\partial_b\chi_{\frac{B_0}4}\right|d\tau+|c_b|\int\frac{d\tau}{\tau}\left|\pa_b\chi_{\frac{B_0}4}\right|+\frac{|d_b|}{1+r^2}|\pa_b\chi_{3B_0}|\\
& \lesssim & \frac{1}{b|\log b|}
\eee
and similarly for higher derivatives. This yields the bound: for $i\geq 0$, $r\leq 2B_1$,
\be
\label{estdsigmadb}
\left|br^i\pa_r^i\pa_b\Sigma_b\right|\lesssim \frac{1}{|\log b|}\left[r^4{\bf 1}_{r\leq 1}+\frac{1+|\log r|}{|\log b|}{\bf 1}_{r\leq 6B_0}\right].
\ee

{\bf step 5} Construction of $m_2,T_2$. We define: 
\be
\label{defsigmatwo}
\Sigma_2=\frac{m_1m'_1}{r}-rm'_1+\Sigma_b.
\ee
We estimate $$\Sigma_2=O(r^4)\ \ \mbox{for}\ \ r\leq 1.$$  For $1\leq r\leq 6B_0$, there holds the bound using  \fref{developpementmonprime}, \fref{estimateradiation}:
\bea
\label{sigma2milieu}
 \nonumber \Sigma_2& = & -4+\frac{8\log r}{|\log b|}+O\left(\frac{1}{|\log b|}\right)+O\left(\frac{|\log r|^2}{1+r^2}\right)\\
 & = & O\left(\frac{|\log (\sqrt{b}r)|}{|\log b|}\right)+O\left(\frac{1}{|\log b|}\right)
 \eea
 and for $r \geq 6B_0$, we have from \fref{defpsiunspdio}, \fref{developpementmonprime}, \fref{propradiaiton}:
 \be
 \label{sigma2infty}
 \Sigma_2 =  -4+4 \psi_1+O\left(\frac{|\log r|^2}{1+r^2}\right) = O\left(\frac{|\log r|^2}{1+r^2}\right).
 \ee 
 Hence, we obtain: $\forall i\geq 0$,
\bea
\label{sigma2bord}
|r^i\pa_r^i\Sigma_2| \lesssim r^4{\bf 1}_{ r \leq 1} + \frac{1 + |\log (r\sqrt b)|}{|\log b|} {\bf 1}_{1 \leq r \leq 6B_0} + \frac{|\log r|^2}{1+r^2} {\bf 1}_{r \geq 6B_0}.
\eea
\begin{remark} The essential feature of \fref{sigma2bord} is the $\frac{1}{|\log b|}$ smallness of the $\Sigma_2$ tail in the parabolic zone $r\sim B_0$ which justifies the introduction of the radiation $\Sigma_b$ in \fref{defsigmatwo}, and this will turn out to modify the modulation equation for $b$ according to \fref{odeintro}.
\end{remark}
The $b$ dependance of $\Sigma_2$ is estimated from \fref{estdsigmadb}, \fref{defsigmatwo}:
\be
\label{dsigmatwodb}
\left|br^i\pa_r^i\pa_b\Sigma_2\right|\lesssim \frac{1}{|\log b|}\left[r^4{\bf 1}_{r\leq 1}+\frac{1+|\log r|}{|\log b|}{\bf 1}_{r\leq 6B_0}\right].
\ee
We now let $m_2$ be the solution to $$L_0m_2=\Sigma_2$$ given by
 \be
 \label{defm2}
 m_2(r)=\frac12\psi_0(r)\int_0^r\frac{\tau^4+4\tau^2\log \tau-1}{\tau}\Sigma_2(\tau)d\tau-\frac12\psi_1(r)\int_0^r\tau \Sigma_2(\tau)d\tau.
 \ee
 Near the origin, $$m_2=O(r^6).$$ For $1\leq r\leq 6B_0$, 
 \bee
 |m_2(r)|& \lesssim &\frac{1}{1+r^2}\int_0^r \tau^3\frac{1 + |\log (\tau\sqrt b)|}{|\log b|} d\tau+\int_0^r\tau\frac{1 + |\log (\tau\sqrt b)|}{|\log b|}d\tau\\
 & \lesssim & r^2\frac{1 + |\log (r\sqrt b)|}{|\log b|}.
 \eee
 For $r\geq 6B_0$, 
 \bee
 |m_2(r)|& \lesssim&  \frac{1}{b|\log b|}+\frac{1}{r^2}\int_{6B_0}^r\tau^3\frac{|\log \tau|^2}{1+\tau^2}d\tau+\int_{6B_0}^r\tau\frac{|\log \tau|^2}{1+\tau^2}d\tau\\
 & \lesssim &  \frac{1}{b|\log b|}+(\log r)^3.
 \eee
 We now observe the following cancellation when exiting the parabolic zone $r\geq 6B_0$ which will be important to treat the error term  in the sequel\footnote{Equivalently, one should observe that the $T_2$ equation is forced by $\Lambda T_1$ which enjoys the improved decay at infinity $\Lambda T_1=O(\frac{|\log r|^2}{r^4})$, see \cite{RaphRod} for related phenomenons.}: for $ 6B_0\leq r\leq 2B_1$,
 \bea
 \label{imprvedmoneout}
 \nonumber 
\left| \frac{m_2'}{r} \right|& = &\left|\frac12\frac{\psi'_0(r)}r\int_0^r\frac{\tau^4+4\tau^2\log \tau-1}{\tau}\Sigma_2(\tau)d\tau-\frac12\frac{\psi'_1(r)}r\int_0^r\tau \Sigma_2(\tau)d\tau\right|\\
\nonumber  & \lesssim & \frac{1}{r^4}\int_0^{r}\frac{\tau^4+4\tau^2\log \tau-1}{\tau}\left[\frac{1 + |\log (\tau\sqrt b)|}{|\log b|} {\bf 1}_{1 \leq \tau\leq 6B_0}+ \frac{|\log \tau|^2}{1+\tau^2} {\bf 1}_{\tau\geq 6B_0}\right]d\tau\\
\nonumber & + & \frac{|\log r|}{r^4}\int_0^{r}\tau\left[\frac{1 + |\log (\tau\sqrt b)|}{|\log b|} {\bf 1}_{1 \leq \tau \leq 6B_0}+ \frac{|\log \tau|^2}{1+\tau^2} {\bf 1}_{\tau \geq 6B_0}\right]d\tau\\
\nonumber& \lesssim & \frac{1}{r^4}\left[\frac{1}{b^2|\log b|}+r^2(\log r)^2\right]+\frac{\log r}{r^4}\left[\frac{1}{b|\log b|}+(\log r)^3\right]\\
& \lesssim & \frac{1}{r^4b^2|\log b|}.
 \eea
 The collection of above bounds yields the control: $\forall r\leq 2B_1$
 \be
 \label{estmtwo}
 |m_2|\lesssim r^6{\bf 1}_{r\le1} +r^2\frac{1 + |\log (r\sqrt b)|}{|\log b|}{\bf 1}_{1\leq r\leq 6B_0}+ \frac{1}{b|\log b|}{\bf 1}_{r\geq 6B_0},
 \ee
  \be
 \label{estmtwoderivaitve}
 |r^i\pa_r^im_2|\lesssim r^6{\bf 1}_{r\le1} +r^2\frac{1 + |\log (r\sqrt b)|}{|\log b|}{\bf 1}_{1\leq r\leq 6B_0}+ \frac{1}{r^2b^2|\log b|}{\bf 1}_{r\geq 6B_0}, \ \ i\geq 1.
 \ee
 The $b$ dependance is estimated using \fref{dsigmatwodb}:
  \bea
 \label{estmtwoddb}
 \nonumber  |b\pa_bm_2|&\lesssim & \frac{r^2}{1+r^4}\int_0^r\frac{1+\tau^4}{\tau}\frac 1{|\log b|}\left[\tau^4{\bf 1}_{\tau\leq 1}+\frac{1+|\log \tau|}{|\log b|}{\bf 1}_{1\leq\tau\leq 6B_0}\right]d\tau\\
 \nonumber & + & \int_0^r\frac{\tau}{|\log b|}\left[\tau^4{\bf 1}_{r\leq 1}+\frac{1+|\log \tau|}{|\log b|}{\bf 1}_{1\le \tau\leq 6B_0}\right]d\tau\\
  & \lesssim & \frac{1}{|\log b|}\left[r^6{\bf 1}_{r\leq 1}+r^2\frac{1+|\log r|}{|\log b|}{\bf 1}_{1\leq r\leq 6B_0}+\frac{1}b {\bf 1}_{r\geq 6B_0}\right]
  \eea
    and for higher derivatives:
  \bee
  \left| \frac{b\pa_b\pa_rm_2}{r} \right|& \lesssim  &\frac{1}{(1+r^4)|\log b|}\int_0^{r}\frac{\tau^4+4\tau^2\log \tau-1}{\tau}\left[\tau^4{\bf 1}_{\tau\leq 1}+\frac{|\log \tau|}{|\log b|}{\bf 1}_{\tau\leq 6B_0}\right]d\tau\\
  & + &  \frac{|\log r|}{(1+r^4)|\log b|}\int_0^{r}\tau\left[\tau^4{\bf 1}_{\tau\leq 1}+\frac{|\log \tau|}{|\log b|}{\bf 1}_{\tau\leq 6B_0}\right]d\tau\\
  & \lesssim & \frac{1}{|\log b|}\left[r^4{\bf 1}_{r\le1} +\frac{|\log r|}{|\log b|}{\bf 1}_{1\leq r\leq 6B_0}+ \frac{1}{r^4b^2}{\bf 1}_{r\geq 6B_0}\right],
  \eee
and hence the bounds for $r\leq 2B_1$:
\be
 \label{estmtwodtdb}
 |b\pa_bm_2|\lesssim \frac{1}{|\log b|}\left[r^6{\bf 1}_{r\le1} +r^2\frac{|\log r|}{|\log b|}{\bf 1}_{1\leq r\leq 6B_0}+ \frac{1}{b}{\bf 1}_{r\geq 6B_0}\right],
 \ee
 and for $i\geq 1$:
  \be
 \label{estmtwoderivaitvedtdb}
 |b\pa_br^i\pa_r^im_2|\lesssim \frac{1}{|\log b|}\left[r^6{\bf 1}_{r\le1} +r^2\frac{|\log r|}{|\log b|}{\bf 1}_{1\leq r\leq 6B_0}+ \frac{1}{r^2b^2}{\bf 1}_{r\geq 6B_0}\right].
 \ee
This yields the estimate on $T_2=\frac{m_2'}{r}$: for $i\geq 0$, $r\leq 2B_1$:
$$
 |r^i\pa_r^iT_2|\lesssim r^4{\bf 1}_{r\le1} +\frac{1 + |\log (r\sqrt b)|}{|\log b|}{\bf 1}_{1\leq r\leq 6B_0}+ \frac{1}{b^2r^4|\log b|}{\bf 1}_{r\geq 6B_0},
$$
$$|b\pa_br^i\pa_r^iT_2|\lesssim\frac{1}{|\log b|}\left[ r^4{\bf 1}_{r\le1} +\frac{ |\log r|}{|\log b|}{\bf 1}_{1\leq r\leq 6B_0}+ \frac{1}{b^2r^4}{\bf 1}_{r\geq 6B_0}\right].
$$
 
  {\bf step 6} Estimate on the error.\\

By construction, $\Phi_b$ given by \fref{defphib} satisfies:
\be
\label{bejbeobveoboe}
\Phi_b=-b^2\Sigma_b+R(r)
\ee
with
$$
R(r) =  b^3\left[ \frac{(m_1m_2)'}{r}-rm'_2\right] +b^4\left[\frac{m_2m_2' }{r} \right]
$$
and the error to self similarity given by \fref{deferreur} satisfies:
\be
\label{defphibbiwbiwbi}
 \Psi_b=\frac1r\Phi_b'+b^2c_bT_1\chi_{\frac{B_0}{4}}=\frac{R'}{r}-b^2\left[\frac{\Sigma'_b}{r}-c_bT_1\chi_{\frac{B_0}{4}}\right].
 \ee
We have from direct check the bound\footnote{the worst term is $b^3rm'_2$}: for $i\geq 0$,
\be
\label{estr}
|r^i\pa_r^iR(r)| \lesssim b^3 \left[ r^6{\bf 1}_{r\le1} +r^2\frac{1 + |\log (r\sqrt b)|}{|\log b|}{\bf 1}_{1\leq r\leq 6B_0}+ \frac{1}{r^2b^2|\log b|}{\bf 1}_{r\geq 6B_0}\right]
\ee
This yields using \fref{propradiaiton}, \fref{estradiationimproved},  the bound for $r\leq 2B_1$:
\bea
\label{esibgeogogeo}
 |\Psi_b|& \lesssim & b^3 \left[ r^4{\bf 1}_{r\le1} +\frac{1 + |\log (r\sqrt b)|}{|\log b|}{\bf 1}_{1\leq r\leq 6B_0}+ \frac{1}{r^4b^2|\log b|}{\bf 1}_{r\geq 6B_0}\right]\\
\nonumber & + &  \frac{b^2}{|\log b|}\left[\frac{1}{r^2}{\bf 1}_{\frac{B_0}{4}\leq r\leq 6B_0}+\frac{1}{br^4}{\bf 1}_{r\geq 6B_0}\right]\\
\eea
We therefore estimate:
\bee
\int_{r\leq 2B_1} |\Psi_b|^2& \lesssim & b^6\int_{r\leq 2B_1}\left[{\bf 1}_{r\leq 1} +\left(\frac{ |\log (r\sqrt b)|}{|\log b|}{\bf 1}_{1\leq r\leq 6B_0}\right)^2+ \frac{1}{r^8b^4|\log b|^2}{\bf 1}_{r\geq 6B_0}\right]\\
& + & \frac{b^4}{|\log b|^2}\int_{\tau\leq 2B_1}\left[\frac{1}{\tau^4}{\bf 1}_{\frac{B_0}{4}\leq \tau\leq 6B_0}+\frac{1}{b^2\tau^8}{\bf 1}_{\tau\geq 6B_0}\right]\\
& \lesssim & \frac{b^5}{|\log b|^2}
\eee
and similarily for higher derivatives:
$$\int _{r\leq 2B_1} |r^i\pa^i_r\Psi_b|^2\lesssim  \frac{b^5}{|\log b|^2}.$$
Moreover, the radial representation of the Poisson field and \fref{defphibbiwbiwbi} ensure: 
\bee
\phi'_{\Psi_b} & = & \frac{\Phi_b}{r}+\frac{b^2c_b}{r}\int_0^rm'_1(\tau)\chi_{\frac{B_0}{4}}d\tau=\frac{R}{r}-\frac{b^2}{r}\int_0^r\left[\frac{\Sigma'_b}{\tau}-c_bT_1\chi_{\frac{B_0}{4}}\right]\tau d\tau. \eee
We estimate from \fref{propradiaiton}, \fref{estimateradiation}, \fref{estradiationimproved}:
\bee
\left|\frac{b^2}{r}\int_0^r\left[\frac{\Sigma'_b}{\tau}-c_bT_1\chi_{\frac{B_0}{4}}\right]\tau d\tau\right|\lesssim \frac{b^2}{|\log b|}\left[\frac{{\bf 1}_{\frac{B_0}{4}\leq r\leq 6B_0}}{r}+\frac{1}{br^3}{\bf 1}_{r\geq 6B_0}\right]
\eee
wich together with \fref{estr} yields the bound:
\bee
&&\int_{\tau\leq 2B_1} \frac{|\phi'_{\Psi_b}|^2}{1+\tau^2}\lesssim  \frac{b^4}{|\log b|^2}\int_{\tau\leq 2B_1}\frac{1}{1+\tau^2}\left[\frac{{\bf 1}_{\frac{B_0}{4}\leq \tau\leq 6B_0}}{\tau^2}+\frac{1}{b^2\tau^6}{\bf 1}_{\tau\geq 6B_0}\right]\\
& + & b^6\int_{r\leq 2B_1}\frac{1}{1+\tau^2}\left[ \tau^3{\bf 1}_{\tau\le1} +\tau\frac{1 + |\log (\tau\sqrt b)|}{|\log b|}{\bf 1}_{1\leq \tau\leq 6B_0}+ \frac{1}{\tau^3b^2|\log b|}{\bf 1}_{\tau\geq 6B_0}\right]^2\\
& \lesssim & \frac{b^5}{|\log b|^2}.
\eee
We further estimate from the definition \fref{formuleL}:
$$
\int_{r\leq 2B_1} \frac{|\mathcal L\Psi_b|^2}{Q} \lesssim  \Sigma_{i=0}^2\int_{r\leq 2B_1}(1+r^{2i})|\pa^i_r\Psi_b|^2+\int_{r\leq 2B_1} \frac{|\pa_r\phi_{\Psi_b}|^2}{1+r^6}\lesssim \frac{b^5}{|\log b|^2}$$ which concludes the proof of \fref{roughboundltaow}.
Finally, we obtain the control of the stronger norm:
\bee
\int_{r\leq 2B_1} Q|\nabla \mathcal M\Psi_b|^2&\lesssim & \int_{r\leq 2B_1} (1+r^2)\left[|r\pa_r\Psi_b|^2+|\Psi_b|^2\right]+\frac{|\phi'_{\Psi_b}|^2}{1+r^4}\\
& \lesssim & b^6\int_{r\le 2B_1}  \left[ 1+\frac{(1 + |\log (r\sqrt b)|^2)(1+r^2)}{|\log b|^2}{\bf 1}_{1\leq r\leq 6B_0}+ \frac{1}{r^6b^4|\log b|}{\bf 1}_{r\geq 6B_0}\right]\\
\nonumber & + &  \frac{b^4}{|\log b|^2}\int_{r\le 2B_1}\left[\frac{1}{r^2}{\bf 1}_{\frac{B_0}{4}\leq r\leq 6B_0}+\frac{1}{b^2r^6}{\bf 1}_{r\geq 6B_0}\right]+\frac{b^5}{|\log b|^2}\\
& \lesssim & \frac{b^4}{|\log b|^2},
\eee
and \fref{cneneoneonoenoe} is proved. This concludes the proof of Proposition \ref{propblowup}.
\end{proof}

\begin{remark} Observe from \fref{deftun} and the readily verified formula $$\mathcal L \left(\frac{m'}{r}\right)=-\frac{1}{r}(L_0m)'$$ that
\be
\label{eqtonekvne}
\matchal LT_1=\matchal L\left(\frac{m_1'}{r}\right)=-\frac{1}{r}(L_0m_1)'=\frac{8\psi_0'}{r}=\Lambda Q.
\ee
\end{remark}
 
 \subsection{Localization}
 

We now proceed to a brute force localization of the tails of $T_1,T_2$ which become irrelevant strictly outside the parabolic zone $r\geq 2B_1$. We need to be careful with localization due to the non local structure of the problem, and we shall rely again on the specific structure of the Poisson equation for radial fields. 
\begin{proposition}[Localization]
\label{proploc}
Given a small parameter
$$
0<b\ll 1,$$ let the localized profile:
\be
\label{deflocprofiles}
\qbt=Q+b\tt_1+b^2\tt_2\ \ \tt_i=\chi_{B_1}T_i \ \ \mbox{for}\  \ i=1,2.
\ee  Let the error:
\bea
\label{deferreurtilde}
\Psit_b& = & \nabla\cdot(\nabla \qbt+\qbt\nabla \phi_{\qbt})-b\Lambda \qbt+ c_b b^2 \chi_{\frac{B_0}{4}}T_1
\eea
then there holds:\\
{\em (i) Control of the tails}: $ \forall i\geq 0$:
\be
\label{esttoneproploc}
|r^i\pa^i_r\tt_1|\lesssim r^2{\bf 1}_{r\le 1}+\frac{1}{r^2}{\bf 1}_{1\leq r\leq 2B_1},
\ee
\be
\label{esttpemfmp}
|b\pa_b r^i\pa^i_rT_1|\lesssim \frac{1}{r^2}{\bf 1}_{B_1\leq r\le 2B_1},
\ee
\be
 \label{esttwoloc}
 |r^i\pa_r^i\tt_2|\lesssim r^2{\bf 1}_{r\le1} +\frac{1 + |\log (r\sqrt b)|}{|\log b|}{\bf 1}_{1\leq r\leq 6B_0}+ \frac{1}{b^2r^4|\log b|}{\bf 1}_{ 6B_0\leq r\leq 2B_1},
 \ee
\be
\label{estttwodtdbloc}
|b\pa_br^i\pa_r^i\tt_2|\lesssim\frac{1}{|\log b|}\left[ r^2{\bf 1}_{r\le1} +\frac{ |\log r|}{|\log b|}{\bf 1}_{1\leq r\leq 6B_0}+ \frac{1}{b^2r^4}{\bf 1}_{6B_0\leq r\leq 2B_1}\right].
\ee
{\em (ii) Control of the error in weighted norms}: 
\be
\label{roughboundltaowloc}
\int |r^i\nabla^i\Psit_b|^2+\int \frac{|\nabla \phi_{\tPsi_b}|^2}{1+r^2}+\int\frac{|\mathcal L\tPsi_b|^2}Q\lesssim \frac{b^5}{|\log b|^2}
\ee
\be
\label{roughboundltaowlocbisbis}
\int Q|\nabla \mathcal M\Psit_b|^2\lesssim \frac{b^4}{|\log b|^2}.
\ee
{\em (iii) Degenerate flux}: Let $\frac{B_0}{20}\leq B\leq 20 B_0$ and $\Phi_{0,B}=\chi_B r^2$, then 
\be
\label{degeencos}
|(\mathcal L\Psit_b,\Phi_{0,B})|\lesssim \frac{b^2}{|\log b|}.
\ee

\end{proposition}

\begin{proof}[Proof of Proposition \ref{proploc}]
{\bf step 1} Terms induced by localization.
We compute from \fref{deflocprofiles}:
$$|b\pa_b\tt_i|=\left|b\chi_{B_1}\pa_bT_i-b\frac{\partial_b B_1}{B_1}(y\chi')\left(\frac{y}{B_1}\right)\right|\lesssim |b\pa_bT_i|{\bf 1}_{r\leq 2B_1}+|T_i|{\bf 1}_{B_1\leq r\leq 2B_1}.$$  and the bounds \fref{esttoneprop}, \fref{esttwo}, \fref{estttwodtdb} now yield \fref{esttoneproploc}, \fref{esttpemfmp}, \fref{esttwoloc}, \fref{estttwodtdbloc}.\\
Let now the partial mass $$m_b(r)=\int_0^r Q_b\tau d\tau, \ \ \tilde{m}_b(r)=\int_0^r\tilde{Q}_b\tau d\tau,$$ and consider the decomposition$$m_b=m_0+n_b,\ \ \tilde{m}_b=m_0+\tn_b, \ \ m_0(r)=\int_0^r Q\tau d\tau,$$ then by definition $$\tn'_b=\chi_{B_1}n'_b.$$ We compute the error induced by localization at the level of the masses. Let like for \fref{defphib}: $$\tPsi_{b}=\frac{\Phit'_b}{r}+b^2c_bT_1\chi_{\frac{B_0}{4}}, \ \ \tilde{\Phi}_b=\tm''_b-\frac{\tm_b'}{r}+\frac{\tm_b\tm_b'}{r}-br\tm'_b,$$ then: 
\bea
\label{defphibtilde}
\nonumber \tilde{\Phi}_b& = & \tn''_b-\frac{\tn'_b}{r}+\frac 1r(\tn_bm'_0+m_0\tn_b'+\tn_b\tn'_b)-br(m'_0+\tn_b')\\
& = & \chi_{B_1}\Phi_b+Z_b
\eea
with
\bea
\label{defzb}
\nonumber Z_b& = & \chi'_{B_1}n'_b+\frac{m'_0}{r}(\tn_b-\chi_{B_1}n_b)+\frac{1}{r}\chi_{B_1}n'_b(\tn_b-n_b)\\
& - & b(1-\chi_{B_1})rm'_0.
\eea
We estimate from \fref{boundm1}, \fref{estmtwo}:
\bea
\label{estnbjop}
\nonumber  |n_b|& \lesssim &  b\left[r^2{\bf 1}_{r\leq 1}+|\log r|{\bf 1}_{r\ge 1}\right]+  b^2\left[r^4{\bf 1}_{r\le1} +r^2\frac{1 + |\log (r\sqrt b)|}{|\log b|}{\bf 1}_{1\leq r\leq 6B_0}+ \frac{1}{b|\log b|}{\bf 1}_{r\geq 6B_0}\right]\\
& \lesssim & b\left[r^2{\bf 1}_{r\leq 1}+|\log r|{\bf 1}_{r\ge 1}\right],
 \eea
 and for $i\ge 1$ from \fref{estmtwoderivaitve}:
\bee
|r^i\pa^i_rn'_b|& \lesssim &  b\left[r^2{\bf 1}_{r\leq 1}+{\bf 1}_{r\ge 1}\right]+ b^2\left[r^4{\bf 1}_{r\le1} +r^2\frac{1 + |\log (r\sqrt b)|}{|\log b|}{\bf 1}_{1\leq r\leq 6B_0}+ \frac{1}{b^2r^2|\log b|}{\bf 1}_{r\geq 6B_0}\right]\\
& \lesssim &   b\left[r^2{\bf 1}_{r\leq 1}+{\bf 1}_{r\ge 1}\right].
\eee
This yields the bound:
\bea
\label{boundzb}
\nonumber |r^i\pa^i_rZ_b| &\lesssim & \frac{b}{r^2}{\bf 1}_{B_1\le r\leq 2B_1}+\frac{b|\log b|}{1+r^4}{\bf 1}_{r\geq B_1}+\frac{1}{r}{\bf 1}_{B_1\leq r\leq 2B_1}\frac{b^2|\log b|}{r}+  \frac{b}{r^2}{\bf 1}_{r\geq B_1}\\
&\lesssim & \frac{b}{r^2}{\bf 1}_{r\geq B_1}.
\eea 
We compute:
\be
\label{bneoibeonaoheovneonv}
\tPsi_{b}=\chi_{B_1}\frac{\tPhi'_b}{r}+\frac{\chi'_{B_1}}{r}\Phi_b+\frac{Z'_b}{r}+b^2c_bT_1\chi_{\frac{B_0}{4}}=\chi_{B_1}\Psi_b+\frac{\chi'_{B_1}}{r}\Phi_b+\frac{Z'_b}{r}.
\ee
We now estimate from \fref{bejbeobveoboe}, \fref{estcninoenoloin}, \fref{boundzb}:
$$
|\tPsi_{b}-\chi_{B_1}\Psi_b| \lesssim \frac{b^2}{r^2}{\bf 1}_{B_1\leq r\leq 2B_1}+  \frac{b}{r^4}{\bf 1}_{r\geq B_1}\\
$$
 and hence using \fref{roughboundltaow}:
\be
\label{firsteatmie}
\int |r^i\pa_r^i\tPsi_{b,0}|^2\lesssim \frac{b^5}{|\log b|^2}+\int_{r\geq B_1}\left[\frac{b^4}{r^4}+\frac{b^2}{r^8}\right]\lesssim \frac{b^5}{|\log b|^2}.
\ee
We now estimate using the radial representation of the Poisson field and \fref{bneoibeonaoheovneonv}, \fref{boundzb}:
$$
\phi'_{\tPsi_{b}}=\phi'_{\Psi_b}\ \ \mbox{for}\ \ r\leq B_1$$ and for $r\geq B_1$: 
\bee
|\phi'_{\tPsi_{b}}|\lesssim \frac{1}{r}\int_{B_1}^r\left[ \frac{b^2}{\tau^2}{\bf 1}_{B_1\leq \tau\leq 2B_1}+\frac{b}{\tau^4}{\bf 1}_{\tau\geq B_1}\right]\tau d\tau\lesssim \frac{b^2}{r}{\bf 1}_{r\geq B_1}
\eee
and thus from \fref{roughboundltaow}:
$$\int\frac{|\phi'_{\tPsi_{b}}|^2}{1+r^2}\lesssim \frac{b^5}{|\log b|^2}+\int_{r\geq B_1} \frac{b^4}{r^4}\lesssim \frac{b^5}{|\log b|^2}.$$ This last estimate together with \fref{firsteatmie} implies from \fref{formuleL}: 
\be
\label{cenkoneoneon}
\int \frac{|\mathcal L\Psit_{b}|^2}{Q}\lesssim \frac{b^5}{|\log b|^2}.
\ee

{\bf step 2} Degenerate flux.\\

We now turn to the proof of \fref{degeencos}. We estimate using \fref{estimationortho}, \fref{cenkoneoneon} with the cancellation \fref{estvooee}:
$$|(\matchal L\Psit_{b},\Phi_{0,B})|\lesssim  B\|\matchal L\Psit_{b}\|_{L^2_Q}\lesssim \frac{b^{\frac 52}}{\sqrt{b}|\log b|}\lesssim \frac{b^2}{|\log b|}$$
and \fref{degeencos} is proved.
\end{proof}

 
 \section{The bootstrap argument}
 \label{sectionboot}
 
 
 In this section, we set up the bootstrap arguments  at the heart of the proof of Theorem \ref{thmmain}. We will in particular give an explicit description of the set of initial data and detail the preliminary information on the dynamical flow, in particular modulation equations. The dynamical bounds on radiation in higher weighted Sobolev norms at the heart of the analysis are then derived in section \ref{htwowbound}. The boostratp is finally closed in section \ref{proofthmmai}, and this will easily yield the statements of Theorem \ref{thmmain}.
 
 
 \subsection{Geometrical description of the set of initial data}
 

 Let us describe explicitly the set of initial data $\mathcal O$ for which the conclusions of Theorem \ref{thmmain} hold.\\
  We first recall the standard modulation argument.
 
 \begin{lemma}[$L^1$ modulation]
 \label{lemmamodutlation}
 Let $M>M^*$ large enough, then there exists a universal constant $\delta^*(M)>0$ such that $\forall v\in L^1$ with $$\|v-Q\|_{L^1}<\delta^*(M),$$ there exists a unique decomposition $$v=\frac{1}{\lambda_1^2}(\qbt+\e_1)\left(\frac{x}{\l_1}\right)$$ such that $$(\e_1,\Phi_M)=(\e_1,\mathcal L^*\Phi_M)=0.$$ Moreover, $$\|\e_1\|_{L^1}+|\lambda_1-1|+|b|\lesssim C(M)\delta^*.$$
\end{lemma} 

\begin{proof}[Proof of Lemma \ref{lemmamodutlation}] This is a standard consequence of the implicit function theorem. Let us denote $$\Mod=(\l_1,b), \ \ v_{\rm Mod}=\lambda_1^2v(\l_1x)-\qbt$$ and consider the $\matchal C^1$ functional $$F(v,\l_1,b)=\bigg[(v_{\rm Mod},\Phi_M),(v_{\rm Mod},\mathcal L^*\Phi_M)\bigg].$$ Then $F(Q,1,0)=0$, and we compute the Jacobian at $(v,\rm Mod)= (Q,1,0)$:
\bee
\frac{\pa}{\pa\l_1}(v_{\rm Mod},\Phi_M)_{|(v,\rm Mod)  =   (Q,1,0,)}& = & \frac{\pa}{\pa\l_1}\left(v,\Phi_M(\frac{x}{\l_1})\right)_{|(v,\rm Mod)= (Q,1,0)}\\
& = & -(Q,y\cdot\nabla \Phi_M)=(\Lambda Q,\Phi_M)\\
& = & -24\log M+O(1),
\eee
\bee
\frac{\pa}{\pa\l_1}(v_{\rm Mod},\mathcal L^*\Phi_M)_{|(v,\rm Mod)  =   (Q,1,0)}=-(Q,y\cdot\nabla \mathcal L^*\Phi_M)=(\mathcal L\Lambda Q,\Phi_M)=0,
\eee
\bee
\frac{\pa}{\pa b}(v_{\rm Mod},\Phi_M)_{|(v,\rm Mod)  =   (Q,1,0)}& = &-(T_1,\Phi_M)=0,
\eee
\bee
\frac{\pa}{\pa b}(v_{\rm Mod},\mathcal L^*\Phi_M)_{|(v,\rm Mod)  =   (Q,1,0)}& = &-(T_1,\mathcal L^*\Phi_M)=-(\Lambda Q,\Phi_M)\\
& = & -24\log M+O(1)
\eee
 The collection of above computations yields the Jacobian:
 $$\left|\frac{\pa F}{\pa Mod} \right|_{|(v,\rm Mod)  =   (Q,1,0)}=\left(24 \log M+O(1)\right)^2>0$$ for $M$ large enough, and the claim follows from the implicit function theorem.
 \end{proof}

 We are now in position to describe explicitly the set of initial data leading to the conclusions of Theorem \ref{thmmain}. 

\begin{definition}[Description of the open set of stable chemotactic blow up]
\label{defo}
Fix $M>M^*$ large enough. For $0<\alpha^*<\alpha^*(M)$ small enough, we let $\mathcal O$ be the set of initial data of the form  
$$
 u_0=\frac{1}{\l_0^2}(\tilde{Q}_{b_0}+\e_0)\left(\frac{x}{\l_0}\right), \ \ \e_0\in \mathcal E_{\rm rad},
 $$
 where $\e_0$ satisfies the orthogonality conditions 
 \be
 \label{initicondit}
 (\e_0,\Phi_M)=(\e_0,\mathcal L^*\Phi_M),
 \ee
  and with the following a priori bounds:\\
 {\em (i) Positivity}: $u_0>0$.\\
  {\em (ii) Small super critical mass}: 
 \be
 \label{intiialmass}
 \int Q<\int u_0<\int Q+\alpha^*.
 \ee
 {\em (iii) Positivity and smallness of $b_0$}: 
 \be
 \label{positivitybzero} 0<b_0<\delta(\alpha^*).
 \ee
 {\em (iv) Initial smallness}:
 \be
 \label{initialsmalneesb}
 \|\e_0\|_{H^2_Q}<b_0^{10}.
 \ee
\end{definition}

Observe using Lemma \ref{lemmamodutlation} that $\matchal O$ is open in $\mathcal E_{\rm rad}\cap \{u>0\}$. It is also non empty and for example $u_0=\tilde{Q}_{b_0}\in \matchal O$ for $b_0>0$ small enough.\footnote{Observe that the positivity $\tilde{Q}_{b_0}>0$ for $b_0>0$ directly follows from \fref{developpementT1}, \fref{esttwo}.} The positivity \fref{positivitybzero} of $b_0$ coupled with the smallness \fref{initialsmalneesb} means that we are imposing a deformation of $Q$ which pushes towards concentration in the regime formally predicted by the construction of $\qbt$. The challenge is to show dynamically that the solution keeps a similar shape in time meaning that $\e(t)$ remains negligible in suitable norms with respect to $b(t)$.

 
 \subsection{Setting up the bootstrap}
 
 
 Let $u_0\in \mathcal O$ and let $u\in \mathcal C([0,T),\mathcal E)$ be the corresponding strong solution to \fref{kps}, then using the implicit function theorem, the solution admits on some small time interval  $[0,T^*)$ a unique decomposition 
 \be
 \label{decompe}
 u(t,x)=\frac{1}{\l^2(t)}(\tilde{Q}_{b(t)}+\e)\left(t,\frac{x}{\l(t)}\right)
 \ee where $\e(t)$ satisfies the orthogonality conditions
 \be
 \label{orthoe}
 (\e(t),\Phi_M)=(\e(t),\mathcal L^*\Phi_M)=0,
 \ee
 and the geometrical parameters satisfy from standard argument\footnote{see for example \cite{MMkdv} for a complete proof in a related setting.} $(\lambda(t))\in \mathcal C([0,T^*),\Bbb R^*_+\times \R^*_+)$.  Using the initial smallness assumption, we may assume on $[0,T^*)$ the bootstrap bounds:
 \begin{itemize}
 \item Positivity and smallness of $b$: 
 \be
 \label{positiv}
0<b(t)<10b_0.
 \ee
 \item $L^1$ bound:
\be
\label{bootsmalllone}
\int |\e(t)|<(\delta^*)^{\frac 14},
\ee
\item Control of $\e$ in smoother norms: let $$\e_2=\mathcal L\e,$$ then:
\be
\label{bootsmallh2q}
\|\e_2(t)\|^2_{L^2_Q}\leq K^*\frac{b^3(t)}{|\log b(t)|^2},
\ee
\end{itemize}

Here $(\delta^*,K^*)$ denote respectively small and large enough universal constants with $$\delta^*=\delta(\alpha^*)\to 0\ \ \mbox{as} \ \ \alpha^*\to 0,$$ and $$K^*=K^*(M)\gg 1.$$ 

We claim that this regime is trapped:

\begin{proposition}[Bootstrap]
\label{propboot}
Assume that $K^*$ in \fref{bootsmallh2q}  has been chosen large enough, then $\forall t\in [0,T^*)$:
\be
 \label{positivboot}
0<b(t)<2b_0,
 \ee
\be
\label{bootsmallloneboot}
\int |\e(t)|<\frac12(\delta^*)^{\frac 14},
\ee
\be
\label{bootsmallh2qboot}
\|\e_2(t)\|^2_{L^2_Q}\leq \frac{K^*}{2}\frac{b^3(t)}{|\log b(t)|^2}.
\ee
\end{proposition}

In other words, the regime is trapped and $$T^*=T.$$ The rest of this section together with section \ref{htwowbound} are devoted to the exhibition of the main dynamical arguments at the heart of the proof of Proposition \ref{propboot} which is completed in section \ref{proofthmmai}, and easily implies Theorem \ref{thmmain}. All along the proof, we shall make an extensive use of the interpolation bounds of Appendix \ref{appendixc} which hold in the bootstrap regime.\\
Our aim for the rest of this section is to derive the modulation equations driving the parameters $(b,\l)$.

 
 \subsection{Setting up the equations}
 
 
 Let  the space time renormalization\footnote{we will show that the rescaled time is global $s(t)\to +\infty$ as $t\to T$.} 
 \be
 \label{defst}
 s(t)=\int_0^t\frac{d\tau}{\l^2(\tau)}, \ \ y=\frac{x}{\lambda(t)}.
 \ee
 We will use the notation:
 \be
 \label{resacling}
 f_{\lambda}(t,x)=\frac{1}{\l^2}f(s,y), \ \ r=|y|
 \ee
 which leads to: 
\be\label{neneoneonoeo}
\pa_tf_\l=\frac{1}{\l^2}\left[\pa_sf-\lsl \Lambda f\right]_\l.
\ee
 Let the renormalized density:
 $$u(t,x)=v_{\lambda}(s,y)$$ then $$\pa_sv-\lsl\Lambda v=\nabla\cdot(\nabla v+v\nabla \phi_v).$$ We expand using the geometrical decomposition \fref{decompe}: $$v(s,y)=\tilde{Q}_{b(s)}(y)+\e(s,y)$$ and obtain the renormalized linearized equations:
 \be
 \label{eqe}
 \pa_s\e-\lsl\Lambda \e=\mathcal L\e+\mathcal F,\ \ \mathcal F=F+\Mod+E
 \ee
 with
 \be
 \label{deff}
 F=\tilde{\Psi}_b+\Theta_b(\e)+N(\e),
 \ee
 \be
 \label{defpsibt}
 \tilde{\Psi}_b=  \nabla\cdot(\nabla \tilde{Q}_b+\qbt\nabla \phi_{\qbt})-b\Lambda \qbt-c_b b^2 \chi_{\frac{B_0}{4}}T_1,
 \ee
 \be
 \label{defss}
 E=-b_s\frac{\partial \qbt}{\pa b}-c_b b^2 \chi_{\frac{B_0}{4}}T_1,
 \ee
 \be
 \label{defmod}
\Mod=\left(\lsl+b\right)\Lambda \qbt,
\ee 
\be
\label{tbe}
\Theta_b(\e)=\nabla\cdot\left[\e(\nabla \phi_{\qbt}-\nabla \phi_Q)+(\qbt-Q)\nabla \phi_\e\right],
 \ee
 \be
 \label{nonlinearterm}
 N(\e)=\nabla\cdot(\e\nabla \phi_\e).
 \ee
 Equivalently, let us consider the fluctuation of density in original variables $$w(t,x)=\e_\l(s,y),$$ then 
 \be
 \label{defflambda}
 \pa_tw=\mathcal L_{\lambda}w+\frac{1}{\l^2}\mathcal F_\l
 \ee
 with $$\mathcal L_\l w=\nabla \cdot(Q_\l\nabla \mathcal M_\l w), \ \ \mathcal M_\l w=\frac{w}{Q_\l}+\phi_w.$$


\subsection{$L^1$ bound}


We start with closing the $L^1$ bound on the radiation \fref{bootsmallloneboot} which is a consequence of the conservation of mass. This bound is very weak but gives a control of the solution far out which allows us to derive bounds with logarithmic losses, see Appendix \ref{appendixc}.

\begin{lemma}[$L^1$  bound]
\label{htwoqmontonbis}
There holds:
\be
\label{lonebound}
\int|\e|<\frac 12(\delta^*)^{\frac 14}.
\ee
\end{lemma}

\begin{proof}[Proof of Lemma \ref{htwoqmontonbis}]
We introduce the decomposition $$u=\frac{1}{\l^2}(Q+\et)\left(t,\frac{x}{\l(t)}\right)\ \ \mbox{i.e.}\ \ \et=\e+\qbt-Q$$ and define:
\be
\label{decompfoejejo}
\et=\et_<+\et_>, \ \ \et_{<}=\et{\bf 1}_{|\et|<Q},\ \  \et_>=\et{\bf 1}_{\et>Q}.
\ee 
The bootstrap bounds \fref{linftyboundbis}, \fref{positiv} with \fref{esttoneproploc}, \fref{esttwoloc} imply: 
\be
\label{etlinftiysmall}
\|\et\|_{L^{\infty}}\lesssim \|\e\|_{L^{\infty}}+|b|\lesssim \delta(\alpha^*).
\ee Pick a small constant $\eta^*\ll 1$, then $|\et_<|>\eta^*Q$ implies $\delta(\alpha^*)\gtrsim \eta^* Q$ and thus $|x|\geq r(\alpha^*)$ with $r(\alpha^*)\to +\infty$ as $\alpha^*\to 0$ from which: 
 \bee
 \int |\et_<| & \leq & \int  |\et_<|{\bf 1}_{\eta^*Q<|\et|< Q}+\int |\et_<|{\bf 1}_{|\et_<|\leq \eta^*Q}\lesssim \int_{|x|\geq r(\alpha^*)}Q+\eta^*\int Q\\
& \lesssim& \delta(\alpha^*)+\eta^*\lesssim \delta(\alpha^*).
\eee
 We now write down the conservation of mass which from \fref{intiialmass} implies $$\int \et<\alpha^*$$ from which using by definition $\et_>>0$: 
 \be
 \label{cneiocneneo}
 \int |\et| \lesssim \delta(\alpha^*).
 \ee
 Together with the bootstrap bound \fref{positiv} and \fref{esttoneproploc}, \fref{esttwoloc}, this yields:
 $$\int |\e|\lesssim \int |\et|+\sqrt{b}\lesssim \delta(\alpha^*)+\sqrt{\delta^*}<\frac 12(\delta^*)^{\frac 14}
$$ 
and concludes the proof of Lemma \ref{htwoqmontonbis}.
\end{proof}


\subsection{Rough modulation equations}


We compute the modulation equations for $(b,\l)$ which follow from our choice of orthogonality conditions \fref{orthoe}. 

\begin{remark}
\label{remarkcm} Here and in the sequel, we note $C(M)$ a generic large constant which depends on $M$ {\it independently} of the bootstrap constant $K^*(M)$ in \fref{bootsmallh2q}, provided $\alpha^*<\alpha^*(M)$ has been chosen
small enough in Definition \ref{defo}.
\end{remark}

\begin{lemma}[Rough control of the modulation parameters]
\label{lemmeparam}
There holds the bounds: 
\be
\label{estlambda}
|\lsl+b|\lesssim C(M)\frac{b^2}{|\log b|},
\ee
\be
\label{estb}
|b_s|\lesssim b^{\frac 32}.
\ee
\end{lemma}

\begin{remark} Note that \fref{estlambda}, \fref{estb} imply the bootstrap bound: 
\be
\label{cnkonenoeoeoi3poi}
\left|\lsl+b\right|+|b_s|\lesssim b^{\frac32}.
\ee
\end{remark}

\begin{proof}[Proof of Lemma \ref{lemmeparam}]
We make an implicit use of the bootstrap bounds of Proposition \ref{bootbound}. Let $$V=\left|\lsl+b\right|+|b_s|,$$ we claim the bounds:
\be
\label{estiboundone}
\left|\lsl+b\right|\lesssim b^{\frac 34}|V|+C(M)\frac{b^2}{|\log b|},
\ee
\be
\label{estiboundtwo}
\left|b_s\right|\lesssim b^{\frac 32}+C(M)b|V|.
\ee
Summing these bounds yields $V\lesssim b^{\frac32}$ which reinjected into \fref{estiboundone}, \fref{estiboundtwo} yields \fref{estlambda}, \fref{estb}.\\
{\it Proof of \fref{estiboundone}}: We project \fref{eqe} onto $\Phi_M$ which is compactly supported in $r\leq 2M$ and compute from \fref{orthoe}, \fref{estphim}, \fref{roughboundltaowloc}, \fref{bootsmallh2q}:
\bea
\label{vneionioehoehoe}
\nonumber \left|(\Mod +E,\Phi_M)\right|&=&\left|-\lsl(\Lambda \e,\Phi_M)-(F,\Phi_M)\right|\\
\nonumber & \lesssim & C(M)|b|(|V|+\|\e_2\|_{L^2_Q})+C(M)\frac{b^2}{|\log b|}\\
& \lesssim &  b^{\frac 34}|V|+C(M)\frac{b^2}{|\log b|}.
\eea
We estimate in brute force using \fref{estphim}, \fref{orthophim}, \fref{defcb}:
$$|(E,\Phi_M)|\lesssim C(M)\left[b|V|+|b_s||(T_1,\Phi_M)|+\frac{b^2}{|\log b|}\right]\lesssim  b^{\frac 34}|V|+C(M)\frac{b^2}{|\log b|}$$
and we further compute using \fref{orthophim}:
\bee
(\Mod, \Phi_M) & = & \left(\lsl+b\right)(\Lambda Q,\Phi_M)+O(C(M)|b||V|)\\
& = & -\left[(32\pi)\log M+O(1)\right]\left(\lsl+b\right)+O(C(M)|b||V|).
\eee Injecting these estimates into \fref{vneionioehoehoe} yields \fref{estiboundone}.\\

{\it Proof of \fref{estiboundtwo}}: We let $\mathcal L$ go through \fref{eqe} and project onto $\Phi_M$. We compute from \fref{orthoe}, \fref{novneioheonoen}, \fref{roughboundltaowloc}, \fref{estfonamentalebus} with the degeneracy \fref{estvooee}:
\bea
\label{cneknvnonveoneon}
\nonumber &&\left|(\mathcal L\Mod+\mathcal LE ,\Phi_M)\right|=\left| -(\mathcal L\e_2,\Phi_M) -\lsl(\mathcal L\Lambda \e,\Phi_M)-(\mathcal LF,\Phi_M)\right|\\
\nonumber & = & \left|-(\e_2,\matchal L^*\Phi_M) -\lsl(\Lambda \e,\mathcal L^*\Phi_M)-(F,\mathcal L^*\Phi_M)\right|\\
\nonumber & \lesssim  & \frac{\|\e_2\|_{L^2_Q}}{M}+C(M)\left[(b+V)\int \frac{|\Lambda \e|+|\nabla \e|}{1+r^2}+\int \frac{|N(\e)|+|\Theta_b(\e)|}{1+r^2}+\frac{b^2}{|\log b|}\right]\\
& \lesssim & b^{\frac 32}+C(M)b|V|
\eea
where the $\Psit_b$ term is estimated from \fref{novneioheonoen}, \fref{roughboundltaowloc}, and the estimates of $(N(\e), \Theta_b(\e))$ terms easily follow from the bootstrap bounds of Appendix \ref{appendixc}. We then estimate:
\bee
(\mathcal LE,\Phi_M)& = & b_s(\mathcal L\tt_1,\Phi_M)+O\left(C(M)|b_s|b+C(M)\frac{b^2}{|\log b|}\right)\\
& =& b_s\left[-(32\pi) \log M+O_{M\to +\infty}(1)\right]+O\left(C(M)b|V|+C(M)\frac{b^2}{|\log b|}\right).
\eee
Moreover, using $\mathcal L\Lambda Q=0$:
\bee
|(\mathcal L\Mod,\Phi_M)|& \lesssim &\left|\lsl+b\right||\Lambda (\qbt-Q),\mathcal L^*\Phi_M)|\lesssim  C(M)|V|\int \frac{1}{1+r^2}\left[b\frac{{\bf 1}_{r\leq 2B_1}}{1+r^2}\right]\\
& \lesssim & C(M)b|V|.
\eee
Injecting the collection of above bounds into \fref{cneknvnonveoneon} yields \fref{estiboundtwo}.
\end{proof}



\section{$H^2_Q$ bound}
\label{htwowbound}


After the construction of the approximate blow up profile and the derivation of the modulation equations of Lemma \ref{lemmasharpmod}, we now turn to the second main input of our analysis which is the control of the $H^2_Q$ norm through the derivation of a suitable Lyapounov functional.  Our strategy is to implement an energy method for $\e_2=\mathcal L\e$ which breaks the $L^1$ scaling invariance of the problem.\\
Here we are facing a technical problem which is that the equation \fref{eqe} is forced on the RHS by the $b_s$ term in \fref{defss}  which satisfies pointwise the weak estimate \fref{estb} only, and this is not good enough to close \fref{bootsmallh2qboot}. This is a consequence of the fact that the elements of the kernel of $\mathcal L^*$ grow too fast in space, and therefore $b_s$ remains of order $1$ in $\e$. We however claim that better estimates hold {\it up to an oscillation in time}, but this requires the introduction of a second decomposition of the flow, see for example \cite{MMkdv} for a related difficulty.


\subsection{The radiation term}


We use the bootstrap bound \fref{bootsmallh2q} to introduce a second decomposition of the flow which will {\it lift} the parameter $b$. We recall the notation \fref{defphimzero}.

\begin{lemma}[Bound on the lifted parameter]
\label{lemmaradiaiton}
There exists a unique decomposition 
\be
\label{defehat}
\tilde{Q}_{b}+\e=\tilde{Q}_{\hat{b}}+\hat{\e}
\ee with $\hb\lesssim b\lesssim \bh$ and $\hat{\e}$ satisfying the orthogonality condition
\be
\label{orthoehat}
(\hat{\e},\matchal L^*\Phi_{0,\hat{B}_0})=0, \ \ \hat{B}_0=\frac{1}{\sqrt{\hat{b}}}.
\ee
Moreover, there holds the bound:
\be
\label{boundbootbbhat}
|b-\hat{b}|\lesssim \frac{b}{|\log b|}.
\ee
\end{lemma}

\begin{remark} It is essential for the rest of the analysis that the constant in \fref{boundbootbbhat} is independent of $K^*(M)$.
\end{remark}

\begin{proof}[Proof of Lemma \ref{lemmaradiaiton}] The claim follows from the implicit function theorem again. For $|b-\bh|\lesssim b$, let the map $F(v,\bh)=(v-\tilde{Q}_{\hat{b}},\matchal L^*\Phi_{0,\hat{B}_0})$, then $F(\qbt,b)=0$ and $$\frac{\partial F}{\partial \hb}_{|(\hb=b,v=\qbt)}=-(\frac{\partial \qbt}{\partial b},\matchal L^*\Phi_{0,B_0}).$$ We compute using the bounds of Proposition \ref{proploc}:
$$\frac{\partial \qbt}{\partial b}=\chi_{B_1}T_1+O\left(\frac{1}{1+r^2}{\bf 1}_{B_1\leq r\leq 2B_1}+b\left[\frac{1+|\log( r\sqrt{b})|}{|\log b|}{\bf 1}_{r\leq 6B_0}+\frac{1}{b^2r^4}{\bf 1}_{6B_0\leq r\leq 2B_1}\right]\right),$$ and thus using \fref{orthophim}, \fref{bjebbeibei}:
\bee
&&(\frac{\partial \qbt}{\partial b},\matchal L^*\Phi_{0,B_0})\\
&= & (\chi_{B_1}T_1,\matchal L^*\Phi_{0,B_0})+bO\left(\int \left[\frac{1 + |\log (r\sqrt b)|}{|\log b|}{\bf 1}_{r\leq 6B_0}+\frac{1}{b^2r^4|\log b|}{\bf 1}_{ 6B_0\leq r\leq 2B_1}\right]\right)\\
& = & -32\pi \log B_0+O(1) <0,
\eee
and the existence of the decomposition \fref{defehat} with $b\sim \hat{b}$ now follows form the implicit function theorem.\\
The size of the deformation can be measured thanks to \fref{bootsmallh2q}. Indeed, we first claim that for all $\frac{1}{\sqrt{b}}\lesssim B\lesssim \frac{1}{\sqrt{b}}$:
\be
\label{fluxoneproof}
(\tilde{Q}_{b}-Q,\mathcal L^*\Phi_{0,B})=-32\pi b\log B+O(b).
\ee
Indeed, we estimate from \fref{bjebbeibei}:
\bee
&& (\qbt-Q,\mathcal L^*\Phi_{0,B})\\
 & = & b(\mathcal LT_1,\Phi_{0,B})+O\left[\int_{r\leq 2B} |b(\tt_1-T_1)+b^2\tt_2|\right].
\eee
By assumption, \be
\label{estnieneoifn}
\left|\frac{\log B}{|\log b|}-\frac 12\right|=\left|\frac{\log (B\sqrt{b})}{|\log b|}\right|\lesssim \frac{1}{|\log b|},
\ee
and we therefore estimate from \fref{orthophim}:
$$b(\mathcal LT_1,\Phi_{0,B})=b(\Lambda Q,\Phi_B)=b\left[-(32\pi) \log B+O(1)\right].$$
We estimate from \fref{esttoneproploc}, \fref{esttwoloc}:
$$
\int_{r\leq 2B} b|\tt_1-T_1| \lesssim \int _{B_1\leq r\leq 2B_1}\frac{b}{r^2}\lesssim b,
$$
 \bee
\int_{r\leq 2B} b^2|\tt_2|& \lesssim &b^2\int _{ r\leq 2B_1}\left[r^2{\bf 1}_{r\leq 1}+\frac{1 + |\log (r\sqrt b)|}{|\log b|}{\bf 1}_{1\leq r\leq 6B_0}+\frac{1}{b^2r^4|\log b|}{\bf 1}_{ 6B_0\leq r\leq 2B_1}\right]\\
&\lesssim & b
\eee
and \fref{fluxoneproof} follows.\\
We now claim the bound
\be
\label{diffbbhatone}
|b-\hat{b}|\lesssim\frac{|b|}{|\log b|}+\frac{\|\e_2\|_{L^2_Q}}{\sqrt{b}|\log b|}
 \ee
 which together with the bootstrap bound \fref{bootsmallh2q} implies \fref{boundbootbbhat}. Indeed, we take the scalar product of \fref{defardiation} with $\mathcal L^*\Phi_{0,\hat{B}_0}$ and conclude from the choice of orthogonality condition \fref{orthoehat} and \fref{fluxoneproof}, \fref{estnieneoifn}:
 $$|\log b||b-\hat{b}|\lesssim |b|+|\hat{b}|+
|(\e_2,\Phi_{0,\hat{B}_0})|\lesssim |b|+\frac{\|\e_2\|_{L^2_Q}}{\sqrt{b}}$$ and \fref{diffbbhatone} is proved. This concludes the proof of Lemma \ref{lemmaradiaiton}.
\end{proof}

We now introduce the radiation term 
\be
\label{defardiation}
\zeta=\hat{\e}-\e=\tilde{Q}_{b}-\tilde{Q}_{\hat{b}}
\ee
which we split in two parts: 
\be
\label{defhiefheoh}
\zeta=\zeta_{\rm big}+\zeta_{\rm sm},
\ee
$$\zb=(b-\hat{b})\chi_{B_1}T_1,$$
$$
\zs= \bh(\chi_{B_1}-\chi_{\hat{B}_1})T_1+\int_{\hat{b}}^b\left[2b\tt_2+b^2\pa_b\tt_2\right]db
$$
We let $$\zeta_2=\matchal L\zeta.$$

\begin{lemma}[Bounds on the radiation]
\label{lemmaradiation}
There holds the pointwise bounds:
\be
\label{estimatonebis} 
\int |r^i\nabla^i\zb|^2+\int \frac{|\nabla \phi_{\zb}|^2}{1+r^2}\lesssim \frac{b^2}{|\log b|^2},
\ee
\be
\label{estimatone} 
\int |r^i\nabla^i\zs|^2+\int \frac{|\nabla \phi_{\zs}|^2}{1+r^2}+\int \frac{|\mathcal L\zs|^2}{Q}+\int \frac{|\mathcal L\Lambda \zs|^2}{Q}\lesssim \frac{
b^3}{|\log b|^4},
\ee
\be
\label{estimatonebisbis} 
\int |\mathcal L\zs|^2\lesssim \frac{b^4}{|\log b|^2},
\ee
\be
\label{nneonevonoer}
\int \frac{|\mathcal L\zb|^2}{Q}+\int \frac{|\mathcal L\Lambda \zb|^2}{Q}\lesssim \frac{b^2}{|\log b|^2},
\ee
\be
\label{degenrateoneone}
\forall v\in L^2_Q\ \ \mbox{with}\ \  \int v=0, \ \ |(\mathcal M\zeta_2,v)|\lesssim \frac{b^{\frac 32}}{|\log b|}\|v\|_{L^2_Q},
\ee
\be
\label{noenoeno}
\left|(\matchal M\zeta_2,\zeta_2)\right|\lesssim \frac{b^3}{|\log b|^2},
\ee
\be
\label{degenrateonebis}
\int Q|\nabla \mathcal M\zeta_2|^2\lesssim \frac{b^4}{|\log b|^2}.
\ee
\end{lemma}

\begin{remark} This lemma quantifies the fact that the radiation is a priori large, in particular larger than $\e$ when comparing \fref{estimatonebis} with \fref{bootsmallh2q}, but because the leading order term in $\zeta$ is supported along $T_1$, some degenerate norms like \fref{degenrateoneone}, \fref{degenrateonebis} are better behaved, and this will be essential to close the $H^2_Q$ bound for $\e$, see step 6 of the proof of Proposition \ref{htwoqmonton}.
\end{remark}

\begin{proof}[Proof of Lemma \ref{lemmaradiation}]

{\it Proof of \fref{estimatonebis}}: We have the pointwise bound using \fref{esttoneproploc}:
$$|r^i\nabla^i\zeta_{big}|\lesssim |\bh-b|\frac{{\bf 1}_{r\leq 2B_1}}{1+r^2},$$
and hence from \fref{boundbootbbhat}:
\bee
\int |r^i\nabla^i\zb
|^2\lesssim |\bh-b|^2\lesssim \frac{b^2}{|\log b|^2}.
\eee
We estimate the Poisson field using the radial representation:
\be
\label{nceoneoneveklnvel}
|\nabla \phi_{\zb}|\lesssim |\bh-b|\frac{1+|\log r|}{1+r}
\ee from which:
$$\int \frac{|\nabla \phi_{\zb}|^2}{1+r^2}\lesssim |\bh-b|^2\lesssim \frac{b^2}{|\log b|^2}.$$
{\it Proof of \fref{estimatone}, \fref{estimatonebisbis}}: We have using \fref{boundbootbbhat}, \fref{esttwo} the pointwise bound:
\bea
\label{estammleon}
|r^i\nabla^i\zs| & \lesssim & |b-\hat{b}|\frac{{\bf 1}_{\frac{B_1}{2}\leq r\leq 3B_1}}{r^2}\\
\nonumber & + & b|b-\bh|\left[r^2{\bf 1}_{r\le1} +\frac{1 + |\log (r\sqrt b)|}{|\log b|}{\bf 1}_{1\leq r\leq 6B_0}+ \frac{1}{b^2r^4|\log b|}{\bf 1}_{ 6B_0\leq r\leq 2B_1}\right]
\eea
from which using \fref{boundbootbbhat}:
\bee
\nonumber \int |r^i\nabla^i\zs|^2& \lesssim & \frac{b^2}{|\log b|^2}\int \frac{{\bf 1}_{\frac{B_1}{2}\leq r\leq 3B_1}}{r^4}\\
\nonumber & + & \frac{b^4}{|\log b|^2}\int\left[{\bf 1}_{r\le1} +\frac{1 + |\log (r\sqrt b)|^2}{|\log b|^2}{\bf 1}_{1\leq r\leq 6B_0}+ \frac{1}{b^4r^8|\log b|^2}{\bf 1}_{ 6B_0\leq r\leq 2B_1}\right]\\
& \lesssim & \frac{b^3}{|\log b|^4}.
\eee
We recall from \fref{estmtwo}, \fref{esttwoloc} the bounds:
$$|\phi'_{\tt_2}|=|\frac{m_2}{r}|\lesssim  r^5{\bf 1}_{r\le1} +\frac{1 + r|\log (r\sqrt b)|}{|\log b|}{\bf 1}_{1\leq r\leq 6B_0}+\frac{1}{rb|\log b|}{\bf 1}_{r\geq 2B_0} \ \ \mbox{for}\ \ r\leq B_1,$$
$$ |\phi'_{\tt_2}|\lesssim \frac{1}{r}\int_0^{2B_1} r|\tt_2|dr\lesssim\frac{1}{rb|\log b|} \ \ \mbox{for}\ \ r\geq 2B_1.$$
We then estimate the Poisson field using the radial representation:
\be
\label{estammleonbis}
|\nabla \phi_{\zs}|\lesssim \frac{|b-\bh|}{r}{\bf 1}_{r\geq \frac{B_1}{2}}+b|b-\bh|\left[ r^5{\bf 1}_{r\le1} +\frac{1 + r|\log (r\sqrt b)|}{|\log b|}{\bf 1}_{1\leq r\leq 6B_0}+\frac{1}{rb|\log b|}{\bf 1}_{r\geq 2B_0}\right]
\ee
 and hence:
\bee
\int \frac{|\nabla\phi_{\zs}|^2}{1+r^2} & \lesssim & \frac{b^2}{|\log b|^2}\int_{r\geq \frac{B_1}{2}}\frac{1}{r^4}\\
& + &   \frac{b^4}{|\log b|^2}\int\frac{1}{1+r^2}\left[r^5{\bf 1}_{r\le1} +\frac{1 + r^2|\log (r\sqrt b)|^2}{|\log b|^2}{\bf 1}_{1\leq r\leq 6B_0}+\frac{1}{r^2b^2|\log b|^2}{\bf 1}_{r\geq 2B_0}\right]\\
& \lesssim & \frac{b^3}{|\log b|^4}.
\eee
 We estimate from \fref{formuleL}, \fref{estammleon}, \fref{estammleonbis}:
\bea
\label{nkeonneonve}
 |r^i\nabla^i\mathcal L\zs|& \lesssim &  |b-\hat{b}|\frac{{\bf 1}_{\frac{B_1}{2}\leq r\leq 3B_1}}{r^4}\\
\nonumber & + & b|b-\bh|\left[{\bf 1}_{r\le1} +\frac{1 + |\log (r\sqrt b)|}{r^2|\log b|}{\bf 1}_{1\leq r\leq 6B_0}+ \frac{1}{b^2r^6|\log b|}{\bf 1}_{ 6B_0\leq r\leq 2B_1}\right]\\
\nonumber& + &  \frac{b-\bh}{r^6}{\bf 1}_{r\geq \frac{B_1}{2}}\\
\nonumber& + & \frac{b|b-\bh|}{1+r^5}\left[r^5{\bf 1}_{r\le1} +\frac{1 + r|\log (r\sqrt b)|}{|\log b|}{\bf 1}_{1\leq r\leq 6B_0}+\frac{1}{rb|\log b|}{\bf 1}_{r\geq 2B_0}\right]\\
\nonumber& \lesssim &  |b-\hat{b}|\frac{{\bf 1}_{\frac{B_1}{2}\leq r\leq 3B_1}}{r^4}\\
\nonumber & + & b|b-\bh|\left[{\bf 1}_{r\le1} +\frac{1 + |\log (r\sqrt b)|}{r^2|\log b|}{\bf 1}_{1\leq r\leq 6B_0}+ \frac{1}{b^2r^6|\log b|}{\bf 1}_{ r\geq 2B_0}\right]
\eea
and thus from \fref{boundbootbbhat}:
\bee
\int |\mathcal L\zs|^2&\lesssim & \int \left[ |b-\hat{b}|^2\frac{{\bf 1}_{\frac{B_1}{2}\leq r\leq 3B_1}}{r^8}\right]\\
& + & \frac{b^4}{|\log b|^2}\left[{\bf 1}_{r\le1} +\frac{1 + |\log (r\sqrt b)|^2}{(1+r^4)|\log b|^2}{\bf 1}_{1\leq r\leq 6B_0}+ \frac{1}{b^4r^{12}|\log b|^2}{\bf 1}_{ 6B_0\leq r\leq 2B_1}\right]\\
& \lesssim & \frac{b^4}{|\log b|^2},
\eee
\bee
\int \frac{|\mathcal L\zs|^2}{Q} &\lesssim & \int \left[ |b-\hat{b}|^2\frac{{\bf 1}_{\frac{B_1}{2}\leq r\leq 3B_1}}{r^4}\right]\\
& + & \frac{b^4}{|\log b|^2}\left[{\bf 1}_{r\le1} +\frac{1 + |\log (r\sqrt b)|^2}{|\log b|^2}{\bf 1}_{1\leq r\leq 6B_0}+ \frac{1}{b^4r^8|\log b|^2}{\bf 1}_{ 6B_0\leq r\leq 2B_1}\right]\\
& \lesssim & \frac{b^3}{|\log b|^4}.
\eee
Finally,  pick a well localized function $f$ and let $f_\l(x)=\l^2f(\lambda x)$, then from direct check:
$$\matchal Lf_\l=\l^2\left[\Delta f+2Q_{\frac1\l}f+\nabla \phi_{Q_{\frac 1\l}}\cdot \nabla f+\nabla Q_{\frac 1\l}\cdot\nabla \phi_f\right]_{\l}$$ and differentiating this relation at $\l=1$ yields: 
\be
\label{cneocneonoeeooe}
\mathcal L\Lambda f=2\mathcal Lf+\Lambda(\mathcal L f)-2(\Lambda Q) f-\nabla\phi_{\Lambda Q}\cdot\nabla f-\nabla (\Lambda Q)\cdot\nabla \phi_f.
\ee Applying this to $\zs$ and using \fref{nkeonneonve}, 
\fref{estammleon}, \fref{estammleonbis} yields the pointwise bound:
\bee
|\mathcal L(\Lambda \zs)|&\lesssim & |b-\hat{b}|\frac{{\bf 1}_{\frac{B_1}{2}\leq r\leq 3B_1}}{r^4}\\
\nonumber & + & b|b-\bh|\left[{\bf 1}_{r\le1} +\frac{1 + |\log (r\sqrt b)|}{r^2|\log b|}{\bf 1}_{1\leq r\leq 6B_0}+ \frac{1}{b^2r^6|\log b|}{\bf 1}_{ r\geq 2B_0}\right]
\eee
and hence $$\int \frac{|\mathcal L\Lambda \zs|^2}{Q} \lesssim \frac{b^3}{|\log b|^2}.$$
{\it Proof of \fref{nneonevonoer}, \fref{degenrateoneone}}: We estimate from \fref{estimatone}, \fref{mcontiniuos}:
$$|(\mathcal M\mathcal L\zs,v)|\lesssim \|\mathcal L\zs\|_{L^2_Q}\|v\|_{L^2_Q}\lesssim \frac{b^{\frac 32}}{|\log b|}\|v\|_{L^2_Q}.$$
We now compute in brute force using \fref{formuleL}, \fref{eqtonekvne}:
\bea
\label{ccjooeccjwo}
\nonumber \mathcal L(\chi_{B_1}T_1) & = & \chi_{B_1}[\Delta T_1+2QT_1+\nabla \phi_{Q}\cdot \nabla T_1]+\nabla Q\cdot\nabla \phi_{\chi_{B_1}T_1}+O\left(\frac{{\bf 1}_{B_1\leq r\leq 2B_1}}{1+r^4}\right)\\
& = & \chi_{B_1}\Lambda Q+\nabla Q\cdot\left[\chi_{B_1}\nabla \phi_{T_1}-\nabla \phi_{\chi_{B_1}T_1}\right]=\Lambda Q+O\left(\frac{{\bf 1}_{r\geq B_1}}{1+r^4}\right)
\eea
and thus:
\be
\label{zetabig}
\mathcal L\zb =  (b-\bh)\Lambda Q+\xi,
\ee
\be
\label{defxi}
\xi=|b-\bh|O\left(\frac{{\bf 1}_{r\geq B_1}}{1+r^4}\right).
\ee
This yields the rough bound using \fref{boundbootbbhat}:
$$\int \frac{|\mathcal L\zb|^2}{Q}\lesssim |b-\hat{b}|^2\lesssim \frac{b^2}{|\log b|^2}.$$ The second estimate in \fref{nneonevonoer} follows similarily using \fref{cneocneonoeeooe}.
We further estimate from \fref{zetabig}, \fref{defxi}, \fref{relationsm}, \fref{efadjointness}, \fref{estvooee}, \fref{boundbootbbhat}:
$$
|(\matchal M\mathcal L\zb,v)|  \lesssim  |(-2,v)|+  \|v\|_{L^2_Q}\left[|b-\bt|^2\int \frac{{\bf 1}_{r\geq B_1}}{Q(1+r^8)}\right]^{\frac 12}\lesssim  \frac{b^{\frac 32}}{|\log b|} \|v\|_{L^2_Q}
$$
and \fref{degenrateoneone} is proved.\\
{\it Proof of \fref{noenoeno}}: From \fref{zetabig}, \fref{degenrateoneone}, \fref{estimatone}:
\bee
|(\mathcal M\zeta_2,\zeta_2)|& = & |(\matchal M\zeta_2,\xi+\matchal L\zs)|\lesssim \frac{b^{\frac 32}}{|\log b|}\left[\|\xi\|_{L^2_Q}+\|\mathcal L\zs\|_{L^2_Q}\right]\\
& \lesssim &  \frac{b^{\frac 32}}{|\log b|}\left[\left(\frac{b^2}{|\log b|^2}\int_{r\geq B_1}\frac{1}{1+r^4}\right)^{\frac 12}+\frac{b^{\frac 32}}{|\log b|}\right]\lesssim \frac{b^3}{|\log b|^2}.
\eee

{\it Proof of \fref{degenrateonebis}}: We estimate from \fref{nkeonneonve} and Hardy Littlewood Sobolev:
\bee
\int Q|\nabla \mathcal M\mathcal L\zs|^2& \lesssim & \int \frac{1}{Q}\left\{ |b-\hat{b}|^2\frac{{\bf 1}_{\frac{B_1}{2}\leq r\leq 3B_1}}{r^{10}}\right.\\
\nonumber & + &\left . \frac{b^4}{|\log b|^2}\left[{\bf 1}_{r\le1} +\frac{1 + |\log (r\sqrt b)|^2}{r^6|\log b|^2}{\bf 1}_{1\leq r\leq 6B_0}+ \frac{1}{b^4r^{14}|\log b|^2}{\bf 1}_{ 6B_0\leq r\leq 2B_1}\right]\right\}\\
& + &\|\nabla \phi_{\mathcal L\zs}\|_{L^4}^2\\
& \lesssim & \frac{b^4}{|\log b|^2}+\|\mathcal L\zs\|_{L^{\frac 43}}^2\lesssim \frac{b^4}{|\log b|^2}.
\eee
We estimate from \fref{zetabig}, \fref{defxi}, \fref{relationsm}, \fref{boundbootbbhat} and Hardy Littlewood Sobolev:
\bee
\int Q|\nabla \mathcal M\mathcal L\zb|^2 &\lesssim & \int \frac{1}{Q}\left||b-\bh|\frac{{\bf 1}_{r\geq B_1}}{1+r^5}\right|^2+ \int Q|\nabla \phi_{\xi}|^2\\
& \lesssim &\frac{|b-\bh|^2}{B_1^4}+\|\xi\|_{L^{\frac 43}}^2 \lesssim \frac{|b-\bh|^2}{B_1^4}\\
& \lesssim & \frac{b^4}{|\log b|^2}.
\eee
This concludes the proof of Lemma \ref{lemmaradiation}.
\end{proof}


\subsection{Sharp modulation equation for $b$}


We now compute the sharp modulation equation for $b$ using the lifted parameters $\hat{b}$.

\begin{lemma}[Sharp modulation equations for $b$]
\label{lemmasharpmod}
 There holds the differential inequations:
\be
\label{poitnzeiboud}
\left|\hat{b}_s+\frac{2b^2}{|\log b|}\right|\lesssim C(M)\frac{\sqrt{b}}{|\log b|}\|\e_2\|_{L^2_Q}+\frac{b^2}{|\log b|^2}.
\ee
\end{lemma}

\begin{remark} Note that \fref{poitnzeiboud} coupled with the bootstrap bound \fref{bootsmallh2q} ensures the pointwise control: 
\be
\label{pointziender}
|\hat{b}_s|\lesssim \frac{b^2}{|\log b|}.
\ee
\end{remark}

\begin{proof}[Proof of Lemma \ref{lemmasharpmod}]

{\it Proof of \fref{poitnzeiboud}}: We take the scalar product of \fref{eqe} with $\mathcal L^*\Phi_{0,\hat{B}_0}$:
\bea
\label{eiqoheoeh}
\nonumber \frac{d}{ds}\left\{(\qbt-Q+\e,\mathcal L^*\Phi_{0,\hat{B}_0})\right\}& = & (\qbt-Q+\e,\pa_s\mathcal L^*\Phi_{0,\hat{B}_0})+(\e_2,\mathcal L^*\Phi_{0,\hat{B}_0})\\
 & - & c_bb^2(\chi_{\frac{B_0}{4}}T_1,\mathcal L^*\Phi_{0,\hat{B}_0})\\
\nonumber & + & \lsl(\Lambda \e,\mathcal L^*\Phi_{0,B})+(F,\mathcal L^*\Phi_{0,\hat{B}_0})\\
\nonumber & + & \left(\lsl+b\right)(\Lambda \qbt,\mathcal L^*\Phi_{0,\hat{B}_0})
\eea
and estimate all terms in \fref{eiqoheoeh}. For the boundary term in time, we have from \fref{defardiation}, \fref{orthoehat}:
$$(\qbt-Q+\e,\mathcal L^*\Phi_{0,\hat{B}_0})=(\tilde{Q}_{\hat{b}}-Q+\hat{\e},\mathcal L^*\Phi_{0,\hat{B}_0})=(\tilde{Q}_{\hat{b}}-Q,\mathcal L^*\Phi_{0,\hat{B}_0}).$$  We then compute:
$$
\frac{d}{ds}(\tilde{Q}_{\hat{b}}-Q,\mathcal L^*\Phi_{0,\hat{B}_0})=  (\pa_s\tilde{Q}_{\hat{b}},\mathcal L^*\Phi_{0,\hat{B}_0})+(\tilde{Q}_{\hat{b}}-Q,\mathcal L^*\pa_s\Phi_{0,\hat{B}_0}).
$$
The leading order term is estimated using \fref{bjebbeibei}:
\bee
 &&(\pa_s\tilde{Q}_{\hat{b}},\mathcal L^*\Phi_{0,\hat{B}_0}) = \hat{b}_s(\tt_1+\hat{b}\pa_{b}\tt_1+2\hat{b}\tt_2+\hat{b}^2\pa_b\tt_2,\mathcal L^*\Phi_{0,\hat{B}_0})\\
 & = & \hat{b}_s(T_1,\mathcal L^*\Phi_{0,\hat{B}_0})+|\hat{b}_s|O\left(\int |\tt_1-T_1|+b|\pa_b\tt_1|+b|\tt_2|+b^2|\pa_b\tt_2|\right).
 \eee
 We have from \fref{orthophim}:
 $$ \hat{b}_s(T_1,\mathcal L^*\Phi_{0,\hat{B}_0})=\hat{b}_s (\Lambda Q,\Phi_{0,\hat{B}_0})=\hat{b}_s\left[-(32\pi) \log \hat{B}_0+O(1)\right],
$$ and we estimate the error terms un brute force using the estimates of Proposition \ref{proploc}:
 \bee
&& \int |\tt_1-T_1|+\bh|\pa_b\tt_1|+\hat{b}|\tt_2|+\hat{b}^2|\pa_b\tt_2| \lesssim  \int_{B_1\leq r\leq 2B_1}\frac{1}{r^2}\\
& + & b\int \left[\frac{1 + |\log (r\sqrt b)|}{|\log b|}{\bf 1}_{1\leq r\leq 6B_0}+\frac{1}{b^2r^4|\log b|}{\bf 1}_{ 6B_0\leq r\leq 2B_1}\right]\\
& \lesssim & 1.
 \eee
 We also estimate using the definition \fref{defphimzero} and Proposition \ref{proploc}:
 \be
 \label{nkeonvorononore}
|(\tilde{Q}_{\hat{b}}-Q,\pa_s\mathcal L^*\Phi_{0,\hat{B}_0})|\lesssim |\hat{b}_s|\int_{\hat{B}_0\leq r\leq 2\hat{B}_0}r^2|\mathcal L(\tilde{Q}_{\hat{b}}-Q)|\lesssim |\hat{b}_s|\int_{\frac{B_0}{4}\leq r\leq 4B_0}\frac{r^2}{r^4}\lesssim |\hat{b}_s|
.\ee
We therefore have obtained the preliminary bound using also \fref{boundbootbbhat}, \fref{estnieneoifn}:
$$
\frac{d}{ds}\left\{(\qbt-Q+\e,\mathcal L^*\Phi_{0,\hat{B}_0})\right\}=\hat{b}_s\left[-(32\pi) \log B_0+O(1)\right].$$
We now estimate the RHS of \fref{eiqoheoeh}. The main linear term is treated using \fref{estvooee}, \fref{estfonamentalebus}  which yield a sharp bound:
$$\left|(\e_2,\mathcal L^*\Phi_{0,\hat{B}_0})\right|\lesssim \int_{r\geq \hat{B}_0}|\e_2|\lesssim \sqrt{b}\|\e_2\|_{L^2_Q}.$$ We estimate like for the proof of \fref{nkeonvorononore}:
\bee
|(\qbt-Q,\pa_s\mathcal L^*\Phi_{0,\hat{B}_0})|\lesssim |\hat{b}_s|.
\eee
The second linear term is estimated in brute force:
\bee
|(\e,\pa_s\mathcal L^*\Phi_{0,\hat{B}_0})|&= & \left|(\e_2,\frac{(\hat{B}_0)_s}{\hat{B}_0}(r\chi')(\frac{r}{\hat{B}_0})r^2)\right|\lesssim \frac{|\hat{b}_s|}{b}\|\e_2\|_{L^2_Q}\left(\int_{B_0\leq r\leq 2B_0} \frac{r^4}{r^4}\right)^{\frac 12}\\
& \lesssim & \frac{|\hat{b}_s|}{b\sqrt{b}}\|\e_2\|_{L^2_Q}\lesssim |\hat{b}_s|
\eee
where we used the bootstrap bound \fref{bootsmallh2q} in the last step.
The leading order {\it flux} term is computed from \fref{orthophim}, \fref{bjebbeibei}, \fref{defcb}, \fref{estnieneoifn}:
\bee
&&- c_bb^2(\chi_{\frac{B_0}{4}}T_1,\mathcal L^*\Phi_{0,\hat{B}_0})=  -c_bb^2\left[(\matchal LT_1,\Phi_{0,\hat{B}_0})+((\chi_{\frac{B_0}{4}}-1)T_1,\mathcal L^*\Phi_{0,\hat{B}_0})\right]\\
& = & \frac{2b^2}{|\log b|}\left[1+O\left(\frac{1}{|\log b|}\right)\right]\left[32\pi\log B_0+O(1)+\int_{r\leq 2\hat{B}_0}\frac{{\bf 1}_{r\geq \frac{B_0}{4}}}{1+r^2}\right]\\
& = & 32\pi b^2\left[1+O\left(\frac{1}{|\log b|}\right)\right].
\eee
Next, from \fref{bjebbeibei}, \fref{bootsmallh2q}:
\bee
\left|\lsl(\Lambda \e,\mathcal L^*\Phi_{0,\hat{B}_0})\right|& \lesssim & b\int_{r\leq 2B_0}(|\e|+|y\cdot\nabla \e|)\lesssim  \sqrt{b}(\|\e\|_{L^2}+\|y\cdot\nabla \e\|_{L^2})\\
& \lesssim & C(M)\sqrt{b}\|\e_2\|_{L^2_Q}. 
\eee

We now treat the $F$ terms. The $\Psit_b$ term is estimated from the degenerate flux estimate \fref{degeencos}:
$$|(\mathcal L\Psit_b,\Phi_{0,\hat{B}_0})|\lesssim \frac{b^2}{|\log b|}.$$
To treat the small linear term $\Theta_b(\e)$, we first extract from Proposition \ref{propblowup} the pointwise bounds:
 \be
 \label{opojinonvo}
 |r^i\pa_r^i(\qbt-Q)|\lesssim b\frac{{\bf 1}_{r\leq 2B_1}}{1+r^2}, \ \ | r^i\pa_r^i\nabla \phi_{\qbt-Q}|\lesssim b\frac{1+|\log r|}{1+r}
 \ee
which imply the pointwise bounds:
 \be
 \label{poietpoeutone}
 |\Theta_b(\e)|\lesssim b\left[\frac{1+|\log r|}{1+r} |\nabla \e|+\frac{|\e|}{1+r^2}+\frac{|\nabla \phi_\e|}{1+r^3}\right],
 \ee
 \be
 \label{poietpoeuttwo}
 |\nabla\Theta_b(\e)|\lesssim  b\left[\frac{1+|\log r|}{1+r} |\nabla^2 \e|+\frac{1+|\log r|}{1+r^2}|\nabla \e|+\frac{|\nabla^2 \phi_\e|+|\e|}{1+r^3}+\frac{|\nabla \phi_\e|}{1+r^4}\right].
 \ee
We also observe by definition the cancellation $(\Theta_b(\e),1)=0$ and therefore estimate from \fref{estfonamentalebus} and Proposition \ref{bootbound}:
\bee
|(\Theta_b(\e),\mathcal L^*\Phi_{0,\hat{B}_0})|&\lesssim & \int_{r\geq \hat{B}_0}|\Theta_b(\e)|\lesssim b\int \left[\frac{(1+|\log r|)|\nabla \e|}{1+r}+\frac{|\e|}{1+r^2}+\frac{|\nabla \phi_\e|}{1+r^3}\right]\\
& \lesssim& bC(M)\|\e_2\|_{L^2_Q}\lesssim \frac{b^2}{|\log b|}
\eee
and similarly for the nonlinear term using \fref{interpolationfield}:
\bee
|(N(\e),\mathcal L^*\Phi_{0,\hat{B}_0})|&\lesssim& \int|N(\e)|\lesssim \int \e^2+|\nabla \phi_\e||\nabla \e|\lesssim |\log b|^C\|\e_2\|_{L^2_Q}^2\lesssim \frac{b^2}{|\log b|}.
\eee
We now treat the terms induced by modulation. We use the $L^1$ critical relation $\int \Lambda Q=\int \Lambda \qbt=0$, \fref{estfonamentalebus}, \fref{cnkonenoeoeoi3poi} and \fref{esttoneproploc}, \fref{esttwoloc} to estimate:
\bee
&&\left|\lsl+b\right||(\mathcal L\Lambda \qbt,\Phi_{0,\hat{B}_0})\lesssim  b^{\frac 32}\int_{r\geq \hat{B}_0} |\Lambda (\qbt-Q)|\\
& \lesssim &b^{\frac 32}\int_{\frac{B_0}{2}\leq r\leq 2B_1} \frac{b}{1+r^2}\lesssim \frac{b^2}{|\log b|}.
\eee
Injecting the collection of above bounds into \fref{eiqoheoeh} yields:
$$\hat{b}_s\left[-(32\pi) \log B_0+O(1)\right]=32\pi b^2+O\left(\frac{b^2}{|\log b|}+C(M)\sqrt{b}\|\e_2\|_{L^2_Q}+|\hb_s|\right).$$
Hence using \fref{estnieneoifn}:
\bee
\left|\tilde{b}_s+\frac{b^2}{|\log B_0|}\right|&\lesssim &\frac{1}{|\log B_0|}\left[\frac{b^2}{|\log b|}+C(M)\sqrt{b}\|\e_2\|_{L^2_Q}+|\hb_s|\right]\\
& \lesssim &
\frac{b^2}{|\log b|^2}+C(M)\frac{\sqrt{b}}{|\log b|}\|\e_2\|_{L^2_Q}+\frac{|\hb_s|}{|\log b|}
\eee
and \fref{poitnzeiboud} now follows using the bootstrap bound \fref{bootsmallh2q}.
\end{proof}


\subsection{The monotonicity formula}


We now turn to the heart of the energy method and claim the following monotonicity formula at the $H^2_Q$ level:

\begin{proposition}[$H^2_Q$ monotonicity]
\label{htwoqmonton}
There holds:
\bea
\label{eetfonafmetea}
&&\frac{d}{dt}\left\{\frac{(\mathcal M \e_2,\e_2)}{\l^2}+O\left(\frac{C(M)b^{\frac32}}{|\log b|}\|\e_2\|_{L^2_Q}+\frac{C(M)b^3}{\l^2|\log b|^2}\right)\right\}\\
\nonumber &  \leq &  \frac{bC(M)}{\l^4}\left[\frac{b^{\frac32}}{|\log b|}\|\e_2\|_{L^2_Q}+\frac{b^3}{|\log b|^2}\right].
\eea

\end{proposition}

{\it Proof of Proposition \ref{htwoqmonton}}\\

{\bf step 1} Introduction of the radiation.\\

The $\e$ equation \fref{eqe} is forced on the RHS by the $b_s$ term which satisfies pointwise the weak estimate \fref{estb} only which is not good enough to close \fref{htwoqmonton}. This is a consequence of the fact that the elements of the kernel of $\mathcal L^*$ grow too fast in space. We therefore aim at relying on the sharp bounds of Lemma \ref{lemmasharpmod} and on the lifted parameter $\bh$. Recall \fref{defardiation} which yields $$\pa_s(\qbt+\e)=\pa_s(\tilde{Q}_{\bh}+\hat{\e}), \ \ \e=\eh-\zeta,$$ and rewrite the $\e$ equation \fref{eqe} as:

\be
\label{eqehat}
\pa_s\eh-\lsl\Lambda \eh=\mathcal L\eh+\hat{\mathcal F}, \ \ \hat{\mathcal F}=F+\widehat{\Mod}+E_1+E_2,
\ee
with $F$ given by \fref{deff} and:
\be
\label{dehhate}
E_1=-\mathcal L\z-\lsl\Lambda \z,\\
\ee
\be
\label{etwobis}
 E_2=-c_b b^2 \chi_{\frac{B_0}{4}}T_1,
\ee
\be
\label{mohat}
\widehat{\Mod}=-\pa_s\tilde{Q}_{\bh}+\left(\lsl+b\right)\tilde{Q}_{b}.
\ee

{\bf step 2} Renormalized variables and energy identity.  Let $$ w=\e_{\l},  \ \ \hat{w}=\eh_{\l}.$$ We introduce the differential operator of order one:
\be
\label{defal}
A_\l w=Q_\l\nabla \mathcal M_\l w=\nabla w+\nabla \phi_{Q_\l}w+Q_\l\nabla \phi_w
\ee and the suitable derivatives\footnote{which correspond to the factorization of  the linearized operator $\mathcal L$ in two operators of order one: $\mathcal L=\nabla \cdot(A)$.}
\be
\label{denkofnekoneovnk}
\left\{\begin{array}{ll}w_1=A_\lambda w, \ \ \wh_1=A_\lambda \wh\\ \e_1=A\e, \ \ \eh_1=A\eh
\end{array}\right ., \  \ \left\{\begin{array}{ll}w_2=\nabla \cdot w_1=\mathcal L_\l w, \ \ \wh_2=\nabla \cdot \wh_1=\mathcal L_\l\wh\\ \e_2=\nabla \cdot \e=\mathcal L\e,\ \ \eh_2=\nabla \cdot \eh=\mathcal L\eh\end{array}\right.
\ee
and similarly for the radiation: $$\zeta_1=A\zeta, \ \ \zeta_2=\mathcal L\zeta.$$
We compute from \fref{eqehat} the equation satisfied by $\hw, \hw_1,\hw_2$:
\be
\label{eqw}
\pa_t\hw=\hw_2+\frac{1}{\l^2}\mathcal \hF_\l,
\ee
\be
\label{eqwone}
\pa_t\hw_1=A_\l \hw_2+\frac{1}{\l^2}A_\l\mathcal \hF_\l+[\pa_t,A_\l]\hw,
\ee
\be
\label{eqwtwo}
\pa_t\hw_2=\matchal L_\l \hw_2+\frac{1}{\l^2}\mathcal L_\l \mathcal \hF_\l+\nabla \cdot([\pa_t,A_\l]\hw).
\ee
We compute the energy identity for $\hw_2$, and will all along the proof repeatedly use \fref{neneoneonoeo}:
\bea
\label{cenoneoneo}
\nonumber &&\frac{1}{2}\frac{d}{dt}(\mathcal M_\l \hw_2,\hw_2) =    (\pa_t\hw_2,\mathcal M_\l \hw_2)-\int\frac{\pa_tQ_\l}{2Q_\l^2}\hw_2^2\\
\nonumber & =&  (\mathcal L_\l \hw_2+\frac{\mathcal L_\l\mathcal \hF_\l}{\l^2}+\nabla \cdot \left([\pa_t,A_\l]\hw\right),\mathcal M_\l \hw_2)-\int\frac{\pa_tQ_\l}{2Q_\l^2}\hw_2^2\\
\nonumber & = & -\int Q_\l|\nabla \mathcal M_\l \hw_2|^2+(\hw_2,\frac{1}{\l^2}\mathcal M_\l \mathcal L_\l F_\l)-\int  [\pa_t,A_\l]\hw\cdot \nabla (\mathcal M_\l \hw_2)\\
\nonumber & + &\int \frac{\hw_2^2}{\l^2Q_\l^2}\left[ \left(\lsl+\bh\right)(\Lambda Q)_{\l}\right]-  \int\frac{\bh(\Lambda Q)_\l}{2\l^2Q_\l^2}\hw_2^2.
\eea

{\bf step 3} Repulsivity and sharp control of tails at infinity. We need to treat the last quadratic term in \fref{cenoneoneo} which has no definite sign and critical size for the analysis, and corresponds to an interaction of radiation with the soliton profile. For this, we will use dissipation and repulsivity properties of $\mathcal L$ measured in suitable weighted norms. We will in particular use the fundamental degeneracy \fref{philambdaq}\footnote{which is  a consequence of the $L^1$ critical identity $\int \Lambda Q=0$.}:
\be
\label{gfijgbeigei}
\frac{\Lambda Q}{Q}+2=-\phi_{\Lambda Q}=O\left(\frac{1}{1+r^2}\right),
\ee
 which will allow us to measure in a sharp way the size of tails at infinity. Our approach is deeply related to the construction of suitable Morawetz type multiplier in the dispersive setting of \cite{RaphRod}, \cite{MRR}.\\
We compute from \fref{eqw}, \fref{eqwtwo}:
\bee
&&\frac{d}{dt}\left\{\int \frac{\bh(\Lambda Q)_{\l}}{\l^2Q_\l^2}\hw_2\hw\right\}=\int \hw_2\hw\frac{d}{dt}\left\{\frac{\bh(\Lambda Q)_{\l}}{\l^2Q_\l^2}\right\}\\
& + & \int \frac{\bh(\Lambda Q)_{\l}}{\l^2Q_\l^2}\hw\left[\mathcal L_\l \hw_2+\frac{1}{\l^2}\mathcal L_\l \mathcal \hF_\l+\nabla \cdot([\pa_t,A_\l]\hw)\right]+  \int \frac{\bh(\Lambda Q)_{\l}}{\l^2Q_\l^2}\hw_2\left[\hw_2+\frac{1}{\l^2}\mathcal\hF_\l\right].
\eee
Let us compute the leading order terms in this identity. First the potential term:
\bee
&&\frac{d}{dt}\left\{\frac{\bh(\Lambda Q)_{\l}}{\l^2Q_\l^2}\right\} =  \frac{d}{dt}\left\{\bh\left(\frac{\Lambda Q}{Q^2}\right)\left(\frac{x}{\l(t)}\right)\right\}\\
& = & \frac{1}{\l^2}\left[\bh_s\frac{\Lambda Q}{Q^2}-\bh\lsl y\cdot\nabla\left(\frac{\Lambda Q}{Q^2}\right)\right](y)\\
& = & \frac{1}{\l^2Q}\left[8\bh\lsl-2\bh_s+O\left(\frac{|\bh_s|+\bh|\lsl|}{1+r^2}\right)\right].
\eee
Next, we observe the integration by parts:
\bee
&&\int \frac{\bh(\Lambda Q)_{\l}}{\l^2Q_\l^2}\hw\mathcal L_\l \hw_2 = -2\bh\int \frac{\hw\mathcal L_\l \hw_2}{\l^2Q_\l}+\int \frac{\bh(2Q+\Lambda Q)_{\l}}{\l^2Q_\l^2}\hw\mathcal L_\l \hw_2\\
& = & \frac{2\bh}{\l^2}\int Q_\l \nabla(\mathcal M_\l \hw_2)\cdot\nabla\left[\frac{\hw}{Q_\l}+\phi_{\hw}\right]-\frac{\bh}{\l^2}\int Q_\l \nabla (\mathcal M_\l \hw_2)\cdot\nabla \left[2\phi_{\hw}+\frac{(2Q+\Lambda Q)_{\l}}{\l^2Q_\l^2}\hw\right]\\
& = & -\frac{2\bh}{\l^2}(\mathcal M_\l \hw_2,\hw_2)-\frac{\bh}{\l^2}\int Q_\l \nabla (\mathcal M_\l \hw_2)\cdot\nabla \left[2\phi_{\hw}+\frac{(2Q+\Lambda Q)_{\l}}{\l^2Q_\l^2}\hw\right].
\eee
We therefore obtain:
\bea
\label{cnioeeoheo}
\nonumber &&\frac{d}{dt}\left\{\int \frac{\bh(\Lambda Q)_{\l}}{\l^2Q_\l^2}\hw_2\hw\right\}= \int\frac{\bh(\Lambda Q)_\l}{\l^2Q_\l^2}\hw_2^2-\frac{2\bh}{\l^2}(\mathcal M_\l \hw_2,\hw_2)\\
\nonumber & + & \int\frac{\eh\eh_2}{\l^6Q}\left[8\bh\lsl-2\bh_s\right]\\
\nonumber & -&\frac{\bh}{\l^2}\int Q_\l \nabla (\mathcal M_\l \hw_2)\cdot\nabla \left[2\phi_{\hw}+\frac{(2Q+\Lambda Q)_{\l}}{\l^2Q_\l^2}\hw\right]\\
\nonumber & + &   \int \frac{\bh(\Lambda Q)_{\l}}{\l^2Q_\l^2}\hw\left[\frac{1}{\l^2}\mathcal L_\l \mathcal \hF_\l+\nabla \cdot([\pa_t,A_\l]\hw)\right]+  \int \frac{\bh(\Lambda Q)_{\l}}{\l^2Q_\l^2}\hw_2\left[\frac{1}{\l^2}\mathcal \hF_\l\right]\\
& + & O\left(\int \left[|\bh_s|+\bh|\lsl|\right]\frac{\eh_2\eh}{\l^6(1+r^2)Q}\right)
\eea
We now compute the energy identity for $\hw_1$ using \fref{eqwone}:
\bee
\frac{1}{2}\frac{d}{dt}\left\{\int \frac{\bh\hw_1^2}{\l^2Q_\l}\right\}=\frac12\int \hw_1^2\frac{d}{dt}\left\{\frac{\bh}{\l^2Q_\l}\right\}+\bh\int\frac{\hw_1}{\l^2Q_\l}\cdot\left[ A_\l \hw_2+\frac{1}{\l^2}A_\l\mathcal \hF_\l+[\pa_t,A_\l]\hw\right].
\eee
We compute the dominant terms:
\bee
& & \frac{d}{dt}\left\{\frac{\bh}{\l^2Q_\l}\right\}=\frac{d}{dt}\left\{\frac{\bh}{Q\left(\frac{x}{\l(t)}\right)}\right\}=\frac{1}{\l^2}\left[\frac{\bh_s}{Q}+\bh\frac{\lsl y\cdot\nabla Q}{Q^2}\right]\\
& = & \frac{1}{\l^2Q}\left[\bh_s-4\bh\lsl+O\left(\frac{|\bh_s|+\bh|\lsl|}{1+r^2}\right)\right],\\
\eee
and integrate by parts using \fref{defal} and the definition of $\hw_i$:
$$
\bh\int\frac{\hw_1}{\l^2Q_\l}\cdot A_\l \hw_2=-\frac{\bh}{\l^2}(\mathcal M_\l \hw_2,\nabla \cdot \hw_1)=-\frac{\bh}{\l^2}(\mathcal M_\l \hw_2,\hw_2).$$ This leads to the second identity\footnote{This identity reflects the fact that the dual Hamiltonian $\nabla \cdot A_\l$ driving the $w_1$ equation \fref{eqwone} is repulsive, and this is reminiscent from geometric equations, see \cite{RodSter}, \cite{RaphRod}.}:
\bee
& & \frac{1}{2}\frac{d}{dt}\left\{\int \frac{\bh \hw_1^2}{\l^2Q_\l}\right\}=-\frac{\bh}{\l^2}(\mathcal M\hw_2,\hw_2)+  \int \frac{\eh_1^2}{2\l^6Q}\left[\bh_s-4\bh\lsl\right]\\
& + & \bh\int\frac{\hw_1}{\l^2Q_\l}\left[ \frac{1}{\l^2}A_\l\mathcal \hF_\l+[\pa_t,A_\l]\hw\right]+  O\left(\int \frac{\eh_1^2}{\l^6Q}\left[\frac{|\bh_s|+\bh|\lsl|}{1+r^2}\right]\right).
\eee
We combine this with \fref{cnioeeoheo} to derive:
\bee
&&\frac{d}{dt}\left\{\int \frac{\bh(\Lambda Q)_{\l}}{\l^2Q_\l^2}\hw_2\hw-2\int \frac{\bh\hw_1^2}{\l^2Q_\l}\right\}=\int\frac{\bh(\Lambda Q)_\l}{\l^2Q_\l^2}\hw_2^2+ \frac{2\bh}{\l^2}(\mathcal M_\l \hw_2,\hw_2)\\
& + &  \int\frac{\eh \eh_2+\eh_1^2}{\l^6Q}\left[8\bh\lsl-2\bh_s\right]\\
\nonumber & -&\frac{\bh}{\l^2}\int Q_\l \nabla (\mathcal M_\l \hw_2)\cdot\nabla \left[2\phi_{\hw}+\frac{(2Q+\Lambda Q)_{\l}}{\l^2Q_\l^2}\hw\right]\\
\nonumber & + &   \int \frac{\bh(\Lambda Q)_{\l}}{\l^2Q_\l^2}\left[\frac{\hw}{\l^2}\mathcal L_\l \mathcal \hF_\l+\hw_2\frac{1}{\l^2}\mathcal \hF_\l\right]-4\bh\int\frac{\hw_1}{\l^2Q_\l}\cdot\left[ \frac{1}{\l^2}A_\l\mathcal \hF_\l\right]\\
& + & \int \frac{\bh(\Lambda Q)_{\l}}{\l^2Q_\l^2}\hw\nabla \cdot([\pa_t,A_\l]\hw)-  4\bh\int\frac{w_1}{\l^2Q_\l}\cdot[\pa_t,A_\l]\hw\\
& + & O\left(\int \frac{|\eh_2\eh|+|\eh_1|^2}{\l^6Q(1+r^2)}(|\bh_s|+\bh|\lsl|)\right).
\eee
We now observe by integration by parts:
$$\int \frac{\eh_2\eh}{Q}=-\int Q\nabla\mathcal M\eh\cdot\nabla\left(\frac{\eh}{Q}\right)=-\int\frac{\eh_1^2}{Q}+\int Q\nabla \mathcal M \eh\cdot\nabla \phi_{\eh}$$ and hence the cancellation\footnote{which corresponds to an improved decay at infinity, each term in the left hand side of \fref{nioennene} being too slowly decaying at infinity to be treated separately.} :
\be
\label{nioennene}
\int \frac{\eh_2\eh+\eh_1^2}{Q}=\int \eh_1\cdot\nabla \phi_{\eh}.
\ee
Similarily:
\bee
&&-\int \frac{2}Q(\eh\mathcal L\mathcal \hF+\eh_2\mathcal \hF)-4\int\frac{1}{Q} \eh_1\cdot A\mathcal \hF\\
& =&  2\int Q\nabla \matchal M\mathcal \hF\cdot\nabla (\mathcal M\eh-\phi_{\eh})+2\int Q\nabla \matchal M\mathcal \eh\cdot\nabla (\mathcal M\mathcal \hF-\phi_\mathcal \hF)-4\int Q \nabla \mathcal M\eh\cdot\nabla \mathcal M\mathcal \hF\\
& = & -2\int\eh_1\cdot\nabla \phi_\mathcal \hF-2\int Q\nabla \matchal M\mathcal \hF\cdot\nabla\phi_{\eh}
\eee
and thus:
\bee
& & \int \frac{\bh(\Lambda Q)_{\l}}{\l^2Q_\l^2}\left[\frac{\hw}{\l^2}\mathcal L_\l \mathcal \hF_\l+\hw_2\frac{1}{\l^2}\mathcal \hF_\l\right]-4\bh\int\frac{\hw_1}{\l^2Q_\l}\cdot\left[ \frac{1}{\l^2}A_\l\mathcal \hF_\l\right]\\
& = & \frac{\bh}{\l^6}\left\{\int\frac{\Lambda Q+2Q}{Q^2}(\eh\mathcal L\mathcal \hF+\e_2\mathcal \hF)-2\int\eh_1\cdot\nabla \phi_{\mathcal \hF}-2\int\nabla \phi_{\eh}\cdot Q\nabla \mathcal M\mathcal \hF\right\}.
\eee
We finally obtain the identity:
\bee
&&\frac{d}{dt}\left\{\int \frac{\bh(\Lambda Q+2Q)}{\l^4Q^2}\eh_2\eh-2\int \frac{\bh\eh_1\cdot\nabla\phi_{\eh}}{\l^4}\right\}=\int\frac{\bh(\Lambda Q)_\l}{\l^2Q_\l^2}\hw_2^2 + \frac{2\bh}{\l^2}(\mathcal M_\l \hw_2,\hw_2)\\
& + &  \int\frac{\eh_1\cdot\nabla \phi_{\eh}}{\l^6}\left[8\bh\lsl-2\bh_s\right]\\
 \nonumber & -&\frac{\bh}{\l^2}\int Q_\l \nabla (\mathcal M_\l \hw_2)\cdot\nabla \left[2\phi_{\hw}+\frac{(2Q+\Lambda Q)_{\l}}{\l^2Q_\l^2}\hw\right]\\
& + & \int \frac{\bh(\Lambda Q)_{\l}}{\l^2Q_\l^2}\hw\nabla \cdot([\pa_t,A_\l]\hw)-  4\bh\int\frac{\hw_1}{\l^2Q_\l}\cdot[\pa_t,A_\l]\hw\\
& + & \frac{\bh}{\l^6}\left\{\int\frac{\Lambda Q+2Q}{Q^2}(\eh\mathcal L\mathcal \hF+\e_2\mathcal \hF)-2\int \eh_1\cdot\nabla \phi_\mathcal \hF-2\int\nabla \phi_{\eh}\cdot Q\nabla \mathcal M\mathcal \hF\right\}\\
& + & O\left(\int \frac{|\eh_2\eh|+|\eh_1|^2}{\l^6Q(1+r^2)}(|\bh_s|+\bh|\lsl|)\right).
\eee
We combine this with the energy identity \fref{cenoneoneo} and obtain the modified and manageable energy identity:
\bea
\label{noumberneergyidentity}
 &&\frac 12\frac{d}{dt}\left\{(\mathcal M_\l \hw_2,\hw_2) +\int \frac{\bh(\Lambda Q+2Q)}{\l^4Q^2}\eh_2\eh-2\int \frac{\bh\eh_1\cdot\nabla\phi_{\eh}}{\l^4}\right\}\\
\nonumber& = &  -\int Q_\l|\nabla \mathcal M_\l \hw_2|^2+\frac{\bh}{\l^2}(\mathcal M_\l \hw_2,\hw_2)+\frac{1}{\l^6}(\eh_2,\mathcal M\mathcal L\mathcal \hF)\\
\nonumber & + &  \int\frac{\eh_1\cdot\nabla \phi_{\eh}}{2\l^6}\left[8\bh\lsl-2\bh_s\right]\\ 
\nonumber& + & \frac{\bh}{2\l^6}\left\{\int\frac{\Lambda Q+2Q}{Q^2}(\eh\mathcal L\mathcal\hF+\eh_2\mathcal \hF)-2\int \eh_1\cdot\nabla \phi_{\mathcal \hF}-2\int\nabla \phi_{\eh}\cdot Q\nabla \mathcal M\mathcal \hF\right\}\\
\nonumber& - & \frac{\bh}{2\l^2}\int Q_\l \nabla (\mathcal M_\l \hw_2)\cdot\nabla \left[2\phi_{\hw}+\frac{(2Q+\Lambda Q)_{\l}}{\l^2Q_\l^2}\hw\right]\\
\nonumber& - & \int  [\pa_t,A_\l]\wh\cdot \nabla (\mathcal M_\l \wh_2)+\int \frac{b(\Lambda Q)_{\l}}{2\l^2Q_\l^2}\wh\nabla \cdot([\pa_t,A_\l]\wh)-  2b\int\frac{\wh_1}{\l^2Q_\l}\cdot[\pa_t,A_\l]\wh\\
\nonumber & + &\int \frac{\hw_2^2}{\l^2Q_\l^2}\left[ \left(\lsl+\bh\right)(\Lambda Q)_{\l}\right]\\
\nonumber & + &   O\left(\int \frac{|\eh_2\e|+|\eh_1|^2}{\l^6Q(1+r^2)}(|\bh_s|+\bh|\lsl|)\right).
\eea
We now aim at estimating all terms in \fref{noumberneergyidentity}. We will implicitly use the bounds of Lemma \ref{bootbound}. We will make an essential use of the improved decay \fref{gfijgbeigei}.\\

{\bf step 4}. Boundary term in time.  We estimate using Lemma \ref{lemmaradiation}:
 \bee
\left| \int \frac{\bh(\Lambda Q+2Q)}{\l^4Q^2}\eh_2\eh\right|&\lesssim& \frac{b}{\l^4}\int \frac{|\eh_2\eh|}{(1+r^2)Q}\lesssim \frac{b}{\l^4}\left(\int \frac{\eh_2^2}{Q}\right)^{\frac 12}\left(\int \eh^2\right)^{\frac 12}\\
& \lesssim &\frac{C(M)b}{\l^4}\left[\|\e_2\|_{L^2_Q}^2+\|\z_2\|_{L^2_Q}^2+\|\zeta\|_{L^2}^2+\|\e\|_{L^2}^2\right]\\
&\lesssim & \frac{C(M)b^3}{|\log b|^2}
\eee
and using \fref{controleun}, \fref{interpolationfield}:
\bee
&&\left|\frac1{\l^4}\int b\eh_1\cdot\nabla \phi_{\eh}\right|  \lesssim  \frac{b}{\l^4}\left(\int(1+r^2)|\eh_1|^2\right)^{\frac12}\left(\int \frac{|\nabla\phi_{\eh}|^2|}{1+r^2}\right)^{\frac 12}\\
& \lesssim & \frac{b}{\l^4}\left[\int(1+r^2)|\e_1|^2+\int \frac{|\nabla\phi_{\e}|^2}{1+r^2}+\int(1+r^2)|\nabla \zeta|^2+\int |\zeta|^2+\int \frac{|\nabla\phi_{\zeta}|^2}{1+r^2}\right]\\
& \lesssim & \frac{b}{\l^4}\left[|\log b|^C\left(\|\e_2\|_{L^2_Q}^2+b^5\right)+\frac{b^2}{|\log b|^2}\right]\lesssim \frac{C(M)b^3}{|\log b|^2}.
\eee

{\bf step 5}. Lower order quadratic terms. We now estimate the lower order quadratic terms in the RHS of \fref{noumberneergyidentity}. First the terms which are estimated using dissipation, \fref{coercbase} and Lemma \ref{lemmaradiation}:
\bee
&&\left|\frac{b}{2\l^2}\int Q_\l \nabla (\mathcal M_\l \wh_2)\cdot\nabla \left[2\phi_{ \wh}+\frac{(2Q+\Lambda Q)_{\l}}{\l^2Q_\l^2} \wh\right]\right|\\
& \leq& \frac{1}{100}\int Q_\l| \nabla \mathcal M_\l  \wh_2|^2+\frac{Cb^2}{\l^6}\int Q\left[|\nabla\phi_{\eh}|^2+\frac{|\nabla \eh|^2}{(1+r^4)Q^2}+\frac{ \eh^2}{(1+r^6)Q^2}\right]\\
& \leq & \frac{1}{100}\int Q_\l| \nabla \mathcal M_\l  \wh_2|^2+C(M)\frac{b^2}{\l^6}\left[\matchal \|\e_2\|_{L^2_Q}^2+\frac{b^2}{|\log b|^2}\right].
\eee
Next, we observe from \fref{defal} and \fref{gfijgbeigei} the decay estimate:
$$|[\pa_t,A_\l]\wh|\lesssim\frac{b}{\l^2}\left[ \frac{|\eh|}{1+r^3}+\frac{|\nabla \phi_{\eh}|}{1+r^4}\right]$$
which leads using \fref{coercbase}, \fref{controleun}, \fref{interpolationfield} to the bounds:
\bee
&&\left| \int  [\pa_t,A_\l]\wh\cdot \nabla (\mathcal M_\l w_2)\right|\leq \frac{1}{100}\int Q_\l| \nabla \mathcal M_\l \wh_2|^2+C(M)\frac{b^2}{\l^6}\int \frac{1}{Q}\left[\frac{|\eh|^2}{1+r^6}+\frac{|\nabla \phi_{\eh}|^2}{1+r^8}\right]\\
& \leq & \frac{1}{100}\int Q_\l| \nabla \mathcal M_\l \hat{w}_2|^2+C(M)\frac{b^2}{\l^6}\left[\matchal \|\e_2\|_{L^2_Q}^2+\frac{b^2}{|\log b|^2}\right],
\eee
\bee
&&\left|\int \frac{b(\Lambda Q)_{\l}}{2\l^2Q_\l^2}\wh\nabla \cdot([\pa_t,A_\l]\wh)\right|\lesssim \frac{b^2}{\l^6}\int \left[ \frac{|\eh|}{1+r^3}+\frac{|\nabla \phi_{\eh}|}{1+r^4}\right]\left[\frac{|\nabla \eh|}{Q}+\frac{|\eh|}{(1+r)Q}\right]\\
& \lesssim & \frac{b^2}{\l^6}\left[|\log b|^C\matchal \|\e_2\|_{L^2_Q}^2+\int |\zeta|^2+\int(1+r^2)|\nabla \zeta|^2+\int \frac{|\nabla\phi_{\zeta}|^2}{1+r^2}\right]\\
& \lesssim & \frac{b^2}{\l^6}\left[|\log b|^C\matchal \|\e_2\|_{L^2_Q}^2+\frac{b^2}{|\log b|^2}\right],
\eee
\bee
&&\left|b\int\frac{\wh_1}{\l^2Q_\l}\cdot[\pa_t,A_\l]\wh\right|\lesssim \frac{b^2}{\l^6}\int \frac{|\eh_1|}{Q}\left[ \frac{|\eh|}{1+r^3}+\frac{|\nabla \phi_{\eh}|}{1+r^4}\right]\\
& \lesssim &\frac{b^2}{\l^6}\left[ \int (1+r^2)|\e_1|^2+\int |\e|^2+\int\frac{|\nabla \phi_\e|^2}{1+r^2}+\int |\zeta|^2+\int(1+r^2)|\nabla \zeta|^2+\int \frac{|\nabla\phi_{\zeta}|^2}{1+r^2}\right]\\
& \lesssim & \frac{b^2}{\l^6}\left[|\log b|^C\matchal \|\e_2\|_{L^2_Q}^2+\frac{b^2}{|\log b|^2}\right].
\eee
We now estimate using from \fref{estlambda}, \fref{boundbootbbhat}, \fref{pointziender} the rough bounds: $$\left|\lsl\right|\lesssim b, \ \ |\hb_s|\lesssim b^2,$$ and therefore:
\bee
&&\left|\int\frac{\eh_1\cdot\nabla \phi_{\eh}}{2\l^6}\left[8\bh\lsl-2\bh_s\right]\right|\\
& \lesssim & \frac{b^2}{\l^6}\int |\eh_1\cdot\nabla \phi_{\eh}|\lesssim \frac{b^2}{\l^6}\left[|\log b|^C\|\e_2\|_{L^2_Q}^2+b^4+\int(1+r^2)|\nabla \zeta|^2+\int |\zeta|^2+\int \frac{|\nabla\phi_{\zeta}|^2}{1+r^2}\right]\\
& \lesssim & \frac{b^2}{\l^6}\left[|\log b|^C\matchal \|\e_2\|_{L^2_Q}^2+\frac{b^2}{|\log b|^2}\right],
\eee
and for the error term:
\bee
&&\int \frac{|\eh_2\eh|+|\eh_1|^2}{\l^6Q(1+r^2)}(|b_s|+b|\lsl+b|+b^2)\\
& \lesssim& \frac{b^2}{\l^6}\int\left[(1+r^4)\eh_2^2+(1+r^2)|\eh_1|^2+\eh^2\right]\\
& \lesssim & C(M)\frac{b^2}{\l^6}\left[\|\e_2\|_{L^2_Q}^2+\frac{b^2}{|\log b|^2}\right].
\eee
We now estimate using the bound \fref{estlambda}, \fref{boundbootbbhat}:
$$\left|\lsl+\bh\right|\lesssim \frac{b}{|\log b|}.$$  Note that this bound is very bad to treat the radiation term $\zeta$ and we will need some additional cancellation. We recall the decomposition \fref{zetabig}, \fref{defxi} and first estimate using Lemma \ref{lemmaradiation}:
\bee
&&\left|\int \frac{\wh_2^2}{\l^2Q_\l^2}\left[ \left(\lsl+\bh\right)(\Lambda Q)_{\l}\right]\right| \lesssim  \frac{b}{|\log b|\l^6}\left| \int \frac{(\e_2+\mathcal L\zs+\xi+(b-\hb)\Lambda Q)^2}{Q^2}\Lambda Q\right|\\
& \lesssim & \frac{b}{|\log b|\l^6}\left[\|\e_2\|_{L^2_Q}^2+\|\mathcal L\zs\|_{L^2_Q}^2+\|\xi\|_{L^2_Q}^2+\frac{b^2}{|\log b|^2}\left|\int\left(\frac{\Lambda Q}{Q}\right)^2\Lambda Q\right|\right.\\
& + & \left.\frac{b}{|\log b|}\left|\int(\e_2+\mathcal L\zs)\frac{(\Lambda Q)^2}{Q^2}\right|+\frac{b^2}{|\log b|^2}\int_{r\geq B_1}\frac{1}{1+r^4}\right]\\
& \lesssim & \frac{b}{\l^6}\left[\frac{b^{\frac 32}}{|\log b|}\|\e_2\|_{L^2_Q}+\frac{b^3}{|\log b|^2}+\frac{b^2}{|\log b|^2}\left|\int\left(\frac{\Lambda Q}{Q}\right)^2\Lambda Q\right|+\frac{b}{|\log b|^2}\left|\int(\e_2+\mathcal L\zs)\frac{(\Lambda Q)^2}{Q^2}\right|\right].
\eee
We now use the cancellation \fref{gfijgbeigei} to estimate $$\left(\frac{\Lambda Q}{Q}\right)^2=4+O\left(\frac{1}{1+r^2}\right)$$ and hence using \fref{estvooee}, \fref{estimatonebisbis}:
$$\left|\int \matchal L\zs\frac{(\Lambda Q)^2}{Q^2}\right|\lesssim \int \frac{|\mathcal L\zs|}{1+r^2}\lesssim \|\mathcal L\zs\|_{L^2}\lesssim \frac{b^2}{|\log b|}.$$
We now claim the algebraic identity 
\be
\label{defnneo}
\mathcal M(\Lambda ^2 Q)= \left(\frac{\Lambda Q}{Q}\right)^2
\ee which is proved below, and conclude using \fref{relationsm}: $$\int\left(\frac{\Lambda Q}{Q}\right)^2\Lambda Q=(\mathcal M(\Lambda ^2 Q),\Lambda Q)=-2\int \Lambda^2Q=0.$$ Finally, using \fref{defnneo}, \fref{relationsm}, \fref{estvooee}, \fref{degenrateonebis}:
\bee 
\left|\left(\e_2,\frac{(\Lambda Q)^2}{Q^2}\right)\right|&\lesssim& \left|( \e_2,\mathcal M(\Lambda^2Q+2\Lambda Q))\right| =\left|\int \mathcal M\e_2\nabla \cdot(r(\Lambda Q+2Q))\right|\\
& \lesssim & \left(\int Q|\nabla(\mathcal M\e_2)|^2\right)^{\frac 12}\left(\int \frac{r^2(\Lambda Q+2Q)^2}{Q}\right)^{\frac 12}\\
& \lesssim &  \left(\int Q|\nabla(\mathcal M\eh_2)|^2+\frac{b^4}{|\log b|^2}\right)^{\frac 12}.
\eee
The collection above bounds yields the admissible control:
\bee
\left|\int \frac{\wh_2^2}{\l^2Q_\l^2}\left[ \left(\lsl+\bh\right)(\Lambda Q)_{\l}\right]\right| & \lesssim &  \frac{b}{\l^6}\left[b^{\frac 32}\|\e_2\|_{L^2_Q}+\frac{b^3}{|\log b|^2}\right]+\frac{1}{100\l^6}\int Q|\nabla(\mathcal M\eh_2)|^2.
\eee

{\it Proof of  \fref{defnneo}}: For any well localized function, we have the commutator formula: $$\mathcal M\Lambda f=\Lambda \mathcal Mf +\frac{x\cdot\nabla Q}{Q^2}f-2\phi_f-\frac{\int f}{2\pi}$$ from which using \fref{relationsm},  \fref{philambdaq} and the critical relation $\int \Lambda Q=0$: 
\bee
\nonumber \mathcal M(\Lambda ^2Q)& = & \Lambda (\mathcal M\Lambda Q) +\frac{x\cdot\nabla Q}{Q^2}\Lambda Q-2\phi_{\Lambda Q}=  -4+\left(\frac{\Lambda Q}{Q}\right)^2-2\left[\phi_{\Lambda Q}+\frac{\Lambda Q}{Q}\right]\\
& = & \left(\frac{\Lambda Q}{Q}\right)^2.
\eee

{\bf step 6} Leading order $E_2$ terms. We claim the bounds:
\be
\label{neovhndkneo}
\left|(\eh_2,\mathcal M\mathcal L E_2)\right|\lesssim b\left[\frac{b^{\frac 32}}{|\log b|}\|\e_2\|_{L^2_Q}+\frac{b^3}{|\log b|^2}\right],
\ee
\bea
\label{neoihnehovinorno}
&&\nonumber \left|\int\frac{\Lambda Q+2Q}{Q^2}(\eh\mathcal LE_2+\eh_2E_2)-2\int \eh_1\cdot\nabla \phi_{E_2}-2\int\nabla \phi_{\eh}\cdot Q\nabla \mathcal ME_2\right|\\
& \lesssim & b\left[\frac{b^{\frac 32}}{|\log b|}\|\e_2\|_{L^2_Q}+\frac{b^3}{|\log b|^2}\right].
\eea
{\it Proof of \fref{estimatone}}: We compute using \fref{estvooee} and $\mathcal LT_1=\Lambda Q$:
\bee
\left|(\e_2,\mathcal M\mathcal L(\chi_{\frac{B_0}{4}}T_1))\right|& = & \left|(\e_2,\mathcal M\mathcal L[(1-\chi_{\frac{B_0}{4}})T_1)]\right|\\
& \lesssim & \|\e_2\|_{L^2_Q}\left(\int \frac{|\mathcal L[(1-\chi_{\frac{B_0}{4}})T_1]|^2}{Q}\right)^{\frac 12}.
\eee
We now estimate using the radial representation of the Poisson field and \fref{esttoneprop}:
\be
\label{estphibvijrbrijb}
\phi'_{(1-\chi_{\frac{B_0}{4}})T_1}=\frac{1}{r}\int_0^r(1-\chi_{\frac{B_0}{4}})T_1\tau d\tau=O\left(\frac{\log r}r{\bf 1}_{r\geq \frac{B_0}{4}}\right)
\ee
and hence the pointwise bound:
\be
\label{cbebei}
|\mathcal L[(1-\chi_{\frac{B_0}{4}})T_1]|\lesssim \frac{1}{1+r^4}{\bf 1}_{r\geq \frac{B_0}{4}}
\ee which leads to the bound:
\bee
\left|(\e_2,c_bb^2\mathcal M\mathcal L(\chi_{\frac{B_0}{4}}T_1))\right|& \lesssim & \frac{b^2}{|\log b|}\|\e_2\|_{L^2_Q}\left(\int_{r\geq \frac{B_0}{4}}\frac{1}{1+r^4}\right)^{\frac 12}\lesssim b\frac{b^{\frac 32}}{|\log b|}\|\e_2\|_{L^2_Q}.
\eee
We treat the term involving the radiation as follows. Arguing like for \fref{ccjooeccjwo}, we estimate:
\be
\label{cneneoeneo}
 \mathcal L(\chi_{\frac{B_0}{4}}T_1)=\Lambda Q+\tilde{\xi},\ \ \tilde{\xi}=O\left(\frac{{\bf 1}_{r\geq\frac{B_0}{4}}}{1+r^4}\right)
 \ee
from which using \fref{ccjooeccjwo}, \fref{defxi}, \fref{mcontiniuos}:
\bee
\left|(\mathcal L\zb,\mathcal M\mathcal L E_2)\right|& = &c_bb^2 \left|((b-\bh)\Lambda Q+\xi,\mathcal M(\Lambda Q+\tilde{\xi}))\right|\\
& \lesssim & \frac{b^2}{|\log b|}\left[|b-\bh|\int |\tilde{\xi}|+\int|\xi|+\|\xi\|_{L^2_Q}\|\tilde{\xi}\|_{L^2_Q}\right]\\
& \lesssim & \frac{b^4}{|\log b|^2}.
\eee
Similarily, using \fref{estvooee}, \fref{estimatone}, \fref{cneneoeneo}:
\bee
\left|(\mathcal L\zs,\mathcal M\mathcal L E_2)\right|& = &c_bb^2\left|(\matchal L\zs,\mathcal M(\Lambda Q+\tilde{\xi}))\right|\\
& = & c_bb^2\left|(\matchal L\zs,\mathcal M\tilde{\xi})\right|\lesssim \frac{b^2}{|\log b|}\|\mathcal L\zs\|_{L^2_Q}\|\tilde{\xi}\|_{L^2_Q}\\
& \lesssim & \frac{b^2}{|\log b|^2}\left(\frac{b^3}{|\log b|^2}\right)^{\frac 12}\left(\frac{1}{B_0^2}\right)^{\frac 12}\lesssim \frac{b^4}{|\log b|^2},
\eee
and \fref{neovhndkneo} is proved.\\
{\it Proof of \fref{neoihnehovinorno}}: We estimate in brute force using \fref{cneneoeneo}, \fref{coercbase} and Lemma \ref{lemmaradiation}:
\bee
b\left|\int\frac{\Lambda Q+2Q}{Q^2}(\eh\mathcal LE_2+\eh_2E_2)\right| & \lesssim & b\frac{b^2}{|\log b|}\int \frac{1}{(1+r^2)Q}\left[\frac{|\eh|}{1+r^4}+\frac{|\eh_2|}{1+r^2}\right]\\
& \lesssim & b\frac{b^2}{|\log b|}\left[\|\e\|_{L^2}+\|\e_2\|_{L^2_Q}+\|\z_2\|_{L^2_Q}+\|\zeta\|_{L^2}\right]\\
& \lesssim & b\left[C(M)\frac{b^2}{|\log b|}\|\e_2\|_{L^2_Q}+\frac{b^3}{|\log b|^2}\right].
\eee
Similarily, using \fref{controleun} and Lemma \ref{lemmaradiation}:
\bee
b\left|\int \eh_1\cdot\nabla \phi_{E_2}\right|& \lesssim & b\frac{b^2}{|\log b|}\int \frac{1+|\log r|}{1+r}\left[|\e_1|+|\nabla \zeta|+Q|\nabla \phi_\zeta|+\frac{|\zeta|}{1+r}\right]\\
& \lesssim & b\frac{b^2}{|\log b|}\left(\int (1+r^2)|\e_1|^2+|\zeta|^2+(1+r^2)|\nabla \zeta|^2+\frac{|\nabla \phi_\zeta|^2}{1+r^2}\right)^{\frac 12}\\
& \lesssim & b\frac{b^2}{|\log b|}\left[C(M)\|\e_2\|_{L^2_Q}+\frac{b}{|\log b|}\right]\lesssim b\left[\frac{b^2}{|\log b|}C(M)\|\e_2\|_{L^2_Q}+\frac{b^3}{|\log b |^2}\right],
\eee
and using an integration by parts:
\bee
b\left|\int \nabla \phi_{\eh}\cdot Q\nabla \mathcal ME_2\right|& =&b\left|(Q\eh+\nabla Q\cdot\nabla \phi_{\eh},\mathcal M E_2)\right|\\
& \lesssim & b\frac{b^2}{|\log b|}\int \left[\frac{|\eh|}{1+r^4}+\frac{|\nabla \phi_{\eh}|}{1+r^5}\right](1+r^2)\\
& \lesssim &b\frac{b^2}{|\log b|}\left(\int |\eh|^2+\frac{|\nabla \phi_{\eh}|^2}{(1+r^2)(1+|\log r|^2)}\right)^{\frac 12}\\
& \lesssim & b\left[\frac{b^2}{|\log b|}C(M)\|\e_2\|_{L^2_Q}+\frac{b^3}{|\log b |^2}\right],
\eee 
and \fref{neoihnehovinorno} is proved.\\

{\bf step 7} $\Psit_b$ terms. We now estimate the $\Psit_b$ type of terms. From \fref{roughboundltaowlocbisbis},
\bee
\left|(\eh_2,\mathcal M\mathcal L\Psit_b)\right|&\lesssim &\left(\int Q|\nabla \mathcal M\eh_2|^2\right)^{\frac12}\left(\int Q|\nabla \mathcal M\Psit_b|^2\right)^{\frac12}\\
& \leq & \frac{1}{100}\int Q|\nabla \mathcal M\eh_2|^2+C\frac{b^4}{|\log b|^2}.
\eee
Next, from \fref{roughboundltaowloc}:
\bee
\left| \frac{b}{\l^6}\int\frac{\Lambda Q+2Q}{Q^2}\eh\mathcal L\Psit_b\right|&\lesssim& \frac{b}{\l^6}\int \frac{|\eh\mathcal L\Psit_b|}{Q(1+r^2)}\lesssim \frac{b}{\l^6}\|\eh\|_{L^2}\|\mathcal L\Psit_b\|_{L^2_Q}\\
& \lesssim & C(M)\frac{b}{\l^6}\frac{b^{\frac 52}}{|\log b|}\left(\|\e_2\|_{L^2_Q}+\|\z\|_{L^2}\right)\lesssim  \frac{b}{\l^6}\left[b^2\|\e_2\|_{L^2_Q}+\frac{b^3}{|\log b|^2}\right].
\eee
From \fref{roughboundltaowloc}, \fref{coercbase}:
\bee
\left| \frac{b}{\l^6}\int\frac{\Lambda Q+2Q}{Q^2}\eh_2\Psit_b\right|&\lesssim&  \frac{b}{\l^6}\int \frac{|\eh_2\Psit_b|}{Q(1+r^2)}\lesssim \frac{b}{\l^6}\|\eh_2\|_{L^2_Q}\|\Psit_b\|_{L^2}\\
& \lesssim &  \frac{b}{\l^6}\frac{b^2}{|\log b|}\left[\|\e_2\|_{L^2_Q}+\|\zeta\|_{L^2_Q}\right]\lesssim \frac{b}{\l^6}\left[b^2\|\e_2\|_{L^2_Q}+\frac{b^3}{|\log b|^2}\right].
\eee
From \fref{roughboundltaowloc}, \fref{controleun}:
\bee
\left|\frac{b}{\l^6}\int \eh_1\cdot\nabla \phi_{\Psit_b}\right|&\lesssim&\frac{b}{\l^6} \left(\int (1+r^2)|\eh_1|^2\right)^{\frac 12}\left(\int \frac{|\nabla \phi_{\Psit_b}|^2}{1+r^2}\right)^{\frac12}\\
& \lesssim &  \frac{b}{\l^6}\frac{b^{\frac 52}}{|\log b|}\left[C(M) \|\e_2\|_{L^2_Q}+\left(\int|\nabla \z|^2+\int \frac{|\zeta|^2}{1+r^2}+\int\frac{ |\nabla \phi_{\zeta}|^2}{1+r^4}\right)^{\frac 12}\right]\\
& \lesssim & \frac{b}{\l^6}\left[b^2\|\e_2\|_{L^2_Q}+\frac{b^3}{|\log b|^2}\right].
\eee
Finally, from \fref{roughboundltaowloc}, \fref{interpolationfield}:
\bee
&&\left|\frac{b}{\l^6}\int\nabla \phi_{\eh}\cdot Q\nabla \mathcal M\Psit_b\right|\lesssim \frac{b}{\l^6}\int |\nabla \phi_{\eh}|\left[|\nabla \Psit_b|+\frac{|\Psit_b|}{1+r}+Q|\nabla \phi_{\Psit_b}|\right]\\
& \lesssim & \frac{b}{\l^6}\left(\int \frac{|\nabla \phi_\e|^2}{1+r^2}+\frac{|\nabla \phi_{\z}|^2}{1+r^2}\right)^{\frac 12}\left(\int (1+r^2)\left[|\nabla \Psit_b|^2+\frac{|\Psit_b|^2}{1+r^2}+Q^2|\nabla \phi_{\Psit_b}|^2\right]\right)^{\frac 12}\\
& \lesssim &  \frac{b}{\l^6}\frac{b^{\frac 52}}{|\log b|}\left[|\log b|^C\|\e_2\|_{L^2_Q}+\frac{b}{|\log b|}\right]\lesssim  \frac{b}{\l^6}\left[b^2\|\e_2\|_{L^2_Q}+\frac{b^3}{|\log b|^2}\right].
\eee

{\bf step 8} Modulation terms. We now treat the modulation parameters given by \fref{mohat} which require the introduction of the lifted parameter $\hb$ and the radiation term $\zeta_b$. We decompose:
$$\widehat{\Mod}=\Mod_0+\Mod_1,
$$
\be
\label{defmodo}
\Mod_0=-\pa_s\tilde{Q}_{\bh},
\ee
\be
\label{defmodone}
\Mod_1=\left(\lsl+b\right)\Lambda \qbt=\left(\lsl+b\right)(\Lambda Q+b\Lambda \tt_1+b^2\Lambda \tt_2),
\ee 
and further split: $$\Mod_0=\Mod_{0,1}+\Mod_{0,2},$$
 $$\Mod_{0,1}=-\hb_sT_1, \ \ \Mod_{0,2}=-\hb_s\left(\tt_1-T_1+\hb\frac{\partial\tt_1}{\partial b}+2\hb\tt_2+\hb^2\frac{\partial \tt_2}{\partial b}\right).$$ We claim the bounds:
\be
\label{cneneonveonveo}
\int |r^i\pa^i_r\Mod_{0,1}|^2+\int \frac{|\mathcal L\Mod_{0,1}|^2}{Q}+ \int \frac{|\nabla \phi_{\Mod_{0,1}}|^2}{1+r^2}\lesssim |\hb_s|^2,
\ee
\be
\label{cneneonveonveobis}
\int |r^i\pa^i_r\Mod_{0,2}|^2+\int \frac{|\mathcal L\Mod_{0,2}|^2}{Q}+ \int \frac{|\nabla \phi_{\Mod_{0,2}}|^2}{1+r^2}\lesssim b|\hb_s|^2,
\ee
\be
\label{modoneone}
\int |r^i\pa^i_r\Mod_{1}|^2+ \int \frac{|\nabla \phi_{\Mod_{1}}|^2}{1+r^2}\lesssim \left|\lsl+b\right|^2,
\ee
\be
\label{modoneonbise}
\int \frac{|\mathcal L\Mod_{1}|^2}{Q}\lesssim b^2\left|\lsl+b\right|^2.
\ee

{\it Proof of \fref{cneneonveonveo}}: This is a direct consequence of \fref{developpementmonprime}, \fref{developpementT1}.\\
{\it Proof of \fref{cneneonveonveobis}}: We estimate from \fref{esttoneproploc}, \fref{esttpemfmp}, \fref{esttwoloc}, \fref{estttwodtdbloc}:
\bea
\label{Pneiovbebvoebb}
&&\left|r^i\pa_r^i\left[\tt_1-T_1+b\frac{\partial\tt_1}{\partial b}+2b\tt_2+b^2\frac{\partial \tt_2}{\partial b}\right]\right|\\
\nonumber &\lesssim &\frac{{\bf 1}_{B_1\leq r\leq 2B_1}}{r^2}+ b\left[r^2{\bf 1}_{r\le1 }+\frac{1+|\log (r\sqrt{b})|}{|\log b|}{\bf 1}_{1\leq r\leq 6B_0}+\frac{1}{b^2r^4}{\bf 1}_{B_0\leq r\leq 2B_1}\right],
\eea
and using the radial representation of the Poisson field:
\be
\label{ceoknekonvoen}
\left|\nabla  \phi_{r^i\pa_r^i\left[\tt_1-T_1+b\frac{\partial\tt_1}{\partial b}+2b\tt_2+b^2\frac{\partial \tt_2}{\partial b}\right]}\right|\lesssim b(1+r){\bf 1}_{r\leq B_0}+\frac{{\bf 1}_{r\geq B_0}}{r}.
 \ee
 This yields:
 \bee
 \int |r^i\pa^i_r\Mod_{0,2}|^2& \lesssim & |\hb_s|^2\int\left[\frac{{\bf 1}_{B_1\leq r\leq 2B_1}}{1+r^4}+b^2\left(\frac{1+|\log (r\sqrt{b})|}{|\log b|}{\bf 1}_{1\leq r\leq 6B_0}\right)^2+\frac{1}{b^2r^8}{\bf 1}_{B_0\leq r\leq 2B_1}\right]\\
 & \lesssim & b|\hb_s|^2,
 \eee
 $$ \int \frac{|\nabla \phi_{\Mod_{0,2}}|^2}{1+r^2}\lesssim |\bh_s|^2\int\frac{1}{1+r^2}\left[ b^2(1+r^2){\bf 1}_{r\leq B_0}+\frac{{\bf 1}_{r\geq B_0}}{r^2}\right]\lesssim b|\bh_s|^2,$$
 and similarly using the explicit representation \fref{formuleL}:
 \bee
 \int \frac{|\mathcal L\Mod_{0,2}|^2}{Q}& \lesssim & b|\hb_s|^2.
 \eee

{\it Proof of \fref{modoneone}}: We extract from \fref{esttoneproploc}, \fref{esttwoloc} the rough bound:
\be
\label{neovnononvne}
\left|r^i\pa_r^i\left(\tt_1+b\tt_2\right)\right|\lesssim \frac{1}{1+r^2}{\bf 1}_{r\leq 2B_1},
\ee
\be
\label{cneonceoneo}
|\nabla \phi_{r^i\pa_r^i(\tt_1+b\tt_2)}|\lesssim \frac{1+|\log r|}{1+r}.
\ee This yields:
$$\int |r^i\pa_r^i\Mod_{1,1}|^2\lesssim\left|\lsl+b\right|^2\int \frac{1}{1+r^4}\lesssim \left|\lsl+b\right|^2,$$ 
$$ \int \frac{|\nabla \phi_{\Mod_{1,1}}|^2}{1+r^2}\lesssim \left|\lsl+b\right|^2\int \frac{1+|\log r|^2}{1+r^4}\lesssim \left|\lsl+b\right|^2.$$
{\it Proof of \fref{modoneonbise}}: We use the cancellation $\mathcal L\Lambda Q=0$, the bound \fref{neovnononvne}, \fref{cneonceoneo} and the explicit formula \fref{formuleL} to estimate:
$$\int \frac{|\mathcal L\Mod_{1,1}|^2}{Q}\lesssim b^2\left|\lsl+b\right|^2\int\left[\frac{1}{1+r^4}{\bf 1}_{1\leq r\leq 2B_1}+ \frac{1}{1+r^6}\frac{1+|\log r|^2}{1+r^2}\right]\lesssim  b^2\left|\lsl+b\right|^2.$$
 We are now in position to estimate all terms in \fref{noumberneergyidentity}. From \fref{degenrateoneone},  \fref{cneneonveonveobis}, \fref{modoneonbise}:
 \bee
\left|(\eh_2,\mathcal M\mathcal L\widehat{\Mod})\right|& \lesssim &\left|(\mathcal M\zeta_2,\mathcal L\widehat{\Mod})\right|+\|\e_2\|_{L^2_Q}\left[\|\mathcal L\Mod_{0,2}\|_{L^2_Q}+\|\mathcal L\Mod_1\|_{L^2_Q}\right]\\
& \lesssim & \left(\|\e_2\|_{L^2_Q}+\frac{b^{\frac32}}{|\log b|}\right)\left[\|\mathcal L\Mod_{0,2}\|_{L^2_Q}+\|\mathcal L\Mod_1\|_{L^2_Q}\right]\\
& \lesssim & \left(\|\e_2\|_{L^2_Q}+\frac{b^{\frac32}}{|\log b|}\right)\left[\sqrt{b}|\hat{b}_s|+b\left|\lsl+b\right|\right].
\eee
 We now use the pointwise bound \fref{estlambda} and in a fundamental way the improved bound \fref{pointziender} which motivates the introduction of the lifted parameter $\hat{b}_s$ to conclude:
\bee
\left|(\eh_2,\mathcal M\mathcal L\widehat{\Mod})\right|& \lesssim & \left(\|\e_2\|_{L^2_Q}+\frac{b^{\frac 32}}{|\log b|}\right)\left[\sqrt{b}\frac{b^2}{|\log b|}+bC(M)\frac{b^2}{|\log b|}\right]\\
& \lesssim & b\left[\frac{b^{\frac 32}}{|\log b|}\|\e_2\|_{L^2_Q}+\frac{b^3}{|\log b|^2}\right].
\eee
Next, from \fref{lemmaradiation}, \fref{estlambda}, \fref{pointziender}, \fref{coercbase}:
\bee
b\left|\int\frac{\Lambda Q+2Q}{Q^2}\eh\mathcal L\widehat{\Mod}\right|& \lesssim &b\|\eh\|_{L^2}\|\mathcal  L\widehat{\Mod}\|_{L^2_Q}\lesssim b\left[C(M)\|\e_2\|_{L^2_Q}+\|\zeta\|_{L^2}\right]\left[|\hb_s|+b\left|\lsl+b\right|\right]\\
& \lesssim & b\left[C(M)\|\e_2\|_{L^2_Q}+\frac{b}{|\log b|}\right]\frac{b^2}{|\log b|}\\
& \lesssim & b\left[\frac{b^{\frac 32}}{|\log b|}\|\e_2\|_{L^2_Q}+\frac{b^3}{|\log b|^2}\right],
\eee
\bee
b\left|\int\frac{\Lambda Q+2Q}{Q^2}\eh_2\widehat{\Mod}\right|& \lesssim & b\left[\|\e_2\|_{L^2_Q}+\|\zeta\|_{L^2_Q}\right]\|\widehat{\Mod}\|_{L^2}\\
& \lesssim& b\left[\|\e_2\|_{L^2_Q}+\frac{b}{|\log b|}\right]\left[|\hb_s|+\left|\lsl+b\right|\right]\\
& \lesssim &  bC(M)\left[\|\e_2\|_{L^2_Q}+\frac{b}{|\log b|}\right]\frac{b^2}{|\log b|}\\
& \lesssim & b\left[\frac{b^{\frac 32}}{|\log b|}\|\e_2\|_{L^2_Q}+C(M)\frac{b^3}{|\log b|^2}\right].
\eee
Similarily,
\bee
b\left|\int \eh_1\cdot\nabla \phi_{\widehat{\Mod}}\right|&\lesssim& b\left[\int (1+r^2)|\e_1|^2+\int |\zeta|^2+(1+r^2)|\nabla \zeta|^2+\frac{|\nabla \phi_\zeta|^2}{1+r^2}\right]^{\frac 12}\left[\int \frac{|\nabla \phi_{\widehat{\Mod}}|^2}{1+r^2}\right]^{\frac 12}\\
& \lesssim & b\left[C(M)\|\e_2\|_{L^2_Q}+\frac{b}{|\log b|}\right]\left[|\hb_s|+\left|\lsl+b\right|\right]\\
& \lesssim & b\left[\frac{b^{\frac 32}}{|\log b|}\|\e_2\|_{L^2_Q}+C(M)\frac{b^3}{|\log b|^2}\right],
\eee
and using \fref{interpolationfield}:
\bee
&&b\left|\int \nabla\phi_{\eh}\cdot Q\nabla \matchal M\widehat{\Mod}\right| \lesssim b\int|\nabla\phi_{\eh}|\left||\nabla\widehat{\Mod}|+\frac{|\widehat{\Mod}|}{1+r}+\frac{|\nabla \phi_{\widehat{\Mod}}|}{1+r^4}\right|\\
& \lesssim & b\left(\int \frac{|\nabla \phi_\e|^2}{1+r^2}+\frac{|\nabla \phi_\z|^2}{1+r^2}\right)^{\frac 12}\left(\int(1+r^2)|\nabla \widehat{\Mod}|^2+|\widehat{\Mod}|^2+\frac{|\nabla \phi_{\widehat{\Mod}}|^2}{1+r^2}\right)^{\frac 12}\\
& \lesssim & b\left(|\log b|^C\|\e_2\|^2_{L^2_Q}+b^{10}+\frac{b^2}{|\log b|^2}\right)^{\frac 12}\left[|\hb_s|+\left|\lsl+b\right|\right]\\
& \lesssim & bC(M)\frac{b^2}{|\log b|}\left(|\log b|^C\|\e_2\|^2_{L^2_Q}+b^{10}+\frac{b^2}{|\log b|^2}\right)^{\frac 12}\\
& \lesssim &  b\left[\frac{b^{\frac 32}}{|\log b|}\|\e_2\|_{L^2_Q}+C(M)\frac{b^3}{|\log b|^2}\right].
 \eee
This concludes the control of the forcing induced by the modulation parameters.\\

{\bf step 9} $E_1$ terms. We split $$E_1=-\zeta_2+E_{1,1}, \ \ E_{1,1}=-\lsl \Lambda \zeta.$$
{\it $\zeta_2$ terms}: Fom \fref{degenrateonebis}:
\bee
|(\eh_2,\matchal M\mathcal L\zeta_2)|& = & \left(\int Q|\nabla\matchal M\eh_2|^2\right)^{\frac 12}\left(\int Q|\nabla\matchal M\zeta_2|^2\right)^{\frac 12}\\
& \leq & \frac{1}{100}\int Q|\nabla\matchal M\eh_2|^2+C\frac{b^4}{|\log b|^2}.
\eee
Next, after an integration by parts:
\bee
b\left|\int \frac{\Lambda Q+2Q}{Q^2}\eh\mathcal L\zeta_2\right| & \lesssim & b\left(\int Q|\nabla\matchal M\zeta_2|^2\right)^{\frac 12}\left(\int |\nabla \eh|^2+\frac{|\eh|^2}{1+r^2}\right)^{\frac12}\\
& \lesssim & \frac{b^3}{|\log b|^2}\left[\frac{b}{|\log b|}+C(M)\|\e_2\|_{L^2_Q}\right]\\
& \lesssim & b\left[\frac{b^{\frac 32}}{|\log b|}\|\e_2\|_{L^2_Q}+C(M)\frac{b^3}{|\log b|^2}\right].
\eee
Similarily, using Lemma \ref{lemmaradiation}:
\bee
b\left|\int\nabla\phi_{\eh}\cdot Q\nabla \mathcal M\zeta_2\right|& \lesssim & b\left(\int\frac{|\nabla \phi_{\eh}|^2}{1+r^4}\right)^{\frac 12}\left(\int Q|\nabla \mathcal M\zeta_2|^2\right)^{\frac 12}\\
& \lesssim & b\frac{b^2}{|\log b|}\left[C(M)\|\e_2\|_{L^2_Q}+\frac{b}{|\log b|}\right]\\
& \lesssim &  b\left[\frac{b^{\frac 32}}{|\log b|}\|\e_2\|_{L^2_Q}+\frac{b^3}{|\log b|^2}\right].
\eee
To control the last terms, we introduce the decomposition from \fref{zetabig}: $$\zeta_2=(b-\bh)\Lambda Q+\tilde{\zeta}_2, \  \ \tilde{\zeta}_2=\xi+\mathcal L\zeta_{\rm sm}, \ \ \xi=|b-\bh|O\left(\frac{{\bf 1}_{r\geq B_1}}{1+r^4}\right).$$ Observe using Lemma \ref{lemmaradiation} that: 
\be
\label{nekvneoneovn}
\|\tilde{\zeta}_2\|_{L^2}^2\lesssim \frac{b^4}{|\log b|^2}+\frac{b^2}{|\log b|^2}\int_{r\geq B_1}\frac{1}{1+r^8}\lesssim \frac{b^4}{|\log b|^2},
\ee and using by construction $\int \tilde{\zeta}_2=0$ and \fref{improvedlinfty}:
\be
\label{cnocneonoenvov}
\int|\nabla \phi_{\tilde{\zeta}_2}|^2\lesssim \|\tilde{\zeta}_2\|_{L^2_Q}^2\lesssim \frac{b^3}{|\log b|^2}.
\ee
We therefore estimate in brute force:
\bee
b\left|\int \frac{\Lambda Q+2Q}{Q^2}\eh_2\tilde{\zeta}_2\right|&\lesssim &b\left[\|\e_2\|_{L^2_Q}+\|\zeta_2\|_{L^2_Q}\right]\|\tilde{\zeta}_2\|_{L^2}\\
& \lesssim &  b\frac{b^2}{|\log b|}\left[\|\e_2\|_{L^2_Q}+\frac{b}{|\log b|}\right]\\
& \lesssim &  b\left[\frac{b^{\frac 32}}{|\log b|}\|\e_2\|_{L^2_Q}+\frac{b^3}{|\log b|^2}\right].
\eee
Similarily, 
\bee
b\int |\e_1\cdot\nabla \phi_{\tilde{\zeta}_2}|\lesssim b\left(\int (1+r^2)|\e_1|^2\right)^{\frac 12}\left(\int \frac{|\nabla\phi_{\tilde{\zeta}_2}|^2}{1+r^2}\right)^{\frac 12}\lesssim bC(M)\|\e_2\|_{L^2_Q}\frac{b^{\frac 32}}{|\log b|}.
\eee
and for the term induced by the radiation using \fref{cnocneonoenvov}, \fref{nekvneoneovn} and the degeneracy \fref{gfijgbeigei}:
\bee
b\left|\int Q\nabla (\mathcal M\zeta)\cdot\nabla \phi_{\tilde{\zeta}_2}\right|& = & b\left|\int \zeta_2\phi_{\tilde{\zeta}_2}\right|\lesssim b|b-\hat{b}|\int\left|\tilde{\zeta}_2\phi_{\Lambda Q}\right|+b\int |\nabla \phi_{\tilde{\zeta}_2}|^2\\
& \lesssim& \frac{b^2}{|\log b|}\|\tilde{\zeta}_2\|_{L^2}+b\frac{b^3}{|\log b|^2}\lesssim \frac{b^4}{|\log b|^2}.
\eee
It remains to estimate the term:
$$
b(b-\hat{b})\left[\int \frac{\Lambda Q+2Q}{Q^2}\eh_2\Lambda Q-2\int \eh_1\cdot\nabla \phi_{\Lambda Q}\right]=  b(b-\hat{b})\int \eh_2\left[\frac{\Lambda Q(\Lambda Q+2Q)}{Q^2}+2\phi_{\Lambda Q}\right].$$
Here we need an additional algebra in order to be able to use dissipation\footnote{for which a better bound than for $\|\e_2\|_{L^2_Q}$ holds in time averaged sense.}. We compute from \fref{philambdaq}:
$$\frac{\Lambda Q(\Lambda Q+2Q)}{Q^2}+2\phi_{\Lambda Q}=\left(\frac{\Lambda Q}{Q}\right)^2-4.$$
We now use \fref{defnneo}, integrate by parts and use the degeneracy \fref{estvooee} and \fref{gfijgbeigei} to estimate:
\bee
&& \left|\int \eh_2\left[\frac{\Lambda Q(\Lambda Q+2Q)}{Q^2}+2\phi_{\Lambda Q}\right]\right|=\left|\int \eh_2\mathcal M(\Lambda^2Q)\right|\\
& = & \left|\int \eh_2\mathcal M(\Lambda^2Q+2\Lambda Q)\right| =\left|\int \mathcal M\eh_2\nabla \cdot(r(\Lambda Q+2Q))\right|\\
& = & \left(\int Q|\nabla(\mathcal M\eh_2)|^2\right)^{\frac 12}\left(\int \frac{r^2(\Lambda Q+2Q)^2}{Q}\right)^{\frac 12}\lesssim  \left(\int Q|\nabla(\mathcal M\eh_2)|^2\right)^{\frac 12}
 \eee
 and hence the control of the last term:
 \bee
 &&\left|b(b-\hat{b})\left[\int \frac{\Lambda Q+2Q}{Q^2}\eh_2\Lambda Q-2\int \eh_1\cdot\nabla \phi_{\Lambda Q}\right]\right|\\
 & \lesssim & \frac{b^2}{|\log b|} \left(\int Q|\nabla(\mathcal M\eh_2)|^2\right)^{\frac 12}\\
 & \leq & \frac 1{100} \int Q|\nabla(\mathcal M\eh_2)|^2+C\frac{b^4}{|\log b|^2}.
 \eee
 {\it $E_{1,1}$ terms}:  
We systematically use the rough pointwise bound from \fref{estlambda}: $$\left|\lsl\right|\lesssim b.$$ For the first term, we estimate first using the decomposition \fref{zetabig}, \fref{defxi} and \fref{estimatone}:
\bee
&&b|(\eh_2,\mathcal M\mathcal L(\Lambda \zs))|\lesssim  b|(\xi+\mathcal L\zs,\matchal M\mathcal L(\Lambda \zs)|+b\|\e_2\|_{L^2_Q}\|\mathcal L(\Lambda \zs)\|_{L^2_Q}\\
& \lesssim & b\left(\int_{r\geq B_1}\frac{b^2}{|\log b|^2}\frac{1}{(1+r^8)Q}+\int\frac{|\mathcal L\zs|^2}{Q}\right)^{\frac 12}\|\mathcal L(\Lambda \zs)\|_{L^2_Q}+ b\frac{b^{\frac 32}}{|\log b|}\|\e_2\|_{L^2_Q}\\
&\lesssim & b\left(\frac{b^3}{|\log b|^2}\right)^{\frac 12}\left(\frac{b^3}{|\log b|^2}\right)^{\frac 12}+ b\frac{b^{\frac 32}}{|\log b|}\|\e_2\|_{L^2_Q}\lesssim \frac{b^4}{|\log b|^2}+ b\frac{b^{\frac 32}}{|\log b|}\|\e_2\|_{L^2_Q}.
\eee
We now observe using \fref{developpementT1} the cancellation $$r^i\pa_r^i(\Lambda T_1)=O\left(\frac{1+|\log r|}{1+r^4}\right)$$ which implies
\bee
\int Q|\nabla \mathcal M\Lambda(\chi_{B_1}T_1)|^2& \lesssim & \int \frac{|\nabla\phi_{\chi_{B_1}T_1}|^2}{1+r^4}+\int (1+r^4)|\nabla \Lambda(\chi_{B_1}T_1)|^2+(1+r^2)|\Lambda (\chi_{B_1}T_1)|^2\\
& \lesssim & 1+\int_{B_1\leq r\leq 2B_1}\frac{1}{1+r^2}\lesssim 1
\eee
from which:
\bee
b\left|(\eh_2,\mathcal M\mathcal L(\Lambda \zb)\right|& = & b|b-\bh|\left|\int Q\nabla \mathcal M\eh_2\cdot\nabla \mathcal M\Lambda (\chi_{B_1}T_1)\right|\\
& \leq & \frac{1}{100}\int Q|\nabla \matchal M\eh_2|^2+C\frac{b^4}{|\log b|^2}\int Q|\nabla \mathcal M\Lambda(\chi_{B_1}T_1)|^2\\
& \leq & \frac{1}{100}\int Q|\nabla \matchal M\eh_2|^2+C\frac{b^4}{|\log b|^2}.
\eee
We further estimate using Lemma \ref{lemmaradiation}:
\bee
b\left|\int \frac{\Lambda Q+2Q}{Q^2}\eh\mathcal L(b\Lambda \zeta)\right|& \lesssim & b^2\|\eh\|_{L^2}\|\mathcal L\Lambda \zeta\|_{L^2_Q}\lesssim b\left[C(M)\|\e_2\|_{L^2_Q}+\frac{b}{|\log b|}\right]\frac{b^2}{|\log b|}\\
& \lesssim & b\left[\frac{b^{\frac 32}}{|\log b|}\|\e_2\|_{L^2_Q}+\frac{b^3}{|\log b|^2}\right],
\eee
\bee
b\left|\int \frac{\Lambda Q+2Q}{Q^2}\eh_2(b\Lambda \zeta)\right|& \lesssim & b^2\|\eh_2\|_{L^2_Q}\|\Lambda \zeta\|_{L^2}\lesssim b\left[\|\e_2\|_{L^2_Q}+\frac{b}{|\log b|}\right]\frac{b^2}{|\log b|}\\
& \lesssim & \frac{b^4}{|\log b|^2}+ b\frac{b^{\frac 32}}{|\log b|}\|\e_2\|_{L^2_Q},
\eee
\bee
&&b\left|\int \eh_1\cdot\nabla \phi_{b\Lambda \zeta}\right|\\
& \lesssim& b^2\left(\int (1+r^2)|\e_1|^2+\int (1+r^2)|\nabla \zeta|^2+\int |\zeta|^2+\int \frac{|\nabla \phi_{\zeta}|^2}{1+r^2}\right)^{\frac 12}\left(\int \frac{|\nabla(y\cdot\nabla  \phi_{\zeta})|^2}{1+r^2}\right)^{\frac 12}\\
& \lesssim & b^2\left[C(M)\|\e_2\|_{L^2_Q}+\frac{b}{|\log b|}\right] \frac{b}{|\log b|} \lesssim  \frac{b^4}{|\log b|^2}+ b\frac{b^{\frac 32}}{|\log b|}\|\e_2\|_{L^2_Q},
\eee
and using Proposition \ref{bootbound}:
\bee
&&b\left|\int \nabla \phi_{\eh}\cdot Q\nabla \mathcal M(b\Lambda \zeta)\right|\lesssim b^2\left|(Q\eh+\nabla Q\cdot\nabla \phi_{\eh},\frac{\Lambda \zeta}{Q})\right|+b^2\left|\int \nabla \phi_{\eh}\cdot Q\nabla\phi_{\Lambda \zeta}\right|\\
& \lesssim & b^2\|\Lambda \zeta\|_{L^2}\left(\int |\eh|^2+\int\frac{|\nabla \phi_{\eh}|^2}{1+r^2}\right)^{\frac 12}+b^2\left(\int\frac{|\nabla \phi_{\eh}|^2}{1+r^4}\right)^{\frac 12}\left(\int\frac{|\nabla \phi_{\Lambda \zeta}|^2}{1+r^4}\right)^{\frac 12}\\
& \lesssim & b^2\frac{b}{|\log b|}\left(|\log b|^C\|\e_2\|_{L^2_Q}+\frac{b}{|\log b|}\right)\lesssim \frac{b^4}{|\log b|^2}+ b\frac{b^{\frac 32}}{|\log b|}\|\e_2\|_{L^2_Q}.
\eee

{\bf step 10} Non linear terms.  We estimate the contribution of the non linear term $$N(\e)=\nabla \cdot(\e\nabla\phi_\e)= \e^2+\nabla \e\cdot\nabla \phi_\e.$$ We decompose using \fref{cnecneoneonenoe}:
\be
\label{defdecompfojep}
Q\nabla \mathcal MN(\e)=\nabla N(\e)+Q\nabla \phi_{N(\e)}+N(\e)\nabla \phi_Q=S_1+S_2
\ee with $$S_1=2\e\nabla \e+\nabla^2\phi_\e\cdot\nabla \e+Q\nabla \phi_{N(\e)}+N(\e)\nabla \phi_Q, \ \ S_2=\nabla^2 \e_\cdot\nabla \phi_\e.$$
We claim the bounds:
\be
\label{controlphin}
\forall 2<p<+\infty, \ \ \|\nabla\phi_{N(\e)}\|_{L^p}\lesssim C(p)|\log b|^{C(p)}\left[\|\e_2\|_{L^2_Q}^{2}+b^{10}\right],
\ee
\be
\label{invoinvoorjg}
\int |N(\e)|^2\lesssim  |\log b|^C\|\e_2\|_{L^2_Q}^{4}+b^{20},
\ee
\be
\label{cneoeoufoef}
\int \frac{|S_1|^2}{(1+r^2)Q}\lesssim \log b|^{C}  \|\e_2\|_{L^2_Q}^4+b^{20},
\ee
\be
\label{noevneonveo}
\int \frac{|\nabla \cdot S_1|^2}{Q}\lesssim|\log b|^{C}  \|\e_2\|_{L^2_Q}^4+b^{20},
\ee
\be
\label{bounfknlsbis}
\int \frac{|S_2|^2}{Q}\lesssim  \|\e_2\|_{L^2_Q}^{4-\frac{1}{100}},
\ee
\be
\label{bounfknls}
 \int (1+r^2)|S_2|^2\lesssim |\log b|^{C}  \|\e_2\|_{L^2_Q}^4+b^{20}.
\ee
{\it Proof of \fref{controlphin}}: We estimate from Sobolev and Plancherel: 
\be
\label{ceniohoeh}
\|\nabla^2 \phi_\e\|_{L^{\infty}}\lesssim \|\nabla^2\phi_\e\|_{H^2}\lesssim \|\e\|_{H^2}\lesssim C(M)\|\e_2\|_{L^2_Q}.
\ee
We then estimate from Hardy Littlewood Sobolev:
$$\|\nabla\phi_{N(\e)}\|_{L^p}\lesssim \|N(\e)\|_{L^r}\ \ \mbox{with}\ \ r=\frac{2p}{p+2}\in (1,2).$$ Then from Sobolev and \fref{estlossyfield}, \fref{ceniohoeh}:
\bea
\label{cneobveooeoeh}
\nonumber \|N(\e)\|_{L^r}& \lesssim& \|\e\|^2_{L^{2r}}+\|\nabla \e\cdot\nabla\phi_\e\|_{L^{r}}\lesssim \|\e\|^2_{H^2}+\|\nabla \e\cdot\nabla \phi_\e\|_{H^1}\\
\nonumber & \lesssim & \|\e\|_{H^2}^2+\left\|\frac{\nabla \phi_\e}{1+r}\right\|_{L^{\infty}}\left[\|(1+r)\nabla \e\|_{L^2}+\|(1+r)\nabla^2\e\|_{L^2}\right]+\|\nabla^2\phi_\e\|_{L^{\infty}}\|\nabla \e\|_{L^2}\\
& \lesssim & |\log b|^C(\|\e_2\|_{L^2_Q}+b^{10})\|\e_2\|_{L^2_Q}
\eea
and \fref{controlphin} follows.\\
{\it Proof of \fref{invoinvoorjg}}: It follows from the chain of estimates \fref{cneobveooeoeh} with $r=2$.\\
{\it Proof of \fref{cneoeoufoef}}:
We estimate in brute force using \fref{ceniohoeh}, \fref{controlphin}, \fref{invoinvoorjg}:
\bee
\int \frac{|S_1|^2}{(1+r^2)Q}& \lesssim & (\|\e\|_{L^{\infty}}^2+\|\nabla^2\phi_\e\|_{L^{\infty}}^2)\int(1+r^2)|\nabla \e|^2+\int \frac{|\nabla \phi_{N(\e)}|^2}{1+r^6}+  \int (N(\e))^2\\
& \lesssim & C(M)\|\e_2\|_{L^2_Q}^4+\|\nabla \phi_{N(\e)}\|_{L^4}^2+ |\log b|^C\|\e_2\|_{L^2_Q}^{4}+b^{20}\\
& \lesssim&  |\log b|^C\|\e_2\|_{L^2_Q}^{4}+b^{20}.
\eee
{\it Proof of \fref{noevneonveo}}: 
We compute:
$$
\nabla \cdot S_1=  2\e\Delta \e+2|\nabla \e|^2+\nabla \cdot(\nabla^2\phi_\e\cdot\nabla\e)+\nabla Q\cdot\nabla \phi_{N(\e)}+2QN(\e)+\nabla \phi_Q\cdot\nabla N(\e)
$$
and hence the bound:
\bea
\label{boundsonehifo}
\nonumber |\nabla \cdot S_1| & \lesssim & |\e||\Delta \e|+|\nabla \e|^2+|\nabla^3\phi_\e||\nabla \e|+|\nabla^2\phi_\e||\nabla^2\e|+\frac{|\nabla \phi_{N(\e)}|}{1+r^5}\\
& + & \frac{|\e|^2+|\nabla \e||\nabla \phi_\e|}{1+r^4}+\frac{|\e||\nabla \e|+|\nabla^2\phi_\e||\nabla \e|+|\nabla^2\e||\nabla \phi_\e|}{1+r}.
\eea
We have the nonlinear estimate using Sobolev:
\bee
\int \frac{|\nabla \e|^4}{Q}&\lesssim& \int |\nabla \e|^4+\Sigma_{i=1}^2\int|\nabla [(1+x_i)\e]|^4\\
& \lesssim& \|\e\|_{H^2}^4+\Sigma_{i=1}^2\left(\int|\nabla[(1+x_i)\e]|^2\right)\left(\int |\nabla^2[(1+x_i)\e]|^2\right)\\
& \lesssim & c(M)\|\e_2\|_{L^2_Q}^4.
\eee
Similarily, using Lemma \ref{lemmahtwobound} with $p=1$:
\bee
\int \frac{|\nabla^3 \phi_\e|^4}{Q}&\lesssim & \int |\nabla^2 \phi_\e|^4+\Sigma_{i=1}^2\int|\nabla [(1+x_i)\nabla^2\phi_\e]|^4\\
& \lesssim &  \|\nabla^2 \phi_\e\|_{H^2}^4+\Sigma_{i=1}^2\left(\int|\nabla[(1+x_i)\nabla^2 \phi_\e]|^2\right)\left(\int |\nabla^2[(1+x_i)\nabla^2\phi_\e]|^2\right)\\
& \lesssim & \|\e\|_{H^2}^4+\Sigma_{i=1}^2\left(\|\e\|_{H^2}^2+\int|x_i|^2|\nabla^2(\nabla \phi_\e)|^2\right)\left(\|\e\|_{H^2}^2+\int|x_i|^2|\nabla^2(\nabla^2 \phi_\e)|^2\right)\\
&\lesssim & \|\e\|_{H^2}^4+\Sigma_{i=1}^2\left(\|\e\|_{H^2}^2+\int|x_i|^2|\Delta(\nabla \phi_\e)|^2\right)\left(\|\e\|_{H^2}^2+\int|x_i|^2|\Delta(\nabla^2 \phi_\e)|^2\right)\\
& = & \|\e\|_{H^2}^4+\left(\|\e\|_{H^2}^2+\int|x_i|^2|\nabla (\Delta\phi_\e)|^2\right)\left(\|\e\|_{H^2}^2+\int|x_i|^2|\nabla^2(\Delta \phi_\e)|^2\right)\\
& \lesssim &  C(M)\|\e_2\|_{L^2_Q}^4.
\eee
We then estimate using \fref{controlphin}, \fref{ceniohoeh}, \fref{boundsonehifo}, \fref{coercbase}, \fref{estlossyfield}:
\bee
\int \frac{|\nabla \cdot S_1|^2}{Q}& \lesssim & \|\e\|_{L^{\infty}}^2\int \frac{|\Delta\e|^2}{Q}+\int \frac{|\nabla \e|^4+|\nabla^3 \phi_\e|^4}{Q}+\|\nabla^2 \phi_\e\|_{L^{\infty}}^2\int\frac{|\nabla^2\e|^2}{Q}\\
& + & \int\frac{|\nabla \phi_{N(\e)}|^2}{1+r^6}+\|\e\|_{L^{\infty}}^4+\|\frac{\nabla \phi_\e}{1+r}\|_{L^{\infty}}^2\int\frac{|\nabla \e|^2}{1+r^2}\\
& + & (\|\e\|_{L^{\infty}}^2+\|\nabla^2\phi_\e\|_{L^{\infty}}^2)\int(1+r^2)|\nabla \e|^2+\|\frac{\nabla \phi_\e}{1+r}\|_{L^{\infty}}^2\int(1+r^4)|\nabla^2\e|^2\\
& \lesssim &C(M)\|\e_2\|_{L^2_Q}^4+\|\nabla \phi_{N(\e)}\|_{L^4}^2+\|\e_2\|_{L^2_Q}^2 |\log b|^C\left(\|\e_2\|_{L^2_Q}^2+b^{10}\right)\\
& \lesssim & |\log b|^{C}  \|\e_2\|_{L^2_Q}^4+b^{20}.
\eee
{\it Proof of \fref{bounfknlsbis}, \fref{bounfknls}}: From \fref{linftiyfilw}:
$$\int \frac{|S_2|^2}{Q}\lesssim \|\nabla \phi_\e\|_{L^{\infty}}^2\int\frac{|\nabla^2\e|^2}Q\lesssim \|\e_2\|_{L^2_Q}^{4-\frac1{100}}.$$
From \fref{estlossyfield}:
$$\int (1+r^2)|S_2|^2\lesssim \int (1+r^2)|\nabla^2\e|^2|\nabla \phi_\e|^2\lesssim \left\|\frac{\nabla \phi_\e}{1+r}\right\|_{L^{\infty}}^2\|\e\|_{H^2_Q}^2\lesssim |\log b|^C\|\e_2\|_{L^2_Q}^{4}+b^{20}.$$
We are now in position to obtain admissible bounds for the nonlinear terms in \fref{noumberneergyidentity}. From \fref{defdecompfojep}, \fref{mcontiniuos} and the nonlinear estimates \fref{noevneonveo}, \fref{bounfknls}:
\bee
&&|(\e_2,\mathcal M\mathcal LN(\e))|\lesssim|(\matchal M\e_2,\nabla\cdot S_1)|+|(\nabla\mathcal M\e_2\cdot S_2)|\\
& \lesssim & \|\e_2\|_{L^2_Q}\|\nabla \cdot S_1\|_{L^2_Q}+\left(\int Q|\nabla \mathcal M\e_2|^2\right)^{\frac 12}\left(\int \frac{|S_2|^2}{Q}\right)^{\frac 12}\\
& \leq &\frac{1}{100}\int Q|\nabla \mathcal M\e_2|^2+C(M)|\log b|^C \|\e_2\|_{L^2_Q}\left[\|\e_2\|^2_{L^2_Q}+b^{10}\right]+C(M)\|\e_2\|_{L^2_Q}^{4-\frac 1{100}}\\
& \leq & \frac{1}{100}\int Q|\nabla \mathcal M\eh_2|^2+C(M)\left[bb^{\frac 14}\|\e_2\|_{L^2_Q}^2+\frac{b^4}{|\log b|^2}\right].
\eee
We estimate the radiation term using Lemma \ref{lemmaradiation}:
\bee
|(\zeta_2,\mathcal M\mathcal LN(\e))|&\lesssim & |(\mathcal M\zeta_2,\nabla \cdot S_1)|+|(\nabla \mathcal M\zeta_2,S_2)|\\
& \lesssim & \frac{b^{\frac 32}}{|\log b|}\|\nabla \cdot S_1\|_{L^2_Q}+\left(\int Q|\nabla \mathcal M\zeta_2|^2\right)^{\frac 12}\left(\int \frac{|S_2|^2}{Q}\right)^{\frac 12}\\
& \lesssim & \frac{b^{\frac 32}}{|\log b|}\left[|\log b|^{C}  \|\e_2\|_{L^2_Q}^2+b^{10}\right]+\frac{b^2}{|\log b|}\|\e_2\|_{L^2_Q}^{2-\frac 1{50}}\\
& \lesssim & \frac{b^4}{|\log b|^2}.
\eee

Next, using the degeneracy \fref{gfijgbeigei} and integrating by parts:
\bee
&&b\left|\int\frac{\Lambda Q+2Q}{Q^2}\eh\mathcal L\mathcal N(\e)\right| \lesssim  b\int\frac{1}{(1+r^2)Q}\left[|\eh\nabla \cdot S_1|+\frac{|\eh||S_2|}{1+r}|+|\nabla \e||S_2|\right]\\
& \lesssim & b(\|\e\|_{L^2}+\|\zeta\|_{L^2})\left[\|\nabla \cdot S_1\|_{L^2_Q}+\|(1+r)S_2\|_{L^2}\right]\\
& + & b(\|(1+r)\nabla \e\|_{L^2}+\|(1+r)\nabla\zeta\|_{L^2})\|(1+r)S_2\|_{L^2}\\
& \lesssim & b\left(\|\e_2\|_{L^2_Q}+\frac{b}{|\log b|}\right)\left[|\log b|^C\|\e_2\|^2_{L^2_Q}+b^{10}\right]\lesssim \frac{b}{|\log b|}\left[ \|\e_2\|_{L^2_Q}^2+b^4\right].
\eee
Using \fref{controlphin}, \fref{controleun}:
\bee
&&b\left|\int \eh_1\cdot\nabla \phi_{N(\e)}\right|\\
&\lesssim& b\left(\int (1+r^2)|\e_1|^2+\int |\zeta|^2+\int(1+r^2)|\nabla \zeta|^2+\int\frac{|\nabla \phi_\zeta|^2}{1+r^2}\right)^{\frac 12}\left(\int \frac{|\nabla \phi_{N(\e)}|^2}{1+r^2}\right)^{\frac 12}\\
& \lesssim & b \left[C(M)\|\e_2\|_{L^2_Q}+\frac{b}{|\log b|}\right]\|\nabla\phi_{N(\e)}\|_{L^4}\lesssim b|\log b|^C \left[C(M)\|\e_2\|_{L^2_Q}+\frac{b}{|\log b|}\right]\left[\|\e_2\|_{L^2_Q}^2+b^{10}\right]\\
& \lesssim & \frac{b}{|\log b|}\left[ \|\e_2\|_{L^2_Q}^2+b^4\right]
\eee
Using \fref{controlphin}, \fref{cneoeoufoef}, \fref{bounfknls} and \fref{interpolationfield}:
\bee
b\left|\int\nabla \phi_{\eh} \cdot Q\nabla \mathcal M\mathcal N(\e)\right|& \lesssim &b \left(\int \frac{|\nabla \phi_{\e}|^2}{1+r^2}+\int \frac{|\nabla \phi_{\zeta}|^2}{1+r^2}\right)^{\frac 12}\left(\int (1+r^2)(|S_1|^2+|S_2|^2)\right)^{\frac 12}\\
& \lesssim &b\left(\|\e_2\|_{L^2_Q}+\frac{b}{|\log b|}\right)|\log b|^C\left[[\|\e_2\|_{L^2_Q}^2+b^{10}\right]\\
& \lesssim & \frac{b}{|\log b|}\left[ \|\e_2\|_{L^2_Q}^2+b^4\right].
\eee

 {\bf step 11} Small linear terms. We estimate the contribution of the small linear error $\Theta_b(\e)$ given by \fref{tbe}. We expand: $$\Theta_b(\e)=\nabla \e\cdot\nabla \phi_{\qbt-Q}+2(\qbt-Q)\e+\nabla (\qbt-Q)\cdot\nabla \phi_\e.$$
 We claim the bounds:
 \be
 \label{nekoneiohnneo|}
 \int |\Theta_b(\e)|^2\lesssim b^2C(M)\|\e_2\|_{L^2_Q}^2,
 \ee
 \be
 \label{nekoneiohnneo|bis}
 \int \frac{|\nabla\phi_{\Theta_b(\e)}|^2}{1+r^2}\lesssim b^2C(M)\|\e_2\|_{L^2_Q}^2,
 \ee
 \be
 \label{njibvbbvirbeo}
 \int Q|\nabla \mathcal M\Theta_b(\e)|^2\lesssim b^2C(M)\|\e_2\|_{L^2_Q}^2.
 \ee
 {\it Proof of \fref{nekoneiohnneo|}, \fref{nekoneiohnneo|bis}, \fref{njibvbbvirbeo}}: We recall the pointwise bounds \fref{poietpoeutone}, \fref{poietpoeuttwo}:
 $$ |\Theta_b(\e)|\lesssim b\left[\frac{1+|\log r|}{1+r} |\nabla \e|+\frac{|\e|}{1+r^2}+\frac{|\nabla \phi_\e|}{1+r^3}\right],
 $$
 $$
 |\nabla\Theta_b(\e)|\lesssim  b\left[\frac{1+|\log r|}{1+r} |\nabla^2 \e|+\frac{1+|\log r|}{1+r^2}|\nabla \e|+\frac{|\nabla^2 \phi_\e|+|\e|}{1+r^3}+\frac{|\nabla \phi_\e|}{1+r^4}\right].
 $$
 The bound \fref{nekoneiohnneo|} now follows from Proposition \ref{bootbound}. We estimate using \fref{bootbound}, HLS and \fref{poietpoeutone}:
\bea
\label{nkononeovn}
\nonumber \int \frac{|\nabla \phi_{\Theta_b(\e)}|^2}{1+r^2}&\lesssim &\|\nabla \phi_{\Theta_b(\e)}\|_{L^4}^2\lesssim \|\Theta_b(\e)\|_{L^{\frac 43}}^2\\
& \lesssim & C(M)b^2\|\e_2\|^2_{L^2_Q}
\eea
and \fref{nekoneiohnneo|bis} is proved. We then compute like for \fref{defdecompfojep}:
$$Q\nabla \mathcal M\Theta_b(\e)=\nabla \Theta_b(\e)+Q\nabla \phi_{\Theta_b(\e)}+\Theta_b(\e)\nabla \phi_Q
$$
and estimate using \fref{poietpoeutone}, \fref{poietpoeuttwo}, \fref{nkononeovn} and Proposition \ref{bootbound}:
\bee
\int Q|\nabla \mathcal M\Theta_b(\e)|^2&\lesssim & b^2\int (1+r^4)\left[\frac{1+|\log r|^2}{1+r^2} |\nabla^2 \e|^2+\frac{1+|\log r|^2}{1+r^4}|\nabla \e|^2\right.\\
& + & \left.\frac{|\nabla^2 \phi_\e|^2+|\e|^2}{1+r^6}+\frac{|\nabla\phi_{\e}|^2}{1+r^8}+\frac{|\nabla \phi_{\Theta_b(\e)}|^2}{1+r^8}\right]\\
& \lesssim & b^2C(M)\|\e_2\|_{L^2_Q}^2,
\eee
this is \fref{njibvbbvirbeo}.\\
We now estimate all corresponding terms in \fref{noumberneergyidentity}. From \fref{njibvbbvirbeo}, \fref{bootsmallh2qboot}:
\bee
\left|(\eh_2,\mathcal M\mathcal L\Theta_b(\e))\right| & = & \left|\int Q\nabla \cdot\mathcal M\eh_2\cdot\nabla \mathcal M\Theta_b(\e)\right|\lesssim \frac1{100}\int Q|\nabla \mathcal M\eh_2|^2+C\int Q|\nabla \mathcal M\Theta_b(\e)|^2\\
& \leq & \frac1{100}\int Q|\nabla \mathcal M\eh_2|^2+C(M)\frac{b^4}{|\log b|^2}.
\eee
Next, integrating by parts:
\bee
&&b\left|\int\frac{\Lambda Q+2Q}{Q^2}\eh\matchal L\Theta_b(\e)\right| \lesssim  b\int(1+r^2)\left(|\nabla \eh|+\frac{|\eh|}{1+r}\right)Q|\nabla \mathcal M\Theta_b(\e)|\\
& \lesssim & b\left(\int |\nabla \eh|^2+\frac{|\eh|^2}{1+r^2}\right)^{\frac 12}\left(\int Q|\nabla \mathcal M\Theta_b(\e)|^2\right)^{\frac 12}\\
& \lesssim & b\left[C(M)\|\e_2\|_{L^2_Q}+\frac{b}{|\log b|}\right]b\|\e_2\|_{L^2_Q} \lesssim  b\frac{b^{\frac 32}}{|\log b|}\|\e_2\|_{L^2_Q}+\frac{b^4}{|\log b|^2}.
\eee
Similarily, using \fref{nekoneiohnneo|}:
\bee
b\left|\int\frac{\Lambda Q+2Q}{Q^2}\eh_2\Theta_b(\e)\right|&\lesssim &b\|\eh_2\|_{L^2_Q}\|\Theta_b(\e)\|_{L^2}\lesssim b\left(\|\e_2\|_{L^2_Q}+\frac{b}{|\log b|}\right)bC(M)\|\e_2\|_{L^2_Q}\\
& \lesssim &  b\frac{b^{\frac 32}}{|\log b|}\|\e_2\|_{L^2_Q}+\frac{b^4}{|\log b|^2}.
\eee
Using \fref{nekoneiohnneo|bis}, \fref{controleun}:
\bee
b\left|\int \eh_1\cdot\nabla \phi_{\Theta_b(\e)}\right|&\lesssim &b\left(\int (1+r^2)|\eh_1|^2\right)^{\frac 12}\left(\int \frac{|\nabla \phi_{\Theta_b(\e)}|^2}{1+r^2}\right)^{\frac12}\\
& \lesssim & bC(M)\left(\|\e_2\|_{L^2_Q}+\frac{b}{|\log b|}\right)b\|\e_2\|_{L^2_Q}\lesssim  b\frac{b^{\frac 32}}{|\log b|}\|\e_2\|_{L^2_Q}+\frac{b^4}{|\log b|^2},
\eee
\bee
b\left|\int \nabla \phi_{\eh}\cdot Q\nabla \matchal M\Theta_b(\e)\right|& \lesssim & b\left(\int \frac{|\nabla \phi_{\eh}|^2}{1+r^4}\right)^{\frac 12}\left(\int Q|\nabla \mathcal M\Theta_b(\e)|^2\right)^{\frac 12}\\
& \lesssim & bC(M)\left(\|\e_2\|_{L^2_Q}+\frac{b}{|\log b|}\right)b\|\e_2\|_{L^2_Q}\\
& \lesssim & b\frac{b^{\frac 32}}{|\log b|}\|\e_2\|_{L^2_Q}+\frac{b^4}{|\log b|^2}.
\eee

{\bf step 12} Conclusion. Injecting the collection of bounds obtained in steps 4 through 8 into \fref{noumberneergyidentity} yields the following preliminary estimate:
\bee
\nonumber &&\frac 12\frac{d}{dt}\left\{\frac{(\mathcal M \eh_2,\eh_2)}{\l^4}+O\left(\frac{C(M)b^3}{\l^4|\log b|^2}\right)\right\}\\
& \leq & -\frac12\int Q_{\lambda}|\nabla \mathcal M_{\lambda}\wh_2|^2+\frac{bC(M)}{\l^6}\left[\frac{b^{\frac32}}{|\log b|}\|\e_2\|_{L^2_Q}+\frac{b^3}{|\log b|^2}\right]+\frac{\hb}{\l^6}(\mathcal M\eh_2,\eh_2).
\eee
We now make an essential of the precise numerology to treat the remaining $(\mathcal M\eh_2,\eh_2)$ term which has the wrong sign. We multiply this identity by $\l^2$ and obtain using the rough bound from \fref{cnkonenoeoeoi3poi}, \fref{boundbootbbhat}: 
\be
\label{nekoneoneneonveo}
\left|\lsl+\hb\right|\lesssim \frac{b}{|\log b|}
\ee
 the control:
\bea
\label{nvonveonveonenoen}
 &&\frac 12\frac{d}{dt}\left\{\frac{(\mathcal M \eh_2,\eh_2)}{\l^2}+O\left(\frac{C(M)b^3}{\l^2|\log b|^2}\right)\right\}\\
\nonumber& \leq & \frac{bC(M)}{\l^4}\left[\frac{b^{\frac32}}{|\log b|}\|\e_2\|_{L^2_Q}+\frac{b^3}{|\log b|^2}\right]\\
\nonumber& + & \frac{1}{\l^4}\left|\lsl+\hb\right||(\mathcal M\eh_2,\eh_2)|+\frac{1}{\l^4}\left|\lsl\right|\frac{C(M)b^3}{|\log b|^2}.
\eea
We now develop the quadratic term and estimate using \fref{degenrateoneone}, \fref{noenoeno}:
\bee
(\mathcal M\eh_2,\eh_2)& = & (\matchal M\e_2,\e_2)+2(\mathcal M\zeta_2,\e_2)+(\mathcal M\zeta_2,\zeta_2)\\
& = &  (\matchal M\e_2,\e_2)+O\left(\frac{b^{\frac32}}{|\log b|}\|\e_2\|_{L^2_Q}+\frac{b^3}{|\log b|^2}\right)
\eee
 which implies the upper bound $$|(\mathcal M\eh_2,\eh_2)|\lesssim \|\e_2\|_{L^2_Q}^2+\frac{b^3}{|\log b|^2}.$$ Injecting these bounds together with \fref{nekoneoneneonveo}, \fref{bootsmallh2q} into \fref{nvonveonveonenoen} yields \fref{eetfonafmetea} and concludes the proof of Proposition \ref{htwoqmonton}.


\section{Sharp description of the singularity formation}
\label{proofthmmai}


We are now in position to conclude the proof of the bootstrap Proposition \ref{propboot} from which Theorem \ref{thmmain} easily follows.


\subsection{Closing the bootstrap}


We now conclude the proof of Proposition \ref{propboot}.

\begin{proof}[Proof of Proposition \ref{propboot}] The $L^1$ bound \fref{bootsmallloneboot} follows from \fref{lonebound}. The pointwise upper bound on b \fref{positivboot} follows from \fref{poitnzeiboud} which implies $\bh_s<0$ and the conclusion follows from \fref{boundbootbbhat}. The lower bound $b>0$ follows from the bootstrap bound \fref{bootsmallh2q}. Indeed, if $b(s^*)=0$, then $\e(s^*)=0$ from \fref{bootsmallh2q} and then $\int u(s^*)=\int Q=\int u_0$ by conservation mass, and a contradiction follows from \fref{intiialmass}.\\
It remains to close the bound \fref{bootsmallh2qboot} which is the heart of the analysis, and this follows from the monotonicity formula \fref{htwoqmonton}. Indeed, we rewrite \fref{eetfonafmetea} using \fref{bootsmallh2q}:
\be
\label{nknvkonbrono}
\frac{d}{dt}\left\{\frac{(\mathcal M \e_2,\e_2)}{\l^2}+O\left(C(M)\sqrt{K^*}\frac{b^3}{\l^2|\log b|^2}\right)\right\}\lesssim  \sqrt{K^*}\frac{b^4}{\l^4|\log b|^2}.
\ee
In order to integrate this between $t=0$ and $t^*$, we first observe from \fref{boundbootbbhat}, \fref{pointziender}, \fref{estlambda} the bounds 
\be
\label{vnonvoneo}
|\hb_s|\lesssim \frac{b^2}{|\log b|}, \ \ |\l\l_t+b|\lesssim C(M)\frac{b^2}{|\log b|}
\ee from which:
\bee
\int_0^{t^*}\frac{b^4}{\l^4|\log b|^2}dt& = & \int -\l_t \frac{\bh^3}{\l^3|\log \bh|^2}dt+O\left(\int_0^t\frac{b^5}{\l^4|\log b|^5}dt\right)\\
& = & \left[\frac{\bh^3}{2\l^2|\log \bh|^2}\right]_0^t-\int_0^t\frac{1}{2\l^4}\frac{d}{ds}\left[\frac{\bh^3}{|\log \bh|^2}\right]dt+O\left(\int_0^t\frac{b^5}{\l^4|\log b|^5}dt\right)\\
& = &  \left[\frac{\bh^3}{2\l^2|\log \bh|^2}\right]_0^t+O\left(\int_0^t\frac{b^4}{\l^4|\log b|^3}dt\right)
\eee
and hence using the a priori smallness of $b$: $$\l^2(t)\int_0^{t^*}\frac{b^4}{\l^4|\log b|^2}dt\lesssim \frac{b^3(t)}{|\log b(t)|^2}+\left(\frac{\l(t)}{\l(0)}\right)^2\frac{b^3(0)}{|\log b(0)|^2}$$ Integrating \fref{nknvkonbrono} in time, we conclude using \fref{coerclwotht}: 
\bea
\label{cneonneeonoe}
\delta(M)\|\e_2(t)\|^2_{L^2_Q}& \lesssim & (\mathcal M\e_2(t),\e_2(t))\\
\nonumber & \lesssim & C(M)\left\{\frac{\l^2(t)}{\l^2(0)}\left[\|\e_2(0)\|_{L^2_Q}^2+\sqrt{K^*}\frac{b^3(0)}{|\log b(0)|^2}\right]+\sqrt{K^*}\frac{b^3(t)}{|\log b(t)|^2}\right\}
\eea
for some small enough universal constants $\delta(M),C(M)>0$ independent of $K^*$. Moreover, \fref{vnonvoneo} implies 
\be
\label{neiovneneonevoneo}
\frac{d}{ds}\left\{\frac{\bh^3}{\l^2|\log \bh|^2}\right\}>0
\ee
and thus \fref{cneonneeonoe} and the initialization \fref{initialsmalneesb} ensure: 
\bee
\delta(M)\|\e_2(t)\|^2_{L^2_Q}& \lesssim & C(M)\sqrt{K^*}\left[\l^2(t)\frac{b^3(0)}{\l^2(0)|\log b(0)|^2}+\frac{b^3(t)}{|\log b(t)|^2}\right]\\
& \lesssim &  C(M)\sqrt{K^*}\left[\l^2(t)\frac{\bh^3(0)}{\l^2(0)|\log\bh(0)|^2}+\frac{b^3(t)}{|\log b(t)|^2}\right]\\
& \lesssim & C(M)\sqrt{K^*}\frac{\bh^3(t)}{|\log \bh(t)|^2}\lesssim  C(M)\sqrt{K^*}\frac{b^3(t)}{|\log b(t)|^2}
\eee
 and \fref{bootsmallh2qboot} follows for $K^*=K^*(M)$ large enough. This concludes the proof of Proposition \ref{propboot}.
  \end{proof}
  

\subsection{Sharp control of the singularity formation}


 We are now in position to conclude the proof of Theorem \ref{thmmain}.
 
 \begin{proof}[Proof of Theorem \ref{thmmain}]
 
 {\bf step 1} Let $u_0\in \mathcal O$ and $u\in \mathcal C([0,T),\mathcal E)$ be the corresponding solution to \fref{kps} with lifetime $0<T\leq +\infty$, then the estimates of Proposition \ref{propboot} hold on $[0,T)$. Observe from \fref{neiovneneonevoneo} the bound $$\l^2\lesssim b^3$$ from which using \fref{vnonvoneo} $$-\l\l_t\gtrsim b\gtrsim C(u_0)\l^{\frac 23} \ \ \mbox{and thus} \ \ -(\l^{\frac 43})_t\gtrsim C(u_0)>0$$ implies that $\l(t)$ touches zero in some finite time $0<T_0<+\infty$. Note that the bounds of Proposition \ref{propboot} injected into the decomposition \fref{decompe} ensure $$\|\e(t)\|_{H^2}\ll1\ \ \mbox{for}\ \  0\leq t<T_0 \  \ \mbox{and}\ \  \lim_{t\uparrow T_0}\|u(t)\|_{H^2}=+\infty$$ and thus from standard Cauchy theory, the solution blows up at $T=T_0<+\infty$.  Moreover, we obtain from \fref{vnonvoneo} the rough bound: $$|\lambda\lambda_t|=\left|\lsl\right|\lesssim 1 \ \ \mbox{and thus}\ \ \lambda(t)\lesssim \sqrt{T-t},$$ and hence from \fref{defst}: 
\be
\label{glovaltime}
s(t)\to+\infty\ \ \mbox{as}\ \ t\to T.
\ee
 
 {\bf step 2} Blow up speed. We now derive the sharp asymptotics at blow up time by reintegrating the modulation equations for $(b,\l)$ in the vicinity of blow up time or equivalently as $s\to \infty$. We estimate from \fref{poitnzeiboud}, \fref{bootsmallh2qboot}: 
 \be
 \label{eqsdbivbbie}
 \left|\hat{b}_s+\frac{2\bh^2}{|\log \bh|}\right|\lesssim\frac{\bh^2}{|\log \bh|^2}.
 \ee
 We then argue as in \cite{RaphRod}. We multiply \fref{eqsdbivbbie} by $\frac{|\log \hb|}{\hb^2}$ and obtain: $$\frac{\bh_s\log \bh}{\bh^2}=-2+O\left(\frac{1}{|\log \bh|}\right).$$ We use $$\left(\frac{\log t}{t}+\frac{1}{t}\right)'=-\frac{\log t}{t^2}$$ to conclude after integration: $$-\frac{\log \bh+1}{\bh}=2s+O\left(\int_0^s\frac{d\sigma}{|\log \bh|}\right)$$ and thus $$\hb(s)=\frac{\log s}{2s}\left(1+o(1)\right), \ \ \log \bh=\log\log s-\log s +O(1).$$ This finally yields the aysmptotic development near blow up time:
 $$\hb(s)=\frac{1}{2s}\left(\log s-\log \log s\right)+O\left(\frac{1}{s}\right).$$ We now rewrite the modulation equations for $\l$ using \fref{vnonvoneo}, \fref{boundbootbbhat}: 
 \be
 \label{cbeiheovejiugveuig}
 -\lsl=\hb+O\left(\frac{b}{|\log b|}\right)=\frac{1}{2s}\left(\log s-\log \log s\right)+O\left(\frac{1}{s}\right)
 \ee
  which time integration yields:
$$-\log \l=\frac14\left[(\log s)^2-2\log s \log \log s\right]+O(\log s)=\frac{(\log s)^2}4\left[1-\frac{2\log\log s}{\log s}+O\left(\frac1{\log s}\right)\right].$$ In particular, $$\sqrt{|\log \l|}=\frac{\log s}2\left[1-\frac{\log\log s}{\log s}+O\left(\frac1{\log s}\right)\right]$$ which also implies:
$$e^{2\sqrt{|\log \l|}+O(1)}=\frac{s}{\log s}, \ \ s=\sqrt{|\log \l|}e^{2\sqrt{|\log \l|}+O(1)}.$$ We use these relations to rewrite the modulation equation \fref{cbeiheovejiugveuig}: $$-s\lsl=-\sqrt{|\log \l|}e^{2\sqrt{|\log \l|}+O(1)}\l\l_t=\sqrt{|\log \l|}+O(1)$$ and thus $$-\l\l_te^{2\sqrt{|\log \l|}}=e^{O(1)}.$$ The time integration with boundary condition $\l(T)=0$ yields $$\lambda(t)=\sqrt{T-t}e^{-\sqrt{\frac{|\log (T-t)|}{2}}+O(1)}\ \ \mbox{as}\ \ t\to T,$$ this is \fref{blkpeoghenepojg}.\\
Observe now that the chain of above estimates ensures in particular $$b(t)\to 0\ \ \mbox{as}\ \ t\to T$$ which easily implies using Proposition \ref{proploc}: $$\|Q-\qbt\|_{H^2_Q}\to 0\ \ \mbox{as}\ \ t\to T,$$ and the strong convergence \fref{boievbebeo} now follows from \fref{bootsmallh2qboot}.\\
This concludes the proof of Theorem \ref{thmmain}.
 
 \end{proof}


\begin{appendix}


\section{Estimates for the Poisson field}


This appendix is devoted to the derivation of linear estimates, in particular for for the Poisson field in $H^2_Q$ and $\mathcal E$. We start with weighted $H^2$ type bounds:

\begin{lemma}[$H^2$ bound]
\label{lemmahtwobound}
Let $p=1,2$, then for all $u\in \mathcal D(\R^2)$,
\be
\label{estwieght}
\int |x|^{2p}|\nabla ^2u|^2\lesssim \int |x|^{2p}|\Delta u|^2+\int |x|^{2p-2}|\nabla u|^2.
\ee
\end{lemma}

\begin{proof}[Proof of Lemma \ref{lemmahtwobound}] We integrate by parts to compute:
\bee
&&\int x_1^{2p}(\pa_{11}u+\pa_{22}u)^2 =  \int x_1^{2p}\left[(\pa_{11}u)^2+(\pa_{22}u)^2\right]+2\int x_1^{2p}\pa_{11}u\pa_{22}u\\
& = & \int x_1^{2p}\left[(\pa_{11}u)^2+(\pa_{22}u)^2\right] -2\int x_1^{2p}\pa_2u\pa_{211}u\\
& = &  \int x_1^{2p}\left[(\pa_{11}u)^2+(\pa_{22}u)^2\right] +2\int \pa_{12}u\left[x_1^{2p}\pa_{12}u+2px_1^{2p-1}\pa_{2}u\right]\\
& = & \int x_1^{2p}\left[(\pa_{11}u)^2+(\pa_{22}u)^2+2(\pa_{12}u)^2\right] -2p(2p-1)\int x_1^{2p-2}(\pa_2u)^2
\eee
and similarly with $x_2$, and \fref{estwieght} follows.
\end{proof}

We now turn to the linear control of the Poisson field:

\begin{lemma}[Estimates for the Poisson field]
\label{lemmainterpolation}
There holds the bounds on the Poisson field:\\
(i) General $L^2_Q$ bounds:\be
\label{linftybound}
\|(1+|\log r|)u\|_{L^1}+\|\phi_u\|_{L^{\infty}(r\leq 1)}+\left\|\frac{|\phi_u|}{1+|\log r|}\right\|_{L^{\infty}(r\geq 1)}\lesssim \|u\|_{L^2_Q},
\ee
\be
\label{lfourbound}
\|\nabla \phi_u\|_{L^4}\lesssim \|u\|_{L^2_Q}.
\ee
(ii) Improved $L^2_Q$ bound: if moreover $\int u=0$, then:
\be
\label{improvedlinftybis}
\|\phi_u\|_{L^{\infty}}\lesssim \|u\|_{L^2_Q},
\ee
\be
\label{improvedlinfty}
\int|\nabla \phi_u|^2=-\int u\phi_u\lesssim \|u\|^2_{L^2_Q}.
\ee
(iii) Energy bound:
\be
\label{estcahmos}
\forall 1\leq p<2, \ \ \|\nabla \phi_u\|_{L^{\infty}}\leq C_p\|u\|^{1-\frac p2}_{L^{\infty}}\|u\|_{L^2}^{p-1}\|u\|^{1-\frac p2}_{L^1}.
\ee
(iv) Decay in the energy space: $\forall 0\leq \alpha<1$,
\be
\label{decaylinfty}
 |u(x)|\leq C_{\alpha} \frac{\|u\|_{\mathcal E}}{1+|x|^\alpha},
\ee
\be
\label{deactfiel}
|\nabla \phi_u(x)|\leq C_{\alpha} \frac{\|u\|_{\mathcal E}}{1+|x|^{\frac\alpha 2}}.
\ee
\end{lemma}

\begin{proof}[Proof of Lemma \ref{lemmainterpolation}]:
{\em Proof of (i)-(ii)}: By Cauchy-Schwarz:
$$\int(1+|\log r|)|u(y)|\lesssim \|u\|_{L^2_Q}\left(\int (1+|\log r|)^2Q\right)^{\frac 12}\lesssim \|u\|_{L^2_Q}.$$
For $|x|\leq 1$, we estimate using Cauchy-Schwarz: 
\bee
 |\phi_{u}(x)|&\lesssim &\int_{|x-y|\leq 1}|\log(|x-y|)|u(y)|dy+ \int_{|x-y|\geq 1}|\log|x-y|||u(y)|dy\\
 & \lesssim & \|u\|_{L^2_Q}+\int (1+|\log |y||)|u(y)|\lesssim  \|u\|_{L^2_Q}.
 \eee
 For $|x|\geq 1$, we rewrite
 \bee
 && \left|\phi_u(x)-\frac{\log |x|}{2\pi}\int u\right|  \lesssim  \int |\log (\frac{|x-y|}{|x|})||u(y)|dy\\
  &  \lesssim& \int_{|x-y|\geq \frac{|x|}{2}} |\log (\frac{|x-y|}{|x|})||u(y)|dy+\int_{|x-y|\leq \frac{|x|}{2}} |\log (\frac{|x-y|}{|x|})||u(y)|dy.
 \eee
 For the outer term, we estimate:
 \bee
 \int_{|x-y|\geq \frac{|x|}{2}} |\log (\frac{|x-y|}{|x|})||u(y)|dy&\lesssim& \int_{|x-y|\geq \frac{|x|}{2}} \left(\frac{|x-y|}{|x|}\right)^{\frac 34}|u(y)|dy\lesssim \int \left(1+\frac{|y|^{\frac34}}{|x|^{\frac 34}}\right)|u(y)|dy\\
 & \lesssim & \|u\|_{L^1}+ \frac{\|u\|_{L^2_Q}}{|x|^{\frac 34}}\lesssim \|u\|_{L^2_Q}.
 \eee
 On the singularity, we estimate using a simple change of variables:
 \bee
&& \int_{|x-y|\leq \frac{|x|}{2}} |\log (\frac{|x-y|}{|x|})||u(y)|dy \lesssim  \left(\int_{|x-y|\leq \frac{|x|}{2}}|\log (\frac{|x-y|}{|x|})|^2dy\right)^{\frac 12}\left( \int_{|x-y|\leq \frac{|x|}{2}} |u(y)|^2dy\right)^{\frac 12}\\
 & \lesssim & \left(\|x\|^2\int_{|z|\leq \frac 12}(\log |z|)^2dz\int_{|x-y|\leq \frac{|x|}{2}}|u(y)|^2dy\right)^{\frac 12}\lesssim \left(\int |y|^2|u(y)|^2dy\right)^{\frac12}\\
 & \lesssim & \|u\|_{L^2_Q}
 \eee
 and this concludes the proof of \fref{linftybound} and \fref{improvedlinftybis}.\\
 We then estimate from the 2 dimensional Hardy-Littlewood-Sobolev and H\"older: $$\|\nabla \phi_u\|_{L^4}\lesssim \|\frac{1}{|x|}\star u\|_{L^4}\lesssim \|u\|_{L^{\frac 43}}\lesssim \|u\|_{L^2_Q}\|Q\|_{L^2}^{\frac 12}$$ and \fref{lfourbound} is proved.  Let a smooth cut off function $\chi(x)=1$ for $|x|\leq 1$, $\chi(x)=0$ for $|x|\geq 2$, and $\chi_R(x)=\chi(\frac xR)$, then from \fref{improvedlinftybis}: 
 \bea
 \label{cneiocneoen}
 \int\chi_R|\nabla \phi_{u}|^2 & = & \frac 12\int\Delta \chi_R\phi_{u}^2-\int\chi_Ru\phi_{u}\\
 \nonumber & \lesssim &
 \|\phi_u\|^2_{L^{\infty}}+\|\phi_u\|_{L^{\infty}}\|u\|_{L^1}\lesssim \|u\|_{L^2_Q}^2
 \eea with constants independent of $R>0$, and hence 
 \be
 \label{cnenceneone}
 \int |\nabla \phi_u|^2<+\infty.
 \ee Moreover, integrating by parts: $$\int_{|x|\leq R}|\nabla \phi_u|^2=R\int_0^{2\pi}\phi_u\pa_r\phi_ud\theta-\int_{|x|\leq R}u\phi_u.$$ From \fref{cnenceneone}, we can find a sequence $R_n\to \infty$ such that $$R^2_n\int_0^{2\pi}|\pa_r\phi_{u}|^2\to 0$$ and then from \fref{improvedlinftybis}: $$R_n\left|\int_0^{2\pi}\phi_u\pa_r\phi_u\right|\lesssim\|\phi_u\|_{L^{\infty} }\left(R^2_n\int_0^{2\pi}|\pa_r\phi_{u}|^2\right)^{\frac 12}\to 0$$ and thus $$\int|\nabla \phi_u|^2=\lim_{R_n\to +\infty}\int_{|x|\leq R_n}|\nabla \phi_u|^2=-\int u\phi_u.$$ The estimate \fref{improvedlinfty} now follows from \fref{improvedlinftybis}.\\
 {\em Proof of (iii)}: Let $1\leq p<2$, we estimate in brute force:
  \bee
  |\nabla \phi_u| & \lesssim & \frac{1}{|x|}\star |u|\lesssim \int_{|x-y|\leq R}\frac{|u(y)|}{|x-y|}dy+\int_{|x-y|\geq R}\frac{|u(y)|}{|x-y|}dy\\
  & \lesssim & \|u\|_{L^{\infty}}\int_{|z|\leq R}\frac{1}{|z|}+\|u\|_{L^{p}}\left(\int_{|z|\geq R}\frac{1}{|z|^{p'}}\right)^{\frac 1{p'}} \lesssim  R\|u\|_{L^{\infty}}+\frac{\|u\|_{L^p}}{R^{1-\frac2{p'}}}.
  \eee
  We optimize in $R$ and interpolate:
  $$  \|\nabla\phi_u\|_{L^{\infty}}\leq C_p \|u\|_{L^p}^{\frac p2}\|u\|_{L^{\infty}}^{1-\frac p2}\lesssim C_p\|u\|^{1-\frac p2}_{L^{\infty}}\|u\|_{L^2}^{p-1}\|u\|^{1-\frac p2}_{L^1},
 $$
 this is \fref{estcahmos}.\\
  {\it Proof of (iv)}: By density, it suffices to prove \fref{decaylinfty} for $u\in \mathcal D(\R^2)$. Let $(v_{i,j}=x_i\pa_ju)_{1\leq i,j\leq 2}$, we estimate from \fref{estwieght} with $p=2$: $$\int |v_{i,j}|^2+\int|\nabla v_{i,j}|^2\lesssim \int (1+|x|^2)|\nabla u|^2+\int (1+|x|^4)\left[(\pa_{11}u)^2+(\pa_{22}u)^2+(\pa_{12}u)^2\right]\lesssim \|u\|_{H^2_Q}^2$$ and thus from Sobolev: 
\be
\label{efneonvnevo}
\forall p>2, \ \ \|v_{i,j}\|_{L^p}\lesssim \|v_{i,j}\|_{H^1}\lesssim \|u\|_{H^2_Q}.
\ee
We now recall the standard Sobolev bound, see for example \cite{Brezis}:
$$\ \ \forall p>2, \ \ \forall f\in \mathcal D(\R^2), \ \ |f(x)-f(y)|\lesssim |x-y|^{1-\frac{2}{p}}\|\nabla f\|_{L^p}.$$ We may find $|a|\leq 1$ such that $$|f(a)|\lesssim \left(\int_{|y|\leq 1}|f(y)|^2dy\right)^{\frac 12}$$ and hence the growth estimate: $$|f(x)|\lesssim \left(\int_{|y|\leq 1}|f(y)|^2dy\right)^{\frac 12}+|x|^{1-\frac{2}{p}}\|\nabla f\|_{L^p}.$$ We apply this to $f_{i}=x_iu$ and conclude from \fref{efneonvnevo}, \fref{linftyboundbis}: $\forall p>2$ and $i=1,2$:
\bee
|x_iu| & \lesssim & \left(\int_{|y|\leq 1}|x_iu|^2dy\right)^{\frac 12}+|x|^{1-\frac{2}{p}}\|\nabla (x_iu)\|_{L^p}\\
& \lesssim & \|u\|_{L^{\infty}}+|x|^{1-\frac{2}{p}}\left[\|v_{i,j}\|_{L^p}+\|u\|_{L^p}\right]\lesssim (1+|x|^{1-\frac{2}{p}})\|u\|_{\mathcal E}
\eee
and hence the decay: $$|u(x)|\lesssim\frac{ \|u\|_{\mathcal E}}{1+|x|^{\frac 2p}}$$ which yields \fref{decaylinfty}.\\
Let $0\leq \alpha<1$ and $|x|\gg 1$, we estimate the Poisson field in brute force using \fref{decaylinfty}:
\bee
|\nabla \phi_u(x)|& \lesssim & \int \frac{|u(y)|}{|x-y|}=\int_{|x-y|>|x|^{\frac{\alpha}{2}}} \frac{|u(y)|}{|x-y|}+\int_{|x-y|<|x|^{\frac{\alpha}{2}}} \frac{|u(y)|}{|x-y|}\\
& \lesssim & \frac{\|u\|_{L^1}}{|x|^{\frac{\alpha}{2}}}+\int_{|y|\geq \frac{|x|}{2}, \ \ |x-y|<|x|^{\frac{\alpha}{2}}}\frac{|u(y)|}{|x-y|}\lesssim  \frac{\|u\|_{L^1}}{|x|^{\frac{\alpha}{2}}}+\frac{\|u\|_{\mathcal E}}{1+|x|^{\alpha}}\int_{|z|\leq |x|^{\frac{\alpha}{2}}}\frac{dz}{|z|}\\
& \lesssim & \frac{\|u\|_{\mathcal E}}{1+|x|^{\frac{\alpha}{2}}}
\eee 
and \fref{deactfiel} is proved.
 \end{proof}


\section{Hardy bounds}


We recall some standard weighted Hardy inequalities:

\begin{lemma}[Weighted Hardy inequality]
\label{weightedhardy}
There holds the Hardy bounds: 
\be
\label{hardyboundbis}
\forall \alpha>-2, \ \  \int r^{\alpha+2}|\pa_ru|^2\geq \frac{(2+\alpha)^2}{4} \int r^{\alpha}u^2,
\ee
\be
\label{nenoeneo}
\int |\Delta \phi|^2\gtrsim \int \frac{|\nabla \phi|^2}{r^2(1+|\log r|)^2}-\int\frac{|\nabla \phi|^2}{1+r^4},
\ee
\be
\label{nenoeneobis}
 \int \frac{|\nabla \phi|^2}{r^2(1+|\log r|)^2}\gtrsim \int\frac{\phi^2}{(1+r^4)(1+|\log r|)^2}-\int \frac{\phi^2}{1+r^6}.
\ee

\end{lemma}

{\it Proof of Lemma \ref{weightedhardy}} Let $u\in \mathcal C^{\infty}_c(\R^2)$. We integrate by parts to estimate:
$$\frac{\alpha+2}{2}\int r^{\alpha} u^2=-\int r^{\alpha+1}u\pa_ru\leq \left(\int r^{\alpha}u^2\right)^{\frac 12}\left(\int r^{\alpha+2}(\pa_ru)^2\right)^{\frac 12}$$ and \fref{hardyboundbis} follows.\\ 
Let now $\phi\in \mathcal D(\R^2)$ and consider the radial continuous and piecewise $\matchal C^1$ function $$f(r)=\left\{\begin{array}{ll}\frac{1}{r(1-\log r)}\ \ \mbox{for}\  \ 0<r\leq 1,\\ \frac{1}{r(1+\log r)}\ \ \mbox{for}\  \ r\geq 1\end{array}\right.,\ \ F(x)=f(r)\frac{x}{r}$$ then $$\nabla\cdot F(x)=\left\{\begin{array}{ll}\frac{1}{r^2(1-\log r)^2}\ \ \mbox{for}\  \ 0<r< 1,\\ -\frac{1}{r^2(1+\log r)^2}\ \ \mbox{for}\  \ r> 1\end{array}\right.,$$ and thus:
\bee
\int \frac{1}{r^2(1+|\log r|)^2}|\nabla \phi|^2& \lesssim & \int_{|x|\leq 1}\nabla \cdot F |\nabla \phi|^2-\int_{|x|\geq 1}\nabla \cdot F |\nabla \phi|^2\\
& \lesssim & \int_{r=1}f|\nabla \phi|^2d\sigma+\int |f||\nabla (|\nabla \phi|^2)|\\
& \lesssim & \int_{r=1}f|\nabla \phi|^2d\sigma+\left(\int |f|^2|\nabla \phi|^2\right)^{\frac 12}\left(\int |\nabla^2\phi|^2\right)^{\frac 12}
\eee
Now by Sobolev: $$ \int_{r=1}f|\nabla \phi|^2d\sigma\lesssim \int_{\frac12\leq r\leq 2}(|\nabla \phi|^2+|\nabla ^2\phi|^2)$$ and thus 
\bee
& & \int \frac{1}{r^2(1+|\log r|)^2}|\nabla \phi|^2\\
& \lesssim & \int_{\frac12\leq r\leq 2}|\nabla \phi|^2+\int |\Delta \phi|^2+\left(\int \frac{1}{r^2(1+|\log r|)^2}|\nabla \phi|^2\right)^{\frac 12}\left(\int |\Delta\phi|^2\right)^{\frac 12}
\eee which implies \fref{nenoeneo}.\\
Similarily, for $r\geq r_0$ large enough: 
\bee
\int_{r\geq r_0}\frac{\phi^2}{(1+r^4)(1+|\log r|)^2}&\lesssim &-\int_{r\geq r_0}\nabla \cdot\left[\frac{1}{(1+r^3)(1+|\log r|)^2}\frac{y}{|y|}\right]\phi^2\\
& \lesssim & \int_{r=r_0}|\phi|^2d\sigma +\int_{r\geq r_0}\frac{|\phi||\nabla \phi|}{(1+r^3)(1+|\log r|)^2} 
\eee
and \fref{nenoeneobis} follows again from Cauchy Schwarz and Sobolev.\\
This concludes the proof of Lemma \ref{weightedhardy}.\\


\section{Interpolation bounds}
\label{appendixc}

We collect in this appendix the bootstrap bounds on $\e$ which are a consequence of the spectral estimates of Proposition \ref{interpolationhtwo} and further interpolation estimates.

\begin{proposition}[Interpolation bounds]
\label{bootbound}
For $i=0,1$:\\
{\it (i)} $H^2_Q$ bound: 
\bea
\label{coercbase}
&&\int (1+r^4)|\nabla^2\e|^2+\int(1+r^2)|\nabla \e|^2+\int \e^2\\
\nonumber  & + & \int|\Delta \phi_\e|^2+\int\frac{|\nabla \phi_\e|^2}{r^2(1+|\log r|)^2}\lesssim  C(M)\|\e_2\|_{L^2_Q}^2,
\eea
\be
\label{controleun}
\int (1+r^2)|\e_1|^2\lesssim  C(M)\|\e_2\|_{L^2_Q}^2.
\ee
{\it (ii)} $L^{\infty}$ bounds: 
\be
\label{linftiyfilw}
\forall 0<\eta\leq \frac 12, \ \ \|\nabla \phi_\e\|_{L^{\infty}}\leq C_\eta\|\e_2\|^{1-\eta}_{L^2_Q},
\ee
and $\forall 0\leq \alpha<\frac 12$, 
\be
\label{linftyboundbis}
 \|(1+|x|^{\alpha})\e\|_{L^{\infty}}+\left\|(1+|x|^{\frac \alpha2})\nabla \phi_\e(x)\right\|_{L^{\infty}}\lesssim \delta(\alpha^*).
\ee
{\it (iii)} $H^2_Q$ bound with logarithmic loss: 
\be
\label{interpolationfield}
\int(1+\log (1+r))^C\frac{|\nabla \phi_\e|^2}{r^2(1+|\log r|)^2}\lesssim |\log b|^C\left(\|\e_2\|_{L^2_Q}^2+b^{10}\right),
\ee
{\it (iv)} $L^{\infty}$ bound with loss:
\be
\label{estlossyfield}
\left\|\frac{\nabla \phi_\e}{1+|x|}\right\|^2_{L^{\infty}}\lesssim |\log b|^C\left(\|\e_2\|_{L^2_Q}^2+b^{10}\right).
\ee
{\it (v)} Weighted bound with loss:
\be
\label{neoneoneo}
\int\frac{|\e|}{1+r}\lesssim C(M)\sqrt{|\log b|}\|\e_2\|_{L^2_Q}+b^{10}.
\ee
\end{proposition}

\begin{proof}[Proof of Proposition \ref{bootbound}]
{\it Proof of (i)}: The estimate \fref{coercbase} follows directly from \fref{coerclwotht}, \fref{contorlcoerc}, our choice of orthogonality conditions \fref{orthoe} and \fref{estwieght} with $p=2$. We then estimate from the definitions \fref{defal}, \fref{denkofnekoneovnk} and \fref{coercbase}:
\bee
\int (1+r^2)|\e_1|^2&\lesssim& \int (1+r^2)\left[|\nabla \e|^2+|\nabla \phi_Q|^2|\e|^2+Q^2|\nabla \phi_\e|^2\right]\\
& \lesssim & \|\e\|^2_{H^2_Q}+\int\frac{|\nabla \phi_\e|^2}{r^2(1+|\log r|)^2}\lesssim C(M)\|\e_2\|_{L^2_Q}^2.
\eee
{\it Proof of (ii)}:  Let $p=2(1-\eta)\in [1,2)$, then from \fref{estcahmos}, Sobolev, \fref{coercbase} and the bootstrap bound \fref{bootsmalllone}:
$$\|\nabla \phi_\e\|_{L^{\infty}}\lesssim C_p(M)\|\e_2\|_{L^2_Q}^{\frac p2}\|\e\|^{1-\frac p2}_{L^1}\lesssim C_\eta\|\e_2\|^{1-\eta}_{L^2_Q}.$$
The decay bound \fref{linftyboundbis} follows from the interpolation bounds \fref{decaylinfty}, \fref{deactfiel}, the $H^2_Q$ bound \fref{coercbase} and the bootstrap bounds \fref{bootsmalllone}, \fref{bootsmallh2q}.\\
{\it Proof of (iii)}: The lossy bound \fref{interpolationfield} follows from \fref{coercbase}, \fref{linftyboundbis} with $\alpha=\frac 12$. Indeed, let $B=b^{-100}$, then: 
\bee
&&\int(1+\log (1+r))^C\frac{|\nabla \phi_\e|^2}{r^2(1+|\log r|)^2}\\
& \lesssim &  \int_{r\leq B}(1+\log (1+r))^C\frac{|\nabla \phi_\e|^2}{r^2(1+|\log r|)^2}+\int_{r\geq B}(1+\log (1+r))^C\frac{|\nabla \phi_\e|^2}{r^2(1+|\log r|)^2}\\
& \lesssim & |\log b|^{C_1(C)}\int \frac{|\nabla \phi_\e|^2}{r^2(1+|\log r|)^2}+\int_{r\geq B}\frac{(1+|\log r|)^C}{r^{2+\frac12}}\\
& \lesssim &|\log b|^{C_1(C)}\left[\|\e_2\|_{L^2_Q}^2+\frac{1}{\sqrt{B}}\right]\lesssim   |\log b|^{C_1(C)}\left(\|\e_2\|_{L^2_Q}^2+b^{10}\right).
\eee
{\it Proof of (iv)}: From Sobolev:
$$ \left\|\frac{\nabla \phi_\e}{1+|x|}\right\|^2_{L^{\infty}}\lesssim \left\|\frac{\nabla\phi_\e}{1+|x|}\right\|^2_{H^2}\lesssim \int \frac{|\nabla\phi_\e|^2}{1+r^2}+\|\e\|_{H^2}^2$$ and \fref{estlossyfield} follows from 
\fref{interpolationfield}.\\
{\it Proof of \fref{neoneoneo}}: We use the global $L^1$ bound \fref{bootsmalllone}, \fref{coercbase} and Cauchy Schwarz to estimate:
\bee
\int\frac{|\e|}{1+r}&\lesssim &\|\e\|_{L^2}\left(\int_{r\leq b^{-20}}\frac{1}{r^2}\right)^{\frac 12}+\int_{r\geq b^{-20}}\frac{|\e|}{1+r}\\
& \lesssim & C(M)\sqrt{|\log b|}\|\e_2\|_{L^2_Q}+b^{20}\|\e\|_{L^1}
\eee and \fref{neoneoneo} is proved.
\end{proof}

 \end{appendix}



\begin{thebibliography}{10}

\bibitem{Beckner} W. Beckner, Sharp Sobolev inequalities on the sphere and the Moser-Trudinger inequality, Ann. of Math., 2, 138 (1993), pp. 213--242.

\bibitem{BHG} Van den Berg, G.J.B.; Hulshof, J.; King, J., Formal asymptotics of bubbling in the harmonic map heat flow, SIAM J. Appl. Math. vol 63, o5. pp 1682--1717.

\bibitem{BT} Bejenaru, I.; Tataru, D., Near soliton evolution for equivariant Schroedinger Maps in two spatial dimensions, arXiv:1009.1608 (2011).

\bibitem{BCC} Blanchet, A.; Carillo, J; Carlen, E.; Functional inequalities, thick tails and asymptotics for the critical mass Patlak-Keller-Segel model,  arXiv:1009.0134 (2011)

\bibitem{BCM}  Blanchet, A.; Carillo, J; Masmoudi, N.;,Infinite Time Aggregation for the Critical Patlak-
Keller-Segel model in $\RR^2$, Comm. Pure Appl. Math., 61 (2008), pp. 144--1481.

\bibitem{BDP} Blanchet, A.; Dolbeault, J.;  Perthame, B., Two-dimensional Keller-Segel model: optimal critical mass and qualitative properties of the solutions, Electron. J. Differential Equations, (2006), No. 44, 32 pp.

\bibitem{Brezis} Brezis, H., Analyse fonctionnelle, Th\'eorie et applications, Collection Math\'ematiques Appliqu\'ees pour la Ma\^itrise, Masson, Paris, 1983.

\bibitem{CF} Carlen, E.; Figalli, A., Stability for a GNS inequality and the Log-HLS inequality, with application to the critical mass Keller-Segel equation, arXiv:1107.5976.

\bibitem{CL} Carlen, E.;  Loss, M., Competing symmetries, the logarithmic HLS inequality and OnofriÕs inequality on
Sn, Geom. Funct. Anal., 2 (1992), pp. 9--104.

\bibitem{LS} Dejak, S.I.; Lushnikov, P.M.; Ovchinnikov, Yu. N.; Sigal, I.M., On Spectra of Linearized Operators for Keller-Segel Models of Chemotaxis, arXiv:1110.6393 (2011)

\bibitem{DNR} Diaz, J.; Nagai, T.; Rakotoson, J.M., Symmetrization techniques on unbounded domains: application to a chemotaxis system on $\RR^N$, J.. Diff. Eq. 145, 156-183 (1998)

\bibitem{NT2} Gustafson, S.; Nakanishi, K.; Tsai, T-P.; Asymptotic stability, concentration and oscillations in harmonic map heat flow, Landau Lifschitz and Schr\"odinger maps on $\Bbb R^2$, Comm. Math. Phys. (2010), 300, no 1, 205-242.

\bibitem{HV} Herrero, M. A.; Vel\'azquez, J.J. L., Singularity patterns in a chemotaxis model, Math. Ann. 306 (1996), no. 3, 58--623. 

\bibitem{KavSou} Kavallaris, N.; Souplet, Ph.; Grow up rate and refined asymptotics for a two dimensional Patlak-Keller-Segel model in a disk, SIAM J. Math. Aanal., vol 40, no 5, pp 1852-1881.

 \bibitem{KS} Keller, E.F.; Segel, L.A, Initiation of slime mold aggregation viewed as an instability, J. Theor. Biol., 26 (1970), pp. 399-415.
 
 \bibitem{LPSS} Landman, M. J.; Papanicolaou, G. C.; Sulem, C.; Sulem, P.-L., Rate of blowup for solutions of the nonlinear Schr\"odinger equation at critical dimension. Phys. Rev. A (3) 38 (1988),
no. 8, 3837--3843. 

\bibitem{LMRlinearized}  Lemou, M.; M\'ehats, F.; Rapha\"el, P.,  Structure of the linearized gravitational Vlasov-Poisson system close to a polytropic ground state. SIAM J. Math. Anal. 39 (2008), no. 6, 1711--1739,

\bibitem{MMkdv} Martel, Y.; Merle, F. Instability of solitons for the critical generalized Korteweg-de Vries equation. Geom. Funct. Anal. 11 (2001), no. 1, 7--123.

\bibitem{MMfintietime} Martel, Y.; Merle, F., Blow up in finite time and dynamics of blow up solutions for the $L^2$ critical generalized KdV equation, J. Amer. Math. Soc. 15 (2002), no. 3, 61--664.

\bibitem{MMR1} Martel, Y.; Merle, F.; Rapha\"el, P.,  Blow up for the critical gKdV equation I: dynamics near the soliton, arXiv:1204.4625 (2012).

\bibitem{MMR2} Martel, Y.; Merle, F.; Rapha\"el, P., Blow up for the critical gKdV equation II: minimal mass dynamics, arXiv:1204.4624 (2012).

\bibitem{MMR3} Martel, Y.; Merle, F.; Rapha\"el, P., Blow up for the critical gKdV III: Blow up for the critical gKdV equation III: initial data with slow decay, in preparation.

 
 \bibitem{MR1} Merle, F.; Rapha\"el, P., Blow up dynamic and upper bound on the blow up rate for critical nonlinear Schr\"odinger equation, Ann. Math. 161 (2005), no. 1, 157--222.

\bibitem{MR2} Merle, F.; Rapha\"el, P., Sharp upper bound on the blow-up rate for the critical nonlinear Schr\"odinger equation, Geom. Funct. Anal. 13 (2003), no. 3, 59--642

\bibitem{MR3} Merle, F.; Rapha\"el, P., On universality of blow-up profile for L2 critical nonlinear Schršdinger equation, Invent. Math. 156 (2004), no. 3, 56--672.

\bibitem{MR4} Merle, F.; Rapha\"el, P., Sharp lower bound on the blow up rate for critical nonlinear Schr\"odinger equation, J. Amer. Math. Soc. 19 (2006), no. 1, 37--90.

\bibitem{MR5} Merle, F.; Rapha\"el, P.,  Profiles and quantization of the blow up mass for critical nonlinear Schr\"odinger equation, Comm. Math. Phys.  253  (2005),  no. 3, 675--704.

\bibitem{MRR} Merle, F.; Rapha\"el, P.; Rodnianski, I., Blow up dynamics for smooth equivariant solutions to the energy critical Schr\"odinger map, C. R. Math. Acad. Sci. Paris 349 (2011), no. 5-6, 279--283

\bibitem{Mizo} Mizoguchi, N., Rate of type II blowup for a semilinear heat equation, Math. Ann. 339 (2007), no. 4, 83--877.
 
 \bibitem{Nagaired} Nagai, T.; Behavior of solutions to a parabolic-elliptic system modelling chemotaxis, J. Korean Math. Soc. 37 (2000), no 5, pp 721-733.


\bibitem{Patlak} C. S. Patlak, Random walk with persistence and external bias, Bull. Math. Biophys., 15 (1953), pp. 311--338.

\bibitem{PV} Perthame, B.; Vasseur, A.;  Regularization in Keller-Segel type systems and the De Giorgi method, preprint 2012, http://www.ma.utexas.edu/users/vasseur/publications.html

\bibitem{RaphRod}  Rapha\"el, P.; Rodnianski, I., Stable blow up dynamics for the critical co-rotational Wave Maps and equivariant Yang-Mills problems, to appear in Pub. Math. IHES (2011).

\bibitem{RS} Rapha\"el, P.; Schweyer, R., Stable blow up dynamics for the 1-corotational energy critical harmonic heat flow, to appear in Comm. Pure App. Math (2011).

\bibitem{RodSter} Rodnianski, I.; Sterbenz, J., On the formation of singularities in the critical $O(3)$ $\sigma$-model, Ann. of Math. (2) 172 (2010), no. 1, 187--242.

\bibitem{Vlinear} Vel\'azquez, J. J. L., Stability of some mechanisms of chemotactic aggregation, SIAM J. Appl. Math. 62 (2002), no. 5, 158--1633.

\bibitem{VICM} Vel\'azquez, J. J. L.,  Singular solutions of partial differential equations modelling chemotactic aggregation, Proceedings of the ICM 2006, http://www.icm2006.org/proceedings/vol3.html.

\bibitem{W2}  Weinstein, M.I., Modulational stability of ground 
states of nonlinear Schr\"odinger equations, SIAM J. Math. Anal.
{\bf 16} (1985), 472---491.

\end{thebibliography}
\end{document}